%% file: DorRV20.tex
\documentclass[final]{siamart171218}

\usepackage{todonotes}

\usepackage{amsmath,amssymb,amsfonts}
\usepackage{enumerate,color,graphicx}
\usepackage{cite}
\usepackage{mymacros}
\usepackage{mathdots} 
\usepackage{mathtools}
\usepackage{pgfplots,tikz}
\usepackage{tikz-cd}
\usepackage{mathrsfs} 
\usepackage{algorithm}
\usepackage{algcompatible}
\usepackage{subfigure}
\usetikzlibrary{fit,chains,decorations.pathreplacing,matrix,calc,decorations.pathmorphing,shapes,arrows,automata,patterns,decorations.markings}

\newtheorem{remark}[theorem]{\it Remark}
\newtheorem{example}[theorem]{\it Example}

\renewcommand{\hat}{\widehat}
\renewcommand{\tilde}{\widetilde}
\newcommand{\ri}{\mathrm{i}}

\DeclareMathAlphabet{\mathpzc}{OT1}{pzc}{m}{it}

\def\G{\mbox{\boldmath$G$}}

\def\G{\mbox{\boldmath$G$}}

\begin{document}

\title{Balanced truncation model reduction for symmetric second order systems -- A passivity-based approach\thanks{This work has been supported by DFG priority program 1897: ``Calm, Smooth and Smart -- Novel Approaches for Influencing Vibrations by Means of Deliberately Introduced Dissipation''.}}

\author{Ines~Dorschky\thanks{Corresponding author. Fachbereich Mathematik, Universit\"at Hamburg, Bundesstr. 55, 20146 Hamburg, Germany. E-Mail: \textbf{ines.dorschky@uni-hamburg.de}.} \and Timo~Reis\thanks{Fachbereich Mathematik, Universit\"at Hamburg, Bundesstr. 55, 20146 Hamburg, Germany. E-Mail: \textbf{timo.reis@uni-hamburg.de}.} \and Matthias Voigt\thanks{Fachbereich Mathematik, Universit\"at Hamburg, Bundesstr. 55, 20146 Hamburg, Germany. E-Mail: \textbf{matthias.voigt@uni-hamburg.de} and Institut f\"ur Mathematik, Technische Universit\"at Berlin, Str. des 17.~Juni 136, 10623 Berlin, Germany. E-Mail:  \textbf{mvoigt@math.tu-berlin.de}}.}

\maketitle	

\begin{abstract}
We introduce a model reduction approach for linear time-invariant second order systems based on positive real balanced truncation. Our method guarantees asymptotic stability and passivity of the reduced order model as well as the positive definiteness of the mass and stiffness matrices. Moreover, we receive an a~priori gap metric error bound. Finally, we show that our method based on positive real balanced truncation preserves the structure of overdamped second order systems.
\end{abstract}
	
\begin{section}{Introduction}

In this article, we consider linear second order systems with co-located inputs and outputs  of the form
\begin{align}\label{eqn:sosys}
M\ddot{p}(t) + D\dot{p}(t) + Kp(t)  = Bu(t),\quad
y(t) = B^{\top}\dot{p}(t)
\end{align}
with symmetric $M, D, K \in \mathbb{R}^{n \times n}$, where the {\em mass matrix} $M$ and the {\em stiffness matrix} $K$ are positive definite, while the {\em damping matrix} $D$ is positive semidefinite and $B \in \mathbb{R}^{n \times m}$.
In applications, the state space dimension $n$ typically becomes unfeasibly large for  simulation, optimization, or control. This causes a demand for a good approximation
of the system. Therefore, an upper bound in a~certain quality measure is desirable for a model order reduction method. Such a~quality measure can be some norm of the error system, such as the
$H^{\infty}$ norm for stable systems, or the so-called {\em gap
metric}~\cite{GeorSmit90a,GeorSmit90b} between original and reduced order system. The latter expresses the distance between the input-output trajectories of two systems. On the other hand, our model has three key attributes that a reduced model should inherit, one of these being
the second order structure together with the symmetries, meaning the reduced system should be of the form
\begin{align}\label{eqn:redsosys}
\begin{aligned}
\wtM\ddot{\wtp}(t) + \wtD\dot{\wtp}(t) + \wtK \wtp(t) & = \widetilde{B} u(t), &
\wty(t) & = \widetilde{B}^{\top}\dot{\wtp}(t),
\end{aligned}
\end{align}
with symmetric $\wtM, \ \wtD,\, \wtK \in \mathbb{R}^{r \times r}$, and
$\widetilde{B} \in \mathbb{R}^{r \times m}$, where $r\ll n$. Other physically meaningful
properties are asymptotic stability and passivity, which should also be preserved. Passivity describes the property that the energy in the system is solely induced by input and output and may be dissipated or conserved in the system.

Some progress has been made in~\cite{ReisStyk08, BennSaak11} and preservation
of physical properties in the reduced order model is possible for
co-located inputs and outputs.
Further second order reduction methods have been developed in~\cite{SaliLohm06,ChahlLemo06, BaiSu05, Freu05, MeyeSrin96, HartVulc10, BennSaak11} (see also \cite{ReisStyk08} for an overview). Besides these, there exists interpolatory methods, which succeed either in preserving the second order structure~\cite{Sali05} or deliver a posteriori $H^{\infty}$ error bounds~\cite{PanzWolfLohm13}. However, all the approaches mentioned lack a combination of the two.

These approaches have all in common that the ansatz for the reduced order model is $p(t)= W_r\tp(t)$ for some ``tall matrix'' $W_r\in\mathbb{R}^{n\times r}$ and an accordant multiplication of the state equation in \eqref{eqn:redsosys} from the left with some ``flat'' matrix  $V_r\in\mathbb{R}^{r\times n}$. Our approach is somewhat different from these. Namely, we first perform a reduction of the first order representation, and accordingly carry out a~transformation yielding a~second order model. This corresponds to a~reduction ansatz  $p(t)= W_{r,1}\wtp(t)+W_{r,2}\dot{\wtp}(t)$ together with a~linear combination of the state equation and its derivative such that again a~second order system is obtained. The basis for our considerations is the technique of \emph{positive real balanced truncation}, which is a~passivity-preserving method for first order systems, see~\cite{DesaPal82}.
Moreover, an a~priori error bound in the gap metric is provided \cite{GuivOpme13}.
We present a modification of positive real balanced truncation, which preserves the second order structure, symmetry of the system matrices, stability, passivity as well as positive definiteness of the mass, and stiffness matrices. Namely, we show that, under an extra condition on the so called \emph{zero sign characteristics}, the reduced order model obtained by positive real balanced truncation of a~first order realization of \eqref{eqn:sosys} can be transformed to a~second order realization of the form \eqref{eqn:redsosys}.
To this end, we use the theory of \emph{standard triples} by {\sc Gohberg}, {\sc Lancaster} and {\sc Rodman} \cite{GohbLanc82a} to derive a second order realization~\eqref{eqn:redsosys} in which the mass and stiffness matrices are positive definite, and the damping matrix has at least $r-m$ positive eigenvalues provided that $r\geq m$. Furthermore, if the original system is \emph{overdamped}, see~\eqref{overdamping condition},
our method will further produce a~reduced order model with positive definite damping matrix.

{\bf Our method in a~nutshell.} The model reduction technique presented in this article is -- in theory -- consisting of three steps:

{\it Step~1:} For $H,G\in\Rnn$ with $M=HH\T$ $K=GG\T$, we set $x(t):=\begin{bsmallmatrix}
G^{\top}H^{-\top}p(t) \\ \dot{p}(t)
\end{bsmallmatrix}$ and rewrite the second order system \eqref{eqn:sosys} as
\begin{equation} \label{eq:linsys1}
\dot{x}(t)
= \underbrace{\begin{bmatrix}
	0 & G\T H^{-\top}  \\ -H^{-1}G & -H^{-1}DH^{-\top}
	\end{bmatrix}}_{=:\cA} x(t)
+ \underbrace{\begin{bmatrix} 0 \\ B \end{bmatrix}}_{=:\cB}u(t), \quad
y(t)  = \underbrace{\begin{bmatrix} 0 & B\T
	\end{bmatrix}}_{=:\cC}x(t).
\end{equation}
The most important feature of this system is that it has an internal symmetry structure $\cA\sS_n=\sS_n \cA^\top$ and $\cC=\cB^{\top}=\cB^\top \sS_n$, where
$\sS_n:=\diag(-I_n,I_n)$. In particular, its \emph{transfer function} $\tfG(s) =\cC(sI-\cA)^{-1}\cB$ is symmetric, i.\,e., it fulfills $\tfG(s) ^{\top}=\tfG(s) $. We will make heavy use of this symmetry structure.

{\it Step~2:} We apply positive real balanced truncation \cite{GuivOpme13} to the first order system \eqref{eq:linsys1}. The internal symmetry structure of \eqref{eq:linsys1} yields that positive real balanced truncation can be done by determining only one (instead of two) solutions of the K\'alm\'an-Yakubovich-Popov inequality. We show that the positive real characteristic values, see Definition~\ref{prbip}~b), can -- in a~certain sense -- be allocated to the symmetry structure of the system \eqref{eq:linsys1}. This is the basis for our finding that the resulting first order model is -- without putting any further computational effort -- of the form
\begin{equation}\label{balanced so0}
\begin{aligned}	\dot{\tx}(t)=&
\begin{bmatrix}
	0 & 0 & 0 & 0 & 0& \cA_{16} \\
	0  & 0 & 0 &0 & \cA_{25} & \cA_{26} \\
	0 & 0 & \cA_{33} & \cA_{34} & 0& \cA_{36} \\
	0 & 0 & -\cA_{34}^{\top} & \cA_{44} & 0 & \cA_{46} \\
    0 & -\cA_{25}^{\top} & 0 & 0 & 0 & 0 \\
	-\cA_{16}^{\top} & -\cA_{26}^{\top} & -\cA_{36}^{\top} & \cA_{46}^{\top} & 0 & \cA_{66}
	\end{bmatrix}{\tx}(t)+\begin{bmatrix}
	0 \\ 0 \\ 0 \\ 0 \\ 0 \\ \cB_6
	\end{bmatrix}u(t),\\
\ty(t)=&\begin{bmatrix}
	0 & 0 & 0 & 0 & 0 & \cB_6^\top
	\end{bmatrix}\tx(t),
\end{aligned}
\end{equation}
where the block sizes from left to right and from top to bottom are $m,\ell,p,p,\ell,m$, with $r=p+m+\ell$. Note that, if $\cA_{33}$ is zero, then
it would -- by merging some of the blocks -- be of the form
\begin{equation} \label{eq:redsys1}
\begin{aligned}
\dot{\wtx}(t)
=&\begin{bmatrix}
	0 & \wtG\T  \\ -\wtG & -\wtD
	\end{bmatrix}\wtx(t)+
\begin{bmatrix} 0 \\ \wtB \end{bmatrix}u(t),\\
\wty(t)=&\begin{bmatrix} 0 & \wtB^\top \end{bmatrix}\wtx(t),
\end{aligned}
\end{equation}
which would result in a~reduced second order model \eqref{eqn:redsosys} with $\wtM=I_r$, $\wtK=\wtG \wtG^\top$ and $\wtD=\wtD\T$. This is regrettably not the case in general, whence we apply

{\it Step~3:} We apply a state space transformation to~\eqref{balanced so0} such that the matrix $\cA_{33}$ vanishes. More precisely, we first intend to find some $\hT\in\Gl_{2p}(\R)$ that preserves the symmetry structure, i.\,e., it fulfills $\hT^\top \sS_p \hT=\sS_p$, and
\begin{equation}\label{T hat}
\hT^{-1}\begin{bmatrix} \cA_{33}&\cA_{34}\\-\cA_{34}^\top&\cA_{44}\end{bmatrix}\hT=\begin{bmatrix} 0&\hat{\cA}_{34}\\-\hat{\cA}_{34}^\top&\hat{\cA}_{44}\end{bmatrix}.
\end{equation}
Then a~state space transformation with
$T=\diag(I_{m+\ell},\hT,I_{\ell+m})$
leads to a~system which is indeed of the form \eqref{eq:redsys1} and can then be rewritten as a~second order system.
To derive such a~transformation, we have to use techniques from indefinite linear algebra. In particular, the aforementioned symmetry structure of our system provides that each real (invariant) zero of the system can be assigned a~signature, see Definition~\ref{def:polezerosignchar}. If the reduced system is minimal, the zeros with positive signature are given by some $\mu^+_1\leq\ldots\leq \mu_k^+<\mu^+_{k+1}=\ldots=\mu^+_{k+m}=0$ and those with negative signature by $\mu^-_1\leq\ldots\leq \mu_k^-<0$. Our main results on the existence and construction of such transformations are the following (see Theorem~\ref{thm:eq cond so}): suppose the reduced system is minimal, then the following are equivalent:

\begin{enumerate}[a)]
\item The system~\eqref{balanced so0} can be written in second order  form~\eqref{eq:redsys1}.
\item It exists a $\hT\in\Gl_{2p}(\R)$ with $\hT^\top \sS_p \hT=\sS_p$, that fulfills~\eqref{T hat}.
\item For $i=1,\ldots,k$ it holds that $\mu_i^-<\mu_i^+$.
\end{enumerate}

Here, the implication "a)$\Rightarrow$c)" is based on \cite[Thm.~16]{LancZaba14}.
If the reduced system is not minimal or c) is not fulfilled, we add equations to the system in a fashion that the newly formed system is also stable, passive and such that the transfer function is preserved. Here, in the worst case the number of states, i.\,e., the size of the mass matrix, doubles.\\

{\bf Outline of the article.} The article is structured as follows. In Section~\ref{sec:prelim} we introduce the passivity preserving balanced truncation and some background material from systems theory. In Section~\ref{sec: prbt so} we present our positive real balanced truncation ansatz for second order systems and prove some of its main features. What is left, the derivation of the balanced form of second order systems, namely~\eqref{balanced so0}, is done in Section~\ref{sec:balancing}. Before we focus on the reconstruction of the second order structure and the proof of the necessary condition on the zeros in Section~\ref{sec:constMDK}, Section~\ref{sec:prelim indef lina} introduces some necessary concepts from indefinite linear algebra. Section~\ref{sec:overdamped} considers the case that the original system is overdamped. Last but not least, in Section~\ref{sec:algorithm} we summarize our method in a~numerical procedure and hereby also discuss the numerical treatment.\\

{\bf Notations.} The set of natural numbers including zero is denoted by $\N$. The symbols $\R[s]$ and $\R(s)$ respectively stand for the ring of real polynomials and the field of real rational functions.
We denote by $\mathbb{C}^+$ and $\overline{\mathbb{C}^+}$ the open and closed complex half-plane. Further, $\tfG(s)\in\R(s)^{m\times m}$ is called {\em proper} if $\lim_{s\rightarrow\infty}\tfG(s) <\infty$ and {\em strictly proper} if the latter limit is zero.
The transpose and conjugate transpose of $T\in\C^{m\times n}$ are denoted by $T\T$ and $T^*$, respectively.
For symmetric matrices $X,Y\in\R^{n\times n}$, we write $X>Y$ if $X-Y$ is positive definite and $X\geq Y$ if $X-Y$ is positive semidefinite.

We call $S=\diag(\varepsilon_1,\ldots,\varepsilon_n)$, where $\varepsilon_i=\pm1$ for $i=1,\ldots,n$, a \emph{signature matrix} and make frequent use of the signature matrix $\sS_n:=\diag(-I_n,I_n)\in\R^{2n\times 2n}$.

The symbol $\Gl_n(\R)$ stands for the set of invertible $n\times n$ matrices with entries in $\R$. By $\cR H^{\infty}(\mathbb{C}^{m\times m})$, we denote the space of proper elements of $\R(s)^{m\times m}$ with entries having no poles in $\overline{\mathbb{C}^+}$, where $\cR H^{2}(\mathbb{C}^{m})$ the space of strictly proper elements of $\R(s)^{m\times m}$ with entries having no poles in $\overline{\mathbb{C}^+}$. Note that  $\cR H^{\infty}(\mathbb{C}^{m\times m})$ becomes a~normed space when equipped with the norm
\[\|\tfG(s) \|_\infty:=\sup_{\omega\in\R}\|\tfG(\ri\omega)\|_2.\]
Moreover, $\cR H^{2}(\mathbb{C}^{m})$ is an inner product space provided with the norm
\[\|u(s)\|_2^2:=\int_\R\|u(\ri\omega)\|_2^2\mathrm{d}\omega.\]
Here, $\|\cdot\|_2$ on the right-hand side stands for the maximum singular value of matrices and the Euclidean norm of vectors, respectively.

\end{section}


\begin{section}{Positive real balanced truncation of first order systems}\label{sec:prelim}
We revisit positive real balanced truncation for linear time-invariant first order systems
\begin{equation}
	\dot{x}(t) = \cA x(t)+\cB u(t), \quad
	y(t)=\cC x(t)+\cD u(t),
	\label{eq:linear ODE}
\end{equation}
with $\cA\in\R^{n\times n}$, $\cB,\cC\T\in\R^{n\times m}$ and $\cD\in\Rmm$. The dynamical system~\eqref{eq:linear ODE} is said to be \emph{minimal} if it is both controllable and observable. The \emph{transfer function} is given by $\tfG(s) =\cC(sI_n-\cA)^{-1}\cB+\cD\in\R(s)^{m\times m}$. We also speak of~\eqref{eq:linear ODE} as a realization of $\tfG(s) $ and use the notation $[\cA,\cB,\cC,\cD]$ to refer to this system, or if $\cD=0$ we write $[\cA,\cB,\cC]$. We call $\mu\in\C$ an~{\em (invariant) zero of $[\cA,\cB,\cC,\cD]$}, if
\begin{equation*}
\rank_\C\begin{bmatrix}-\mu I_n+\cA&\cB\\\cC&\cD\end{bmatrix}<\rank_{\R(s)}\begin{bmatrix}-s I_n+\cA&\cB\\\cC&\cD\end{bmatrix}.
\end{equation*}
In other words, the set of zeros of $[\cA,\cB,\cC,\cD]$ equals to the set of eigenvalues of the matrix pencil $\left[\begin{smallmatrix}-s I_n+\cA&\cB\\\cC&\cD\end{smallmatrix}\right]$. 
If the latter pencil is square and invertible as a~matrix over $\R(s)$ (which is, by taking the Schur complement, equivalent to the transfer function being invertible), then $\mu\in\C$ is a~zero of $[\cA,\cB,\cC,\cD]$, if and only if there exists some $v\in\C^{n+m}\setminus\{0\}$ with $\left[\begin{smallmatrix}-\mu I_n+\cA&\cB\\\cC&\cD\end{smallmatrix}\right]v=0$.

In the following we deal with systems having {\em positive real} transfer functions.
That is,
\begin{enumerate}[a)]
	\item $\tfG(s) $ has no poles in $\C^+$.
	\item $\tfG(\lambda)+\tfG(\lambda)^{*}\geq 0$ for all $\lambda\in\C^+$.
\end{enumerate}
If the inequality in (ii) is strict, $\tfG(s) $ is called \emph{strictly positive real}.
If $\tfG(s) $ is positive real and invertible, also $\tfG^{-1}(s)$ is positive real. As a consequence, realizations $[\cA,\cB,\cC,\cD]$ of positive real functions with the property that all eigenvalues of $\cA$ have nonpositive real part do not have any zeros in $\C^+$. In particular, minimal realizations of positive real transfer functions do not have any zeros in $\C^+$. The famous positive real lemma draws a link between positive realness of the transfer function and the solvability of a certain linear matrix inequality.

\begin{theorem}[Positive real lemma {\cite[Chap.~V]{AndeVong73}}]\label{PRL}
	Let $[\cA,\cB,\cC,\cD]$ be a~minimal system of the form~\eqref{eq:linear ODE}, with transfer function $\tfG(s) \in\R(s)^{m\times m}$. Then $\tfG(s) $ is positive real, if and only if there exists some $P>0$, such that the {\em K\'alm\'an-Yakubovich-Popov inequality (KYP)}
	\begin{equation}\label{KYP}
		\sW_{[\cA,\cB,\cC,\cD]}(P):=
		\begin{bmatrix}
		\cA^{\top}P+P \cA & P \cB-\cC^{\top} \\
		\cB^{\top}P-\cC & -\cD-\cD\T
		\end{bmatrix}\leq 0
	\end{equation}
	is fulfilled. Moreover, there exists a \emph{minimal} solution $P_{\min}> 0$ and a~\emph{maximal} solution $P_{\max}>0$ of~$\sW_{[\cA,\cB,\cC,\cD]}(P)\leq0$, i.\,e., for all other solutions $P$ of~\eqref{KYP} it holds that $P_{\max}\geq P\geq P_{\min}$.
\end{theorem}

The KYP inequality admits the so-called {\em dissipation inequality}. That is, for all locally square integrable solutions $(x,u,y)$ of $[\cA,\cB,\cC,\cD]$ and $t>0$ it holds that
\begin{equation}	
	x(t)^\top P x(t)\leq x(0)^\top P x(0)+\int_0^{t} y(\tau)^\top u(\tau)\dtau.\label{eq:dissineq}
\end{equation}

Such systems are also called {\em passive}.
Since $\tfG(s) ^\top$ is the transfer function of $[\cA^\top,\cC^\top,\cB^\top,\cD\T]$, this system is passive as well. Therefore, the dual KYP inequality $\sW_{[\cA\T,\cC\T,\cB\T,\cD\T]}(Q)\leq0$ has again a~minimal solution. Moreover, $P>0$ solves $\sW_{[\cA,\cB,\cC,\cD]}(P)\leq0$, if and only if $P^{-1}$ is a~solution of $\sW_{[\cA\T,\cC\T,\cB\T,\cD\T]}(Q)\leq0$. As a~consequence, if $P_{\min}$ is the minimal such solution in the sense of Theorem~\ref{PRL}, then $P_{\min}^{-1}$ is the maximal solution of the dual KYP inequality.
\begin{definition}[positive real balanced, internally passive]\label{prbip}
	With the notation of Theorem~\ref{PRL}, a~system $[\cA,\cB,\cC,\cD]$ is called
	\begin{enumerate}[a\emph{)}]
	  \item \emph{positive real balanced}, if for the minimal solutions $P_{\min}$, $Q_{\min}$  of the KYP inequalities $\sW_{[\cA,\cB,\cC,\cD]}(P)\leq0$ and $\sW_{[\cA^\top,\cC^\top,\cB^\top,\cD\T]}(Q)\leq0$ we have
	\[P_{\min}=Q_{\min}=\diag(\sigma_1 I_{n_1},\ldots,\sigma_{h} I_{n_h}),\]
	where $\sigma_1,\ldots,\sigma_{h}$ are distinct values in $(0,1]$. The latter are called the \emph{positive real characteristic values} of $[\cA,\cB,\cC,\cD]$.\\
	  \item  \emph{internally passive}, if $\sW_{[\cA,\cB,\cC,\cD]}(I_n)\leq0$.
	\end{enumerate}
\end{definition}

In contrast to the conventional definition of positive real balanced, we do not assume that the positive real characteristic values of $[\cA,\cB,\cC,\cD]$ are ordered.

Note that internal passivity implies that in the dissipation inequality \eqref{eq:dissineq}, the quadratic form with $P$ is the square of the norm of the state. It has been shown in Theorem 7 of~\cite{ReisWill11} that any~positive real balanced realization is internally passive.

If there exist minimal solutions $P_{\min}\geq0$, $Q_{\min}\geq0$ of the KYP inequalities\linebreak $\sW_{[\cA,\cB,\cC,\cD]}(P)\leq0$ and $\sW_{[\cA^\top,\cC^\top,\cB^\top,\cD\T]}(Q)\leq0$, then a~certain state space transformation leads to a~positive real balanced realization. We refer to this as {\em positive real balancing}.
Our main emphasis is on \emph{positive real balanced truncation}, that is balancing the system and accordingly removing the blocks corresponding to some positive real characteristic values. This leads to a~reduced system which is passive and if the transfer function of the original system is strictly positive real, is again  asymptotically stable and positive real balanced, see Theorem 4 or Lemma 2 and Theorem 1 in~\cite{HarsJonc83}.
Both steps can be done at once while also removing the uncontrollable and unobservable parts \cite{ToPo87}. Namely, by using factorizations $P_{\min}=L\T L, \ Q_{\min}=R\T R$ and the singular value decomposition
\begin{equation}\label{SVD LRT}
	LR\T=\begin{bmatrix}
	U_1 & U_2
	\end{bmatrix}\begin{bmatrix}
	\Sigma_1 & 0 \\ 0 & \Sigma_2
	\end{bmatrix}\begin{bmatrix}
	Z_1 \\ Z_2
	\end{bmatrix},
\end{equation}
we are able to define the reduction matrices $W\T =\Sigma_1^{-1/2}Z_1 L$ and $V=R\T U_1\Sigma_1^{-1/2}$. A reduced model received from positive real balanced truncation is
$[W\T \cA V,W\T \cB,\cC V,\cD]$. If we do not truncate any singular values, then we obtain a~minimal system as the following lemma shows.

\begin{lemma}\label{lem:prb minimal}
	Let a~system $[\cA,\cB,\cC,\cD]$ with positive real transfer function $\tfG(s) $ be given. Assume that $L$ and $R$ are matrices with full row rank and $P_{\min}=L\T L, \ Q_{\min}=R\T R$, where $P_{\min}$, $Q_{\min}$ are minimal solutions of the KYP inequalities $\sW_{[\cA,\cB,\cC,\cD]}(P)\leq0$ and $\sW_{[\cA^\top,\cC^\top,\cB^\top,\cD\T]}(Q)\leq0$. Further, let $LR\T=U\Sigma Z$ be a~singular value decomposition and let $W\T =\Sigma^{-1/2}Z L$, $V=R\T U\Sigma^{-1/2}$. Then
	$[W\T \cA V,W\T \cB,\cC V,\cD]$ is a~minimal and positive real balanced realization of $\tfG(s) $. \\ 
	Moreover, for any further minimal positive real balanced realization	$[\hat{\cA},\hat{\cB},\hat{\cC},D]$ of  $\tfG(s) $ with positive real characteristic values in the same order, there exist orthogonal matrices $Q_{i}\in\R^{n_i\times n_i}$ for $i=1,\ldots,h$ such that for $T:=\diag(Q_{1},Q_2,\ldots,Q_{h})$, $\hat{\cA}=T^{-1}\cA T$, $\hat{B}=T^{-1}\cB$ and $\hat{\cC}=\cC T$.
\end{lemma}
\begin{proof}
	By using an appropriate state space transformation, we can assume that $P_{\min}=:\diag(P_1,0)$ for some positive definite matrix $P_1$, and partition $[\cA,\cB,\cC,\cD]$ accordingly, i.\,e., $\cA=\begin{bsmallmatrix}
		\cA_{11} & \cA_{12} \\ \cA_{21} & \cA_{22}
		\end{bsmallmatrix}, \ B=\begin{bsmallmatrix}
		\cB_1 \\ \cB_2
		\end{bsmallmatrix}$, and $\cC=\begin{bsmallmatrix}\cC_1 & \cC_2\end{bsmallmatrix}$. Since $\sW_{[\cA,\cB,\cC,\cD]}\leq 0$, the block form yields $\cA_{12}=0$ and $\cC_2=0$, i.\,e., we obtain a~K\'alm\'an observability decomposition in which the kernel of $P$ corresponds to the unobservable states. On the other hand, if $[\cA,\cB,\cC,\cD]$ is in K\'alm\'an observability decomposition, and $P_{\min}=\begin{bsmallmatrix}
		P_{11} & P_{12} \\ P_{12}\T & P_{22}
		\end{bsmallmatrix}$ according to the block structure of the K\'alm\'an controllability decomposition, then a~simple calculation shows that $P=\begin{bsmallmatrix}
		P_{11} & 0 \\ 0 & 0
		\end{bsmallmatrix}$ with $\sW_{[\cA,\cB,\cC,\cD]}(P)\leq0$. The minimality of $P_{\min}$ therefore leads to $P_{12}=0$ and $P_{22}=0$. As a~consequence, the kernel of $P$ indeed corresponds to the space of unobservable states. Likewise, the image of $Q_{\min}$ corresponds to the space of controllable states, whence, by using the results from~\cite{ToPo87}, $[W\T \cA Z,W\T \cB,\cC Z,\cD]$ is a~minimal positive real balanced realization of $\tfG(s) $.
		
The second statement is Lemma~6 from~\cite{ReisWill11}.
\end{proof}

A consequence of this lemma is that the reduced transfer function does not depend on the specific minimal positive real balanced realization of the original transfer function.

Now we present details on the error bound of positive real balanced truncation.
A~{\em (right) coprime factorization} of $\tfG(s) \in\R(s)^{p\times m}$ is $\begin{bsmallmatrix}\tfM(s) \\ \tfN(s)\end{bsmallmatrix}$
consisting of $\tfN(s)\in\cR H^{\infty}(\Cpm)$, $\tfM(s)\in\cR H^{\infty}(\Cmm)$ such that $\tfG(s) =\tfN(s)\tfM(s)^{-1}$, and if there exist $\tfX(s)\in\cR H^{\infty}(\Cmm)$ and $\tfY(s)\in\cR H^{\infty}(\C^{m\times p})$
that satisfy the {\em B\'ezout identity}
\[\tfX(s)\tfM(s)+\tfY(s)\tfN(s)=I_m.\]
A~coprime factorization $\begin{bsmallmatrix}\tfM(s) \\ \tfN(s)\end{bsmallmatrix}$ is called {\em normalized} if additionally,
\[\tfM\T(-s)\tfM(s)+\tfN\T(-s)\tfN(s)=I_m.\]
Such factorizations can be computed using techniques as in~\cite{MeyeFran87,Vidy88}. Considering normalized coprime factorizations, a~distance measure for general transfer functions can be introduced.

\begin{definition}\label{gap metric}
	Let $\tfG_1(s),\tfG_2(s)\in\R(s)^{p\times m}$ be given with normalized coprime factorizations
	$\begin{bsmallmatrix}
	\tfM_1(s) \\ \tfN_1(s)
	\end{bsmallmatrix}$ and $\begin{bsmallmatrix}
	\tfM_2(s) \\ \tfN_2(s)
	\end{bsmallmatrix}$, respectively, where $\tfM^{-1}(s)$ is proper. Let  $\Pi_1,\Pi_2:\cR H^2(\C^{m+p})\To \cR H^2(\C^{m+p})$ be orthogonal projectors with
	\[
	\im\Pi_1=\begin{bmatrix}\tfM_1(s) \\ \tfN_1(s)\end{bmatrix}\cdot H^2(\C^{m}), \qquad
	\im\Pi_2=\begin{bmatrix}\tfM_2(s) \\ \tfN_2(s)\end{bmatrix}\cdot H^2(\C^{m}).
	\]
	Then the {\em gap between $\tfG_1(s)$ and $\tfG_2(s)$} is defined via
	\[\delta_g(\tfG_1(s),\tfG_2(s)):=\|\Pi_1-\Pi_2\|_{L(\cR H^2(\C^{m+p}))},\]
	where $\|\cdot\|_{L(\cR H^2(\C^{m+p}))}$ denotes the operator norm on $\cR H^2(\C^{m+p})$.
\end{definition}
It is shown in~\cite{Vidy84} that $\delta_g(\cdot,\cdot)$ fulfills the axioms of a~metric. The gap metric between two systems is simply the gap metric between their transfer functions. The properness of $\tfM^{-1}(s)$ ensures that the gap metric expresses a~measure for the distance between the input-output trajectories of two systems. Further note that the gap metric is also applicable to unstable systems.

Positive real balanced truncation has the following gap metric error bound.
\begin{theorem}\cite[Cor.~2.2]{GuivOpme13}\label{gap metric error}
	Let $[\cA,\cB,\cC,\cD]$ be a~realization of the positive real function $\tfG(s) \in\R(s)^{m\times m}$. Denote the positive real characteristic values by $(\sigma_i)_{i=1}^h$ and, for $r<h$, let $[\widetilde{\cA},\widetilde{\cB},\widetilde{\cC},\cD]$ be obtained by positive real balanced truncation of $[\cA,\cB,\cC,\cD]$ by removing the blocks corresponding to $\sigma_{r+1},\ldots,\sigma_h$. Then the transfer function $\widetilde{\tfG}(s)$ of $[\widetilde{\cA},\widetilde{\cB},\widetilde{\cC},\cD]$ fulfills
	\begin{equation*}
		\delta_g(\tfG(s) ,\widetilde{\tfG}(s))\leq 2 \sum_{i=r+1}^h \sigma_i.
	\end{equation*}	
\end{theorem}

Note that \cite{GuivOpme13} considers positive real balanced truncation in which the states corresponding to the smallest characteristic values are removed. A~careful inspection of the proof of the error bound (and those of the therein used results) yields that the above error bound still holds when states corresponding to arbitrary positive real characteristic values are removed. In our method for second order systems we will make use of this fact.

\end{section}


\begin{section}{Positive real balanced truncation for second order systems}\label{sec: prbt so}
We introduce positive real balanced truncation for systems having a symmetric second order structure. Now we consider linear time-invariant first order systems $[\cA,\cB,\cC]$ 
with $\cA\in\R^{2n\times 2n}$ and $\cB\in\R^{2n\times m}$ structured as in \eqref{eq:linsys1}, that is
\begin{equation}
	\cA=\begin{bmatrix}
	0 & G\T  \\ -G & -D
	\end{bmatrix}\in\Gl_{2n}(\R),\quad \cB=\begin{bmatrix} 0 \\ B \end{bmatrix}=\cC\T
	\label{eq:ABstruc2}
\end{equation}
for some $D\in\R^{n\times n}$ with $D=D^\top\geq 0$ and $B\in\R^{n\times m}$. The transfer function is given by
\[\tfG(s) =\cC(sI_{2n}-\cA)^{-1}\cB=s B^\top(s^2I_n+sD+K)^{-1}B\in\R(s)^{m\times m},\]
where $K=GG\T$.
We first notice the following:
\begin{remark}[Second order systems, passivity, and zeros]\label{rem: sosap}
	\begin{enumerate}[a)]
		\item We assume throughout the remaining sections that $\rank \cB=m$. This is no restriction, since otherwise, there exists an orthogonal matrix $T\in\Rmm$ such that $\cB T=\begin{bmatrix}
		\cB_1 & 0
		\end{bmatrix}$, where $\rank \cB_1=m$. Hence $T\T \tfG(s) T=\begin{bsmallmatrix}
		\tfG_1(s) & 0 \\ 0 & 0
		\end{bsmallmatrix}$ for some $\tfG_1(s)\in\R(s)^{k\times k}$ with $k=\rank \cB_1$. In this case one can approximate $\tfG_1(s)$ instead and afterwards add the zero rows and columns to the reduced transfer function.
		\item It can be seen that $\sW_{[\cA,\cB,\cB\T,0]}(I_{2n})\leq 0$ and hence, $\tfG(s) $ is positive real by the positive real lemma (see Theorem~\ref{PRL}).
		Then \cite[Thm.~15]{Reis11} guarantees the existence of respective minimal solutions $P_{\min},Q_{\min}\geq 0$ of $\sW_{[\cA,\cB,\cB\T,0]}(P)\leq 0$ and $\sW_{[\cA\T,\cB,\cB\T,0]}(Q)\leq 0$, if $(\cA,\cB)$ and $(\cA\T,\cC\T)$ are stabilizable. By using the symmetry structure of the system~\eqref{eq:ABstruc2}, i.\,e., $\cA\T=\sS_n \cA\sS_n$ and $\cB\T=\cC\sS_n$ for $\sS_n=\diag(-I_n,I_n)$, the latter two properties are however equivalent due to $(\cA\T,\cC\T)=(\sS_n \cA\sS_n,\sS_n \cB)$. Since further, $\cA$ does not have any eigenvalues in $\mathbb{C}^+$, the absence of uncontrollable purely imaginary eigenvalues is sufficient for the existence of minimal solutions $P_{\min},Q_{\min}\geq 0$.
\item The assumption that $\rank B=m$ furthermore implies that the transfer function of $[\cA,\cB,\cC]$ with matrices in \eqref{eq:ABstruc2} is strictly positive real. Consequently, the transfer function $\tfG(s)\in\R(s)^{m\times m}$ is invertible. 
    Further note that positive realness of $\tfG(s) $ together with $\cA$ having no eigenvalues in $\C^+$ implies that the system $[\cA,\cB,\cC]$ has no zeros in $\C^+$.
	\end{enumerate}
\end{remark}

The symmetry structure further implies that $P\geq 0$ solves $\sW_{[\cA,\cB,\cC,0]}(P)\leq 0$, if and only if $Q:=\sS_n P \sS_n$ solves $\sW_{[\cA\T,\cC\T,\cB\T,0]}(Q)\leq 0$. In particular, $Q_{\min}=\sS_n P_{\min} \sS_n$ and thus, for $L\T L=P_{\min}$, we obtain that $\sS_n L\T L \sS_n=Q_{\min}$. Altogether, instead of the singular value decomposition~\eqref{SVD LRT} we can compute the eigendecomposition
\begin{equation}
	L\sS_nL\T =\underbrace{\begin{bmatrix}U^-&U^+\end{bmatrix}}_{=:U}\sS_n\underbrace{\begin{bmatrix}\Sigma^-&0\\0&\Sigma^+\end{bmatrix}}_{=:\Sigma}
	\begin{bmatrix}U^-&U^+\end{bmatrix}\T, \label{symmeig2}
\end{equation}
where $U\in\Rnn$ is orthogonal and
\begin{equation}
\begin{aligned}
\Sigma^-=&\,\diag\big(\sigma^-_{1}I_{n_1^-},\ldots,\sigma^-_{h^-}I_{n_{h^-}^-}\big),\qquad&0\leq&\sigma^{-}_{h^{-}}<\ldots<\sigma^{-}_{1}\leq 1,\\ \Sigma^{+}=&\diag\big(\sigma^+_{h^+}I_{n_{h^+}^+},\ldots,\sigma^+_{1}I_{n_1^+}\big),\qquad&0\leq&\sigma^{+}_{h^{+}}<\ldots<\sigma^{+}_{1}\leq 1.
\end{aligned}
\end{equation}
Next we choose some positive real characteristic values of positive and negative type which correspond to truncated states. 
To this end, let $r^+,r^-\in\N$ be such that for some $q^+,q^-\in\N$, $r^\pm=\sum_{j=1}^{q^\pm}n_j^\pm$. 
Additionally, these numbers have to be chosen such that the states corresponding to zero characteristic values are truncated, and the set of characteristic values corresponding to the truncated states does not intersect with the set of  characteristic values corresponding to the preserved states. This  means that
\begin{equation}\label{eq:rpm 1}
\begin{split}
      1&\leq q^{\pm} \leq h^{\pm}, \qquad \sigma_{q^{\pm}} > 0, \\
      \sigma^{-}_{q^{-}+j^{-}}&\neq \sigma^{+}_{i^{+}} \ \text{ and } \ \sigma^{+}_{q^{+}+j^{+}}\neq \sigma^{-}_{i^{-}} \ \text{ for all } \ i^{\pm}=1,\ldots,q^{\pm}, \ j^{\pm}=1,\ldots,h^{\pm}-q^{\pm}.     
\end{split}
\end{equation}
Note that the above condition is only of pathological nature and does not impose any serious restriction from a~numerical point of view, since generically, it holds that the characteristic values in $(0,1)$ are simple. 

The general purpose is that the reduced system can be transformed into a second order system. To this end we require that the reduced system has a~symmetry structure with respect to a~matrix $\sS_r$. 
Hence it is essential that we find $r^{\pm}$ which additionally fulfill $r^-=r^+$.
We partition~\eqref{symmeig2} as
\begin{equation}\label{Sigma pm12}
	L\sS_nL\T =\begin{bmatrix}U^-_{1}&U_{2}&U^+_{1}\end{bmatrix}\begin{bmatrix}-\Sigma^-_{1}&0&0\\0& S\Sigma_{2}&0\\0&0&\Sigma^+_{1}\end{bmatrix}
	\begin{bmatrix}(U^-_{1})\T\\(U_{2})\T\\(U^+_{1})\T\end{bmatrix},
\end{equation}
where $U_1^{\pm}\in\R^{2n\times r^{\pm}}, \ U_2\in\R^{2n\times (2n-r^--r^+)}, \ \Sigma_2\in\R^{(2n-r^--r^+)\times (2n-r^--r^+)}$, $\Sigma_1^{\pm}\in\R^{r^{\pm}\times r^{\pm}}$, and $S=\diag(-I_{n-r^-},I_{n-r^+})$.
We set
$
\Sigma_1:=\diag(\Sigma_1^-, \Sigma_1^+)$ and $U_1:=\begin{bmatrix}U^-_{1} & U^+_{1}\end{bmatrix}
$. Using the reduction matrices
\begin{equation}\label{proj WV}
	W\T:=\Sigma_1^{-\frac12}\sS_rU_1\T L \quad \text{and} \quad V:=\sS_nL\T U_1\Sigma_1^{-\frac12},
\end{equation}
we construct the reduced reduced first order model 
\begin{equation}
[\widetilde{\cA},\widetilde{\cB},\widetilde{\cC}]:=[W\T \cA V,W\T \cB,\cC V].
	\label{eq:redfo}
\end{equation}
Next, we state some important properties of the above reduced order model.
\begin{theorem}\label{thm: prop redsys}
	Let a~stabilizable system $[\cA,\cB,\cC]$ be given
	with $\cA\in\R^{2n\times 2n}$ and $\cB\in\R^{2n\times m}$ as in \eqref{eq:ABstruc2} with $G\in\Gl_n(\R)$, $D\in\R^{n\times n}$ with $D=D^\top\geq 0$, and $B\in\R^{n\times m}$. Consider the reduced system $[\widetilde{\cA},\widetilde{\cB},\widetilde{\cC}]$ as constructed as in \eqref{Sigma pm12}--\eqref{eq:redfo} for some $r^{+},r^-\in\N$ which fulfill~\eqref{symmeig2}--\eqref{eq:rpm 1}. Then the following statements are satisfied:
	\begin{enumerate}[a)]
		\item We have $\sW_{[\widetilde{\cA},\widetilde{\cB},\widetilde{\cC},0]}(\Sigma_1)\leq0$ and $\sW_{[\widetilde{\cA}\T,\widetilde{\cC}\T,\widetilde{\cB}\T,0]}(\Sigma_1)\leq0$. In particular, $[\widetilde{\cA},\widetilde{\cB},\widetilde{\cC}]$ is passive.
		\item We have $\diag(-I_{r^-},I_{r^+})\widetilde{\cA} \diag(-I_{r^-},I_{r^+})=\widetilde{\cA}\T$ and $\widetilde{\cB}=\diag(-I_{r^-},I_{r^+})\widetilde{\cB}=\widetilde{\cC}\T$.
		\item The gap metric between the transfer functions $\tfG(s) $ and $\widetilde{\tfG}(s)$ of 
		$[{\cA},{\cB},{\cC}]$
and 		$[\widetilde{\cA},\widetilde{\cB},\widetilde{\cC}]$
		can be estimated by
		\begin{equation*}
		\delta_g(\tfG(s) ,\widetilde{\tfG}(s))\leq 2\sum_{i=q^-+1}^{h^-}\sigma_i^- +2\sum_{i=q^++1}^{h^+}\sigma_i^+.
		\end{equation*}
		\item If $\sigma^-_{q^-+1}=0=\sigma^+_{q^++1}$, then $[\widetilde{\cA},\widetilde{\cB},\widetilde{\cC}]$ is a~minimal positive real balanced realization of $\tfG(s)$.
	\end{enumerate}
\end{theorem}
\begin{proof}
	The first two statements can be inferred from the arguments in the proof of \cite[Thm.~8]{ReisWill11}, whereas
	the third part follows from Theorem~\ref{gap metric error} and the last statement from Lemma~\ref{lem:prb minimal}.
\end{proof}

The exact block structure of the reduced system, as introduced in~\eqref{balanced so0}, will be part of Section~\ref{sec:balancing}. First we need some results from the study of indefinite linear algebra.

\end{section}


\begin{section}{Preliminaries from indefinite linear algebra}\label{sec:prelim indef lina}

Next, we introduce some notions and results from the study of indefinite linear algebra.
\begin{definition}\label{def:consim}
	Two pairs $(S_j,A_j)\in\Rnn\times\Rnn$, $j=1,2$, consisting of a~symmetric matrix $S_j\in\Gl_n(\R)$ and an {\em $S_j$-self-adjoint} matrix $A_j$, i.\,e., $A_j\T S_j=S_jA_j$, are called \emph{congruent-similar}, if there exists a $T\in\Gl_n(\R)$ such that $T^{-1}A_1T=A_2$ and $T\T S_1 T=S_2$. 
\end{definition}
In \cite{GohbLanc83}, a canonical form under congruence-similarity is given. For the sake if simplicity we will focus on diagonalizable matrices.

\begin{theorem}\label{can unsim}\cite[Sec.~I.5, Thm.~5.3]{GohbLanc83}
	Let $(S,A)\in\Rnn\times\Rnn$, where $S\in\Gl_n(\R)$ is symmetric, and $A$ is $S$-self-adjoint and diagonalizable over $\C$. Then there exists some $T\in\Gl_n(\R)$, such that for some $k,c\in\N$ with $2c+k=n$ and $\varepsilon_1,\ldots,\varepsilon_k\in\{\pm 1\}$, $\lambda_1,\ldots,\lambda_k\in\R$,  $\sigma_1,\ldots,\sigma_c\in\R$, $\tau_1,\ldots,\tau_c>0$, and
	\begin{equation*}
		\begin{split}
		&  \sJ_1:=\begin{bmatrix}
		0 & 1 \\ 1 & 0
		\end{bmatrix}, \ \sJ_c:=\diag(\underbrace{\sJ_1,\ldots,\sJ_1}_{c  \text{ times}}), \
		\sP_{\sigma_i,\tau_i}:=\begin{bmatrix}
		\sigma_i & \tau_i \\ -\tau_i & \sigma_i
		\end{bmatrix}, \\
		\end{split}
	\end{equation*}
	we obtain
	\begin{equation}\label{can unsim eq}
		\begin{split}
	 	T\T ST=&\diag(\varepsilon_1,\ldots,\varepsilon_k,\sJ_c), \\ T^{-1}AT=&\diag(\lambda_1,\ldots,\lambda_k,\sP_{\sigma_1,\tau_1},\ldots,\sP_{\sigma_{c},\tau_{c}}).
		\end{split}
	\end{equation}
\end{theorem}
It has been further shown in \cite[Sec.~I.5, Thm.~5.3]{GohbLanc83} that the above is a canonical form for the pair $(S,A)$ under congruence-similarity, if the tuples $(\varepsilon_1,\lambda_1),\ldots,(\varepsilon_k,\lambda_k)$ and  $(\sigma_1,\tau_1),\ldots,(\sigma_c,\tau_c)$ are ordered increasingly with respect to the lexicographical order.
It can be seen that the eigenvalues of $A$ in Theorem~\ref{can unsim} are given by $\lambda_1,\ldots,\lambda_k$ and $\sigma_1\pm \iunit \tau_1,\ldots,\sigma_1\pm \iunit\tau_1$.
Based on this normal form we derive a~special form for $S$-self-adjoint and diagonalizable matrices whose eigenvalues are in the closed complex left half plane.

\begin{corollary}\label{right symmetry form}
	Let $(S,A)\in\R^{n\times n}\times\R^{n\times n}$ be as in Theorem~\ref{can unsim}. If all eigenvalues of $A$ have negative real part, then there exist $T\in\Gl_n(\R)$ and $c,k,k_1,k_2\in\N$ with $k_1+k_2=k$ and $2c+k=n$, such that
	\begin{equation}\label{eq:right symmetry form}
		T^{-1}AT=\begin{bmatrix}
		0 & 0 & \cV \\
		0 & \Lambda  & 0 \\
		-\cV & 0 & \cE \\
		\end{bmatrix}, \quad T\T S T=\diag(-I_{c+k_1},I_{k_2+c}),
	\end{equation}
	where $\cE,\cV\in\R^{c\times c}$ and $\Lambda\in\R^{k\times k}$ are negative definite and diagonal.
\end{corollary}
\begin{proof}
	Without loss of generality we can assume that $(S,A)$ is in the canonical form of Theorem~\ref{can unsim}. The assumption on the spectrum of $A$ implies that $\sigma_i> 0$ for $i=1,\ldots,c$ and $\ell=0$. 
	Since further, $\tau_i>0$ for all $i=1,\ldots,c$, the matrix
	\begin{equation}\label{Pi}
		\Theta_i:=\frac{1}{\sqrt{2\tau_i}}\begin{bsmallmatrix}
		\sqrt{-\sigma_i+\sqrt{\sigma_i^2+\tau_i^2}} & -\sqrt{-\sigma_i+\sqrt{\sigma_i^2+\tau_i^2}} \\ \sqrt{\tau_i}\left(-2\sqrt{-\sigma_i+\sqrt{\sigma_i^2+\tau_i^2}}\right)^{-1} & \sqrt{\tau_i}\left(-2\sqrt{-\sigma_i+\sqrt{\sigma_i^2+\tau_i^2}}\right)^{-1}
		\end{bsmallmatrix}
	\end{equation}
	is real and straightforward computations show that
	\begin{equation*}
		\begin{split}
			\Theta_i^{\top}\sJ_1 \Theta_i=&\sS_1 \quad \text{and} \quad \Theta_i^{-1}\cP_{\sigma_i,\tau_i}\Theta_i=\sS_1\Theta_i\T\sS_1\cP_{\sigma_i,\tau_i}\Theta_i=\begin{bmatrix}
			0 & \nu_i \\
			-\nu_i & \eta_i
			\end{bmatrix},
		\end{split}
	\end{equation*}
	where $\eta_i=\tfrac{-3\sigma_i^2+4\sigma_i\sqrt{\sigma_i^2+\tau_i^2}}{-2\sigma_i+2\sqrt{\sigma_i^2+\tau_i^2}}< 0$. 
	Setting $T:=\diag(I_{2k},\Theta_1,\ldots,\Theta_{c+ \ell})$ and suitably interchanging the rows and columns of the tuple $(T\T S T,T^{-1}AT)$ leads to the form~\eqref{eq:right symmetry form}.
\end{proof}

\begin{definition}
	Let $(S,A)\in\R^{n\times n}\times\R^{n\times n}$ be as in Theorem~\ref{can unsim}. Regarding the canonical form in Theorem~\ref{can unsim}, we call the tuple $\left((\varepsilon_1,\lambda_1),\ldots,(\varepsilon_k,\lambda_k)\right)$ the \emph{sign characteristics} of $(S,A)$. Further we call an eigenvalue $\lambda$ of $A$ an eigenvalue of $(S,A)$, and we say that a~real eigenvalue $\lambda$ of $(S,A)$ is of \emph{positive} $($\emph{negative}$)$ type if $(1,\lambda)$ $((-1,\lambda))$ is contained in the sign characteristics of $(S,A)$.
\end{definition}

We will often indicate that an eigenvalue is of positive (negative) type by equipping it with a superscript $"^+"("^-")$. Note that $\lambda$ can be of both negative and positive type.
The next result follows rather directly from the definition of the sign characteristics.

\begin{proposition}\label{prop: compute signs}
	Let $(S,A)\in\R^{n\times n}\times\R^{n\times n}$ be as in Theorem~\ref{can unsim}. An~eigenvalue $\lambda\in\R$ of $(S,A)$ is of positive (negative) type, if and only if there exists an~eigenvector $v\in\Rn\backslash\{0\}$ of $A$ corresponding to the eigenvalue $\lambda$ of $A$ such that $v\T Sv>0$ $(v\T S v<0)$.
\end{proposition}
\begin{proof}
	This is obviously true, if $(S,A)$ is in the canonical form~\eqref{can unsim eq}. The general statement then follows by a~transformation of $(S,A)$ into this canonical form.
\end{proof}


We are now able to define the pole and zero sign~characteristics of a system $[\cA,\cB,\cC]$. Recall that the zeros of a~system $[\cA,\cB,\cC]$ are the eigenvalues of the pencil $\begin{bsmallmatrix}
-s I_n+\cA & \cB \\ \cC & 0
\end{bsmallmatrix}$. If the transfer function is square and invertible, we call a~zero $\mu$ of $[\cA,\cB,\cC]$  \emph{semi-simple}, if it is a~semi-simple eigenvalue of the latter pencil. That is, the order of the zero $\mu$ of $\det\begin{bsmallmatrix}
-s I_n+\cA & \cB \\ \cC & 0
\end{bsmallmatrix}$ equals to the dimension of the kernel of the complex matrix $\begin{bsmallmatrix}
-\mu I_n+\cA & \cB \\ \cC & 0
\end{bsmallmatrix}$.

\begin{definition}\label{def:polezerosignchar}
	Let a~system $[\cA,\cB,\cC]$ with invertible transfer function $\tfG(s)\in\R^{m\times m}$ be given such that, for a~signature matrix $S\in\Gl_n(\R)$, it holds that $S\cA S=\cA\T$ and $\cB=S\cB=C\T$. We say that such a~system is \emph{internally symmetric} (w.r.t. $S$). We call the sign characteristics of $(S,\cA)$ the \emph{pole sign characteristics} of $[\cA,\cB,\cC]$. Further suppose that the zeros of $[\cA,\cB,\cC]$ are semi-simple and denote the real zeros of $[\cA,\cB,\cC]$ by $\mu_1,\ldots,\mu_{k}$. For $i=1,\ldots,k$, let $\begin{bsmallmatrix}
	v_i \\ w_i
	\end{bsmallmatrix}\in\begin{bsmallmatrix}
	-\mu_i I_n+\cA&B\\\cC&0
	\end{bsmallmatrix}$.
	Then we call \begin{equation*}
	\left(\left(-\mathrm{sign}(v_1\T Sv_1),\mu_1\right),\ldots,\left(-\mathrm{sign}(v_k\T Sv_k),\mu_k\right)\right)
	\end{equation*}
	the \emph{zero sign characteristics} of $[\cA,\cB,\cC]$. We say that an eigenvalue $\lambda\in\R$ of $\cA$ is a pole of positive (negative) type of $[\cA,\cB,\cC]$, if $(1,\lambda)$ $((-1,\lambda))$ is contained in the pole sign characteristics of $[\cA,\cB,\cC]$. Similarly, we say that a zero $\mu\in\R$ of $[\cA,\cB,\cC]$ is a zero of positive (negative) type of $[\cA,\cB,\cC]$ if $(1,\mu)$ $((-1,\mu))$ is contained in the zero sign characteristics of $[\cA,\cB,\cC]$.
\end{definition}
Straightforward calculations show that the system $[\cA,\cB,\cC]$ with the properties as in Definition~\ref{def:polezerosignchar} has a~symmetric transfer function. 
On the other hand, note that the notions of pole and zero sign characteristics are defined for symmetric and invertible transfer functions in~\cite[Sec.~II.3.2]{GohbLanc83} by means of pole and zero sign characteristics of a~minimal realization. It is further shown that these are well-defined in the sense that they do not depend on the minimal realization of a~given symmetric and invertible transfer function. The basis for this is that, by the results in~\cite[Sec.~II.3.2]{GohbLanc83}, for any realization $[\cA,\cB,\cC]$ of a~symmetric and invertible transfer function $\tfG(s)\in\R(s)^{m\times m}$, there exists a~unique nonsingular Hermitian matrix $S$ with
\begin{align}
	S\cA &=\cA\T S, & S\cB&=\cC\T, & \cC&=\cB\T S, \label{symmetry of W} \end{align}
Further, for two minimal realizations $[\cA_i,\cB_i,\cC_i]$, $i=1,2$ of $\tfG(s)$, with Hermitian matrices $S_{1}$ and $S_{2}$ from~\eqref{symmetry of W}, the unique state space transformation $T\in\Gl_n(\R)$ between the two realizations fulfills
\begin{equation}\label{eq:change in int symmetry}
	T^{-1}=S_{2}^{-1}T\T S_{1}.
\end{equation}

\end{section}


\begin{section}{Positive real balanced realizations of second order systems}\label{sec:balancing}
The aim in this part is to prove the block structure, introduced in~\eqref{balanced so0}, of the reduced system from Section~\ref{sec: prbt so}. For this purpose, we apply three successive transformations to this system. First, we derive a different first order representation of our second order system. Then develop an input-output normal form from which one can read off the different types of system zeros, namely the real and complex ones and those on the imaginary axis. We use this form to arrive at a positive real balanced realization of the original system and deduce that the reduced model has a balanced realization of the same block structure. Later in Section~\ref{sec:constMDK}, in order to find a second order realization of the reduced system, we actually proceed conversely.

As a~first step towards a~positive real balanced system we consider a~system of the form~\eqref{eq:ABstruc2} with a~special structure which displays the zeros on the imaginary axis.

\begin{lemma}\label{inverse real lemma}
	Let a~system $[\cA,\cB,\cC]$ be given
	with $\cA\in\R^{2n\times 2n}$ and $\cB\in\R^{2n\times m}$ structured as in \eqref{eq:ABstruc2} for some $D\in\R^{n\times n}$ with $D=D^\top$, $G\in\Gl_n(\R)$ and $B\in\R^{n\times m}$ with $\ker B=\{0\}$. Then the transfer function $\tfG(s)$ of $[\cA,\cB,\cC]$ is invertible. Moreover, there exists a state space transformation $T=\diag(T_1,T_2)\in\Gl_{2n}(\R)$ with orthogonal $T_1,T_2\in\Gl_{n}(\R)$, such that $[\cA_{\rms},\cB_{\rms},\cC_{\rms}]=[T^{-1}\cA T,T^{-1}\cB,\cC T]$ has the form
	\begin{equation}\label{zero char form}
	\begin{split}
	\cA_{\rms}=\begin{bmatrix}
	0 & 0 & 0 & 0 & 0 & G_{31}\T \\
	0 & 0 & 0 & 0 & G_{22}\T & G_{32}\T \\
	0 & 0 & 0 & G_{13}\T & 0 & G_{33}\T \\
	0 & 0 &-G_{13}& -D_{11} & 0& -D_{13} \\
	0& -G_{22} & 0 & 0 &0& 0\\
	-G_{31}& -G_{32} & -G_{33} & -D_{13}\T &0& -D_{33}
	\end{bmatrix}, \ \cB_{\rms}=\begin{bmatrix}
	0 \\ 0 \\ 0 \\ 0 \\ 0 \\ B_3
	\end{bmatrix}= \cC_{\rms}\T,
	\end{split}
	\end{equation}
	where, for some $\ell\in\N$, the blocks in the above form are of sizes $m$, $\ell$, $n-m-\ell$, $n-m-\ell$, $\ell$ and $m$. Further, $G_{11},B_3\in\R^{m\times m}$, $G_{22}\in\R^{\ell\times \ell}$ and $G_{13}\in\R^{(n-m-\ell)\times(n-m-\ell)}$ are invertible. All eigenvalues of $\begin{bsmallmatrix}0&G_{13}\T\\-G_{13}&-D_{11}\end{bsmallmatrix}$ have negative real part.
	
	Moreover, the set of zeros of $[\cA,\cB,\cC]$ is given by the union of $\{0\}$ and the spectra of  $\begin{bsmallmatrix} 0 & G_{13}\T\\-G_{13} & -D_{11}\end{bsmallmatrix}$ and $\begin{bsmallmatrix} 0 & G_{22}\T\\-G_{22} & 0\end{bsmallmatrix}$.
	If further, $[\cA,\cB,\cC]$ has semi-simple zeros, then all zero sign characteristics of $[\cA,\cB,\cC]$ at zero are of positive type, whereas the sign characteristics of the remaining real zeros coincide with the sign characteristics of 
\[\left(\begin{bmatrix}
	I_{n-m-\ell} & 0\\  0&-I_{n-m-\ell} 
	\end{bmatrix}, \begin{bmatrix}
	 0 & G_{13}\T\\ -G_{13} & -D_{11}
	\end{bmatrix}\right).\]
\end{lemma}
\begin{proof}
{\em Step 1:} We prove the existence of a~block diagonal state space transformation $T\in\Gl_{2n}(\R)$, such that $[\cA_{\rms},\cB_{\rms},\cC_{\rms}]=[T^{-1}\cA T,T^{-1}\cB,\cC T]$ has the form
 \eqref{zero char form} such that $B_3$, $G_{13}$, $G_{22}$ and $G_{31}$ are invertible, and $D_{11}v\neq0$ for each eigenvector $v\in\R^{n-m-\ell}\setminus\{0\}$ of $G_{13}$.
 
 Since $B$ has full column rank, we have a~QR-decomposition $B={T}_{21}\left[\begin{smallmatrix}0\\B_3\end{smallmatrix}\right]$ with invertible $B_3$. Further, by QR-decomposition of $G^\top T_{21}$ and permutation of rows, we see that there exists some orthogonal $T_{11}\in\R^{n\times n}$ with
\[
  G^\top T_{21}=T_{11}\T\begin{bmatrix}0_{m\times (n-m)}&\widetilde{G}_{21}^\top\\\widetilde{G}_{12}^\top&\widetilde{G}_{22}^\top\end{bmatrix}.
  \]
Now applying the state space transformation $\hat{T}_1=\diag(T_{11},T_{21})$ to $[\cA,\cB,\cC]$ we obtain a~system of the form
\[\hat{T}_1^{\top}\cA \hat{T}_1=\begin{bmatrix}0&0&0&\tilde{G}_{21}\T\\0&0&\tilde{G}_{12}\T&\tilde{G}_{22}\T\\
0&-\tilde{G}_{12}&-\tilde{D}_{11}&-\tilde{D}_{12}\\
-\tilde{G}_{21}&-\tilde{G}_{22}&-\tilde{D}_{12}\T&-\tilde{D}_{22}
\end{bmatrix},\quad \hat{T}_1^{\top}\cB=\big(\cC\hat{T}_1\big)\T=\begin{bmatrix}0\\0\\0\\B_3\end{bmatrix}.\]
Our next step is to apply a~further state space transformation which separates the purely imaginary eigenvalues of
$\begin{bsmallmatrix}0&\tilde{G}_{12}\T\\
-\tilde{G}_{12}&-\tilde{D}_{11}
\end{bsmallmatrix}$
from those with negative real part. To this end, we perform a~singular value decomposition $\widetilde{G}_{12} = T_{12} \hat{G}_{12} T_{22}\T$ with orthogonal matrices $T_{12},T_{22} \in \R^{(n-m) \times (n-m)}$ and a diagonal matrix $\hat{G}_{12} \in \R^{(n-m)\times (n-m)}$. Further, for each eigenspace of $\widehat{G}_{12}$, we perform an orthogonal decomposition into the intersection of this eigenspace with $\ker T_{22}\T\tilde{D}_{11}T_{22}$ and its orthogonal complement in this eigenspace. An accordant orthogonal transformation along with a~permutation matrix leads to the existence of some orthogonal matrices $T_{23},T_{13}\in\R^{(n-m)\times (n-m)}$ that lead to the following transformed matrices: First, for some $\ell\in\N$, $G_{22}\in\R^{\ell\times\ell}$, $G_{13}\in\R^{(n-m-\ell)\times(n-m-\ell)}$ we have
\[
T_{23}\T \widetilde{G}_{12} T_{13} = \begin{bmatrix}0&G_{13}\\G_{22}&0\end{bmatrix}.
\]
Moreover, for some $D_{11}\in \R^{(n-m-\ell)\times(n-m-\ell)}$, $D_{13}\in \R^{(n-m-\ell)\times m}$, $D_{33}\in \R^{m\times m}$ it holds that
\[
\begin{bmatrix}
T_{23}\T T_{22}\T&0\\0&I_m
\end{bmatrix}T_{21}\T DT_{21}\begin{bmatrix}
T_{22}T_{23}&0\\0&I_m
\end{bmatrix}=\begin{bmatrix}
D_{11}&0&D_{13}\\0&0&0\\D_{13}\T&0&D_{33}
\end{bmatrix}.
\]
Last, for each eigenvector $v\in\R^{n-m-\ell}\setminus\{0\}$ of the diagonal matrix $G_{13}$ we have $D_{11}v\neq0$ and for some $G_{31}\in\R^{m\times m}$, $G_{32}\in\R^{m\times(n-m-\ell)}$, $G_{33}\in\R^{m\times(n-m-\ell)}$ it holds that
\[\begin{bmatrix}
T_{23}\T T_{22}\T&0\\0&I_m
\end{bmatrix}T_{21}\T GT_{11}\begin{bmatrix}
I_m&0\\0&T_{12}T_{13}
\end{bmatrix}=\begin{bmatrix}
0&0&G_{13}\\
0&G_{22}&0\\
G_{31}&G_{32}&G_{33}
\end{bmatrix}.\]
In particular, $G_{13}$,  $G_{22}$ and $G_{31}$ are invertible since $G$ is invertible. Altogether, for the orthogonal matrices $T_1=T_{11}\begin{bsmallmatrix}
I_m&0\\0&T_{12}T_{13}
\end{bsmallmatrix}$, $T_2=T_{21}\begin{bsmallmatrix}
T_{22}T_{23} & 0 \\ 0 & I_m
\end{bsmallmatrix}$, a state space transformation with $T=\diag(T_1,T_2)\in\Gl_{2n}(\R)$ results into a~system $[\cA_{\rms},\cB_{\rms},\cC_{\rms}]=[T^{-1}\cA T,T^{-1}\cB,\cC T]$ which is of the form
\eqref{zero char form}.

{\em Step~2:} We prove that for the construction in Step~1 all eigenvalues of $\begin{bsmallmatrix}0&G_{13}\T\\-G_{13}&-D_{11}\end{bsmallmatrix}$ have negative real part.

The fact that all eigenvalues of $\begin{bsmallmatrix}0&G_{13}\T\\-G_{13}&-D_{11}\end{bsmallmatrix}$ have nonpositive real part follows from the fact that the sum of this matrix and its Hermitian is negative semidefinite. To show that it does not have any eigenvalues on the imaginary axis, assume that $\omega\in\R$, $v_1,v_2\in\C^{n-m-\ell}$ are given such that 
\begin{equation}\begin{bmatrix}0&G_{13}\T\\-G_{13}&-D_{11}\end{bmatrix}
\begin{pmatrix}v_1\\v_2\end{pmatrix}=\iunit \omega\begin{pmatrix}v_1\\v_2\end{pmatrix}.\label{eq:13eig}
\end{equation}
A~multiplication of \eqref{eq:13eig} from the right with $\begin{psmallmatrix}v_1\\v_2\end{psmallmatrix}^*$ and taking the real part yields $v_2^* D_{11} v_2=0$, whence, by $D_{11}\geq0$, $D_{11}v_2=0$. Hence
\begin{equation}\begin{bmatrix}0&G_{13}\T\\-G_{13}&0\end{bmatrix}
\begin{pmatrix}v_1\\v_2\end{pmatrix}=\iunit \omega\begin{pmatrix}v_1\\v_2\end{pmatrix}\;\wedge\;D_{11}v_2=0.\label{eq:13eig2}
\end{equation}
On the other hand, the first relation in \eqref{eq:13eig2} yields $G_{13} G_{13}\T v_2=\omega^2 v_2$, whence, by the fact that $G_{13}$ is diagonal, $v_2$ is an eigenvector of $G_{13}$. Then we obtain $v_2=0$ by the results of Step~1, and the invertibility of $G_{13}$ further gives rise to $v_1=0$.

{\em Step~3:} We show that the transfer function $\tfG(s)$ is invertible, and the set of zeros of $[\cA,\cB,\cC]$ is given by the union of $\{0\}$ and the spectra of  $\begin{bsmallmatrix} 0 & G_{13}\T\\-G_{31} & -D_{11}\end{bsmallmatrix}$ and $\begin{bsmallmatrix} 0 & G_{22}\T\\-G_{22} & 0\end{bsmallmatrix}$.

This follows by the fact that
\[\det\begin{bmatrix}
	-sI_{2n}+\cA & \cB \\ \cC & 0
	\end{bmatrix}=c\cdot s^{m}\cdot \det\begin{bmatrix}sI_{n-m-\ell} & -G_{13}\T\\ G_{13} & sI_{n-m-\ell}+D_{11}
	\end{bmatrix}\cdot\det\begin{bmatrix}sI_\ell & -G_{22}\T\\ G_{22} & sI_\ell
	\end{bmatrix}\]
for some $c\in\R\setminus\{0\}$.

{\em Step~4:} We prove the statement about the sign characteristics.

This follows, since $\begin{bsmallmatrix}v_3\\ v_4\end{bsmallmatrix}\in\C^{2(n-m-\ell)}$ is an eigenvector of
$\begin{bsmallmatrix}0&G_{13}\T\\
	-G_{13} & -D_{11}
	\end{bsmallmatrix}$ corresponding to the eigenvalue $\lambda\in\C$, if and only if there exists some $v_7\in\C^m$ such that 
\[\begin{bsmallmatrix}0\\ 0\\ v_3\\ v_4\\ 0\\0\\v_7\end{bsmallmatrix}\in\ker \begin{bmatrix}
	-\lambda I_{2n}+\cA_{\rms} & \cB_{\rms} \\ \cC_{\rms} & 0
	\end{bmatrix}.\]
The statement for the sign characteristics of the zeros at zero is completely analogous. 
\end{proof}

Starting from the second order system structured as in the above lemma, we can determine a~normal form which displays the different type of zeros. In particular, we find that solutions of the KYP inequalities from the positive real lemma are block diagonal matrices structured accordingly to the zero blocks. This helps us to derive a~positive real balanced system of the form~\eqref{balanced so0}. Moreover, in Section~\ref{sec:constMDK}, we take this normal form as a~basis to bring the reduced system back to second order form and to check a~necessary condition whether this is actually possible. Before we present the normal form we need a~small lemma.

\begin{lemma}\label{lem:skewsymm}
	Let $Y\in\R^{n\times n}$ be skew-symmetric and $Z\in\R^{n\times n}$ be symmetric and positive semidefinite. Then
	$ZY+Y\T Z\leq 0$, if and only if $ZY+Y\T Z=0$. 
\end{lemma}
\begin{proof}
Since by using $Z=Z\T$ and $Y=-Y\T$, an evaluation of the diagonal entries of $ZY+Y\T Z$ yields that these vanish, $ZY+Y\T Z\leq 0$ implies that $ZY+Y\T Z= 0$.
\end{proof}

\begin{theorem}\label{nf io}
	Let a~stabilizable system $[\cA,\cB,\cC]$ be given. Assume that $\cA\in\R^{2n\times 2n}$ and $\cB,\cC^\top\in\R^{2n\times m}$ are structured as in \eqref{eq:ABstruc2}  for some $G,D\in\R^{n\times n}$ with $D=D^\top\geq 0$, $G\in\Gl_n(\R)$ and $B\in\R^{n\times m}$. Suppose that the zeros of the system are semi-simple. Then there exists a~state space transformation $T\in\Gl_{2n}(\R)$ such that the system $[\cA_{\rmn}, \cB_{\rmn},\cC_{\rmn}]:=[T^{-1}\cA T,T^{-1}\cB,\cC T]$ has the block form
	\begin{equation}\label{input output normal form2}
		\cA_{\rmn}=\begin{bmatrix}
		0 & 0 & 0 & 0 & 0 & 0 & 0 & \cA_{18} \\
		0 & 0 & 0 & 0 & 0 & 0 & \cA_{27} & \cA_{28} \\
		0 & 0 & 0 & 0 & 0 & \cA_{36} & 0 & \cA_{38} \\
		0 & 0 & 0 & \cA_{44} & 0 & 0 & 0 & \cA_{48} \\
		0 & 0 & 0 & 0 & \cA_{55} & 0 & 0 & \cA_{58} \\
		0 & 0 & -\cA_{36}\T & 0 & 0 & \cA_{66} & 0 & \cA_{68} \\
		0 & -\cA_{27}\T & 0 & 0 & 0 & 0 & 0 & 0 \\
		-\cA_{18}^{\top} & -\cA_{28}^{\top} & -\cA_{38}^{\top} & -\cA_{48}^{\top} & \cA_{58}\T & \cA_{68}\T & 0 & \cA_{88}
		\end{bmatrix},  \quad \cC_{\rmn}\T=\cB_{\rmn}=\begin{bmatrix}
		0 \\ 0 \\ 0 \\ 0 \\ 0 \\ 0 \\ 0 \\ \cB_8
		\end{bmatrix},
	\end{equation}
	where
	\begin{enumerate}[a)]
		\item $\cA_{18},\cB_8\in\Gl_m(\R)$, $\cA_{27}\in\Gl_{\ell}(\R)$, $\cA_{36}\in\Gl_c(\R)$ and $\cA_{66}<0$;
		\item $\cA_{44}=\diag(\mu_1^+,\ldots,\mu_k^+)$ and $\cA_{55}=\diag(\mu_1^-,\ldots,\mu_k^-)$, with $\mu_1^{\pm}\leq\ldots\leq\mu_k^{\pm}< 0$. If $[\cA,\cB,\cC]$ is minimal, then the $\mu_i^+$ and the $\mu_i^-$ are the negative (real and nonzero) zeros of $[\cA,\cB,\cC]$ of positive and negative type, respectively;
		\item $n=k+c+\ell+m$ and $2\ell+m$ is the number of zeros of $[\cA,\cB,\cC]$ on $\iunit\R$ counted with multiplicities.
	\end{enumerate}
	Further, all solutions $P\geq 0$, $Q\geq0$ of the KYP inequalities $\sW_{[\cA_{\rmn},\cB_{\rmn},\cC_{\rmn},0]}(P)\leq 0$ and  $\sW_{[\cA_{\rmn}^{\top},\cC_{\rmn}\T,\cB_{\rmn}\T,0]}(Q)\leq 0$ have the block form $P=\diag(I_{m+\ell},P_2,I_{m+\ell})$ and $ Q=\diag(I_{m+\ell},Q_2,I_{m+\ell})$ for some $P_2,Q_2\in\R^{2(c+k)\times 2(c+k)}$.
\end{theorem}
\begin{proof}
{\em Step 1:} We show that there exists a~state space transformation $\widetilde{T}\in\Gl_{2n}(\R)$ such that $\widetilde{T}\T\sS_{n}\widetilde{T}=\sS_{n}$ and
$[\widetilde{\cA}_{\rmn}, \widetilde{\cB}_{\rmn},\widetilde{\cC}_{\rmn}]:=[\widetilde{T}^{-1}{\cA} \widetilde{T},\widetilde{T}^{-1}\cB,\cC \widetilde{T}]$ has the form~\eqref{input output normal form2}. Moreover, we show that this realization fulfills a)--c) in the above theorem.

    Without loss of generality to assume that the system is already in the form  $[\cA,\cB,\cC]=[\cA_{\rms},\cB_{\rms},\cC_{\rms}]$  as in Lemma~\ref{inverse real lemma}. In particular, all eigenvalues of $\hat{\cA}_{\rms}:=\begin{bsmallmatrix}
	 0 & G_{13}\T  \\
	 -G_{13} & -D_{11}
	\end{bsmallmatrix}\in\R^{2(n-m-\ell)\times2(n-m-\ell)}$
	have negative real part. 
	Since the non-real eigenvalues occur in pairs and $\hat{\cA}_{\rms}\in\R^{2(n-m-\ell)\times2(n-m-\ell)}$, there exists some $c\in\N_0$ such that $2c$ is the number of non-real eigenvalues  (counted with multiplicities). Hence, the number of real eigenvalues of $\hat{\cA}_{\rms}$ (again counted with multiplicities) is $2k$ for $k=n-c-\ell-m$.
	Now, according to Corollary~\ref{right symmetry form}, $k$ real eigenvalues $\mu_1^{+}\leq\ldots\leq\mu_k^{+}< 0$ of $(\sS_{n-m},\hat{\cA}_{\rms})$ are of positive type, whereas the remaining $k$ real eigenvalues $\mu_1^{-}\leq\ldots\leq\mu_k^{-}< 0$ are of negative type.
This corollary further implies that there exists some $\hT\in\Gl_{2(n-m)}$ with $\hT\T \sS_{n-m} \hT=\sS_{n-m}$ and	
	\begin{align*}
		\hT^{-1}\hA_{\rms}\hT=
		\begin{bmatrix}
		 0 & 0 & 0 & {\cA}_{36} \\
		 0 & {\cA}_{44} & 0 & 0  \\
		 0 & 0 & {\cA}_{55} & 0  \\
		 -{\cA}_{36}\T & 0 & 0 & {\cA}_{66} 
	\end{bmatrix},
	\end{align*}	
	where, for some $\eta_i< 0$,  $\nu_i,\in\R\backslash\{0\}$ for $i=1,\ldots,c$ we have
\[\begin{aligned}	
	{\cA}_{44}=&\,\diag(\mu^+_1,\ldots,\mu^+_k)\in\Gl_k(\R),& {\cA}_{55}=&\,\diag(\mu^-_1,\ldots,\mu^-_k)\in\Gl_k(\R),\\ {\cA}_{66}=&\,\diag(\eta_1,\ldots,\eta_{c})\in\Gl_c(\R),& {\cA}_{36}=&\diag(\nu_1,\ldots,\nu_{c})\in\Gl_k(\R).
	\end{aligned}\]
	Then $\widetilde{T}:=\diag(I_{m+\ell},\hT,I_{m+\ell})$ fulfills $\widetilde{T}\T\sS_{n}\widetilde{T}=\sS_{n}$, and an application of this state space transformation leads to a realization $[{\cA}_{\rmn},{\cB}_{\rmn},{\cC}_{\rmn}]$ of the desired form. Further, since $\cA_{18}=G_{31}\T$, $\cA_{27}=G_{22}\T$ and $\cB_8=B_3$, those matrices are invertible by Lemma~\ref{inverse real lemma}.  
	
	{\em Step 2:} Suppose that $P\geq0$ solves $\sW_{[{\cA}_{\rmn},{\cB}_{\rmn},{\cC}_{\rmn},0]}(P)\leq 0$ and partition $P=(P_{ij})_{i,j=1,\ldots,8}$ according to the block structure of ${\cA}_{\rmn}$. We show that $P_{11}=P_{88}=I_m$ and $P_{1i}=0$, $P_{j8}=0$  for $i=2,\ldots,8$ and $j=1,\ldots,7$.
	
	Suppose $P\geq0$ solves the KYP inequality. Then $P \cB_{\rmn}-\cC_{\rmn}^{\top}=0$ together with $\cB_8\in\Gl_m(\R)$ implies that $P_{j8}=0$, for $j=1,\ldots,7$ and $P_{88}=I_m$.
	This implies that the upper left block of size $m\times m$ of the left-hand side of $\cA_{\rmn}^{\top}P+P\cA_{\rmn}\leq 0$ is zero, whence all the corresponding off-diagonal blocks of $\cA_{\rmn}^{\top}P+P\cA_{\rmn}$ have to be zero as well. Equivalently, $-{\cA}_{27} P_{17}\T=0$, $-{\cA}_{36} P_{16}\T=0$, ${\cA}_{44}P_{14}\T=0$, ${\cA}_{55}P_{15}\T=0$, ${\cA}_{36}\T P_{13}\T+{\cA}_{66}P_{16}\T=0$ and ${\cA}_{27}\T P_{12}\T=0$.  Since ${\cA}_{27}$, ${\cA}_{36}$, ${\cA}_{44}$, ${\cA}_{55}$ and ${\cA}_{66}$ are invertible, $P_{1i}=0$ for $i=2,\ldots,8$.
	
	{\em Step 3:} We show that $P_{2i}=0$ and $P_{7i}=0$ for $i=3,\ldots,6$.
	
	We have that
	\begin{equation}\label{eq:purely imaginary blocks}
		\begin{bmatrix}
		0 & -{\cA}_{27} \\ {\cA}_{27}\T & 0
		\end{bmatrix}\begin{bmatrix}
		P_{22} & P_{27} \\
		P_{27}\T & P_{77}
		\end{bmatrix}+\underbrace{\begin{bmatrix}
		P_{22} & P_{27} \\
		P_{27}\T & P_{77}
		\end{bmatrix}}_{=:Z}\underbrace{\begin{bmatrix}
		0 & {\cA}_{27} \\ -{\cA}_{27}\T & 0
		\end{bmatrix}}_{=:Y}\leq 0,
	\end{equation}
	where, due to Lemma~\ref{lem:skewsymm}, equality holds. We set
	\begin{equation*}
		\brA
		:=\begin{bmatrix}
		0 & 0 & 0 & {\cA}_{36} \\ 0 & {\cA}_{44} & 0 & 0 \\ 0 & 0 & {\cA}_{55} & 0 \\
		-{\cA}_{36}\T & 0 & 0 & \cA_{66}
		\end{bmatrix}, \ \breve{P}_1:=\begin{bmatrix}
		P_{33} & \cdots & P_{36} \\
		\vdots & & \vdots \\
		P_{36}\T & \cdots & P_{66}
		\end{bmatrix}, \ \breve{P}_{2}:=\begin{bmatrix}
		P_{23}\T & P_{37} \\ P_{24}\T & P_{47} \\ P_{25}\T & P_{57} \\ P_{26}\T & P_{67}
		\end{bmatrix},
	\end{equation*}
	By considering the principal submatrix of ${\cA}_{\rmn}\T P+P\cA_{\rmn}$ obtained by removing the first and last $m$ rows, we obtain
	\begin{equation*}
		0\geq\begin{bmatrix}
		\brA & 0 \\
		0 & Y \\
		\end{bmatrix}\begin{bmatrix}
		\breve{P}_1 & \breve{P}_2 \\
		\breve{P}_2\T & Z \\
		\end{bmatrix}+\begin{bmatrix}
		\breve{P}_1 & \breve{P}_2 \\
		\breve{P}_2\T & Z \\
		\end{bmatrix}\begin{bmatrix}
		\brA\T & 0 \\
		0 & Y\T \\
		\end{bmatrix}=\begin{bmatrix}
		\brA\T\breve{P}_1+\breve{P}_1\brA & \brA \breve{P}_2+\breve{P}_2Y\T \\ Y \breve{P}_2+\breve{P}_2\brA\T & 0
		\end{bmatrix},
	\end{equation*}
	with $Z,Y$ as in~\eqref{eq:purely imaginary blocks}. Thus, we obtain the Sylvester equation $Y \breve{P}_2\T+\breve{P}_2\T\brA\T=0$ for $\breve{P}_2$. The spectrum of $Y$ is contained on the imaginary axis, and all eigenvalues of $\brA\T$ have negative real part. Since $Y$ and $-\brA\T$ have no common eigenvalues, we obtain from \cite[Thm.~2.4.4.1]{HornJohn12} that $\breve{P}_2=0$. 
	
{\em Step 4:} We show that $\left(\begin{bsmallmatrix}0&-\cA_{27}\\
\cA_{27}\T&0
\end{bsmallmatrix},\begin{bsmallmatrix}\cA_{28}\\0\end{bsmallmatrix}\right)$ is controllable.

Assume that $\lambda\in\C$ and $v_2, v_7\in\C^{\ell}$, such that
\[\begin{bmatrix}v_2^* & v_7^*\end{bmatrix}\begin{bmatrix}\lambda I_{\ell}&\cA_{27}&-\cA_{28}\\
-\cA_{27}\T&\lambda I_{\ell}&0\end{bmatrix}=0.\]
Since   $\begin{bsmallmatrix}0&-\cA_{27}\\
\cA_{27}\T&0
\end{bsmallmatrix}$ is skew-symmetric, it follows that $\lambda=\iunit\omega$ for some $\omega\in\R$. By setting $v:=\begin{bsmallmatrix}0 & v_2^* & 0& 0& 0& 0 & v_7^* & 0 \end{bsmallmatrix}^*\in\C^{2n}$ with block structure according to that of $\cA_{\rmn}$, we now obtain that $v^*\begin{bmatrix}
\iunit\omega I_{2n}-\cA_{\rmn} & \cB_{\rmn}
\end{bmatrix}=0$. Since $[\cA, \cB,\cC]$ is assumed to be stabilizable,   $[\cA_{\rmn}, \cB_{\rmn},\cC_{\rmn}]$ is stabilizable as well. This leads to $v=0$, whence $v_2=v_7=0$.

{\em Step~5:} Let $P\geq 0$, $Q\geq0$ be solutions of the KYP inequalities $\sW_{[\cA_{\rmn},\cB_{\rmn},\cC_{\rmn},0]}(P)\leq 0$ and  $\sW_{[\cA_{\rmn}^{\top},\cC_{\rmn}\T,\cB_{\rmn}\T,0]}(Q)\leq 0$. We show that $P=\diag(I_{m+\ell},P_2,I_{m+\ell})$ and $ Q=\diag(I_{m+\ell},Q_2,I_{m+\ell})$ for some $P_2,Q_2\in\R^{2(c+k)\times 2(c+k)}$.

Since the solutions $P\geq0$ fulfills $\sW_{[\cA_{\rmn},\cB_{\rmn},\cC_{\rmn},0]}(P)\leq 0$, if and only if $Q=\sS_n P\sS_n\geq0$ fulfills  $\sW_{[\cA_{\rmn}^{\top},\cC_{\rmn}\T,\cB_{\rmn}\T,0]}(Q)\leq 0$, it suffices to prove the statement only for $P\geq 0$ with $\sW_{[\cA_{\rmn},\cB_{\rmn},\cC_{\rmn},0]}(P)\leq 0$. We partition $P=(P_{ij})_{i,j=1,\ldots,8}$ according to the block structure of $\cA_{\rmn}$.
As $[\cA_{\rmn},\cB_{\rmn},\cC_{\rmn}]$ is structured as the system $[\widetilde{\cA}_{\rmn},\widetilde{\cB}_{\rmn},\widetilde{\cC}_{\rmn}]$, we can use our findings in Step~2 and Step~3 to see that
$P_{11}=P_{88}=I_m$ and $P_{1i}=0$, $P_{j8}=0$  for $i=2,\ldots,8$ and $j=1,\ldots,7$; and $P_{2i}=0$ and $P_{7i}=0$ for $i=3,\ldots,6$.
Now a~straightforward calculation now yields that
\begin{equation}\begin{bmatrix}0&-\cA_{27}&-\cA_{28}\\
\cA_{27}\T&0&0\\
\cA_{28}\T&0&\cA_{88}
\end{bmatrix}\begin{bmatrix}P_{22}&P_{27}&0\\
P_{27}\T&P_{77}&0\\0&0&I_m
\end{bmatrix}+\begin{bmatrix}P_{22}&P_{27}&0\\
P_{27}\T&P_{77}&0\\0&0&I_m
\end{bmatrix}\begin{bmatrix}0&\cA_{27}&\cA_{28}\\
-\cA_{27}\T&0&0\\
-\cA_{28}\T&0&\cA_{88}
\end{bmatrix}\leq0.\label{eq:33ineq}\end{equation}
An evaluation of the upper right two block gives
\[\begin{bmatrix}0&-\cA_{27}\\
\cA_{27}\T&0
\end{bmatrix}\begin{bmatrix}P_{22}&P_{27}\\
P_{27}\T&P_{77}
\end{bmatrix}+\begin{bmatrix}P_{22}&P_{27}\\
P_{27}\T&P_{77}\end{bmatrix}\begin{bmatrix}0&\cA_{27}\\
-\cA_{27}\T&0
\end{bmatrix}\leq0,\]
whence, by Lemma~\ref{lem:skewsymm}, the latter inequality becomes an equality. Invoking this, an evaluation of the blocks "31" and "32" leads to
\[\begin{bmatrix}\cA_{28}\T&0\end{bmatrix}\begin{bmatrix}P_{22}&P_{27}\\
P_{27}\T&P_{77}\end{bmatrix}=\begin{bmatrix}\cA_{28}\T&0\end{bmatrix}.\]
This altogether yields
\begin{align*}
    \begin{bmatrix}0&\cA_{27}\\
-\cA_{27}\T&0
\end{bmatrix}\left(I_{2\ell}-\begin{bmatrix}P_{22}&P_{27}\\
P_{27}\T&P_{77}
\end{bmatrix}\right)&=\left(I_{2\ell}-\begin{bmatrix}P_{22}&P_{27}\\
P_{27}\T&P_{77}\end{bmatrix}\right)\begin{bmatrix}0&\cA_{27}\\
-\cA_{27}\T&0
\end{bmatrix},\\ \begin{bmatrix}\cA_{28}\T&0\end{bmatrix}\left(I_{2\ell}-\begin{bmatrix}P_{22}&P_{27}\\
P_{27}\T&P_{77}\end{bmatrix}\right)&=0.
\end{align*}
Hence, $\im \left(I_{2\ell}-\begin{bsmallmatrix}P_{22}&P_{27}\\
P_{27}\T&P_{77}\end{bsmallmatrix}\right)$ is an $\begin{bsmallmatrix}0&\cA_{27}\\
-\cA_{27}\T&0
\end{bsmallmatrix}$-invariant subspace and is contained in $\ker\begin{bmatrix}\cA_{28}\T&0\end{bmatrix}$. 
However, by Step~4, $\left(\begin{bsmallmatrix}0&-\cA_{27}\\
\cA_{27}\T&0
\end{bsmallmatrix},\begin{bsmallmatrix}\cA_{28}\\0\end{bsmallmatrix}\right)$ is controllable, thus 
\[\im \left(I_{2\ell}-\begin{bmatrix}P_{22}&P_{27}\\
P_{27}\T&P_{77}\end{bmatrix}\right)=\{0\}\] and  
$P_{22}=P_{77}=I_\ell$ and $P_{27}=0$.
\end{proof}

With the previous theorem at hand and the normal form therein we can now exploit the structure of the original system and the solutions of the KYP inequalities to find a~positive real balanced realization of the original system of the form~\eqref{balanced so0}. 

\begin{theorem}\label{prbt second order}
	Let a~stabilizable system $[\cA,\cB,\cC]$ be given with transfer function $\tfG(s)$. Assume that $\cA\in\R^{2n\times 2n}$ and $\cB,\cC^\top\in\R^{2n\times m}$ are structured as in \eqref{eq:ABstruc2} for some $G,D\in\R^{n\times n}$ with $D=D^\top\geq 0$, $G\in\Gl_n(\R)$ and $B\in\R^{n\times m}$ with $\ker B=\{0\}$. Suppose that the system has semi-simple zeros. Then $\tfG(s)$ has a~positive real balanced realization $[\cA_{\rmb},\cB_{\rmb},\cC_{\rmb}]$ of the block form
	\begin{equation}\label{balanced so2}
		\begin{aligned}	\cA_{\rmb}=&
		\begin{bmatrix}
		0 & 0 & 0 & 0 & 0& \hA_{16} \\
		0  & 0 & 0 &0 & \hA_{25} & \hA_{26} \\
		0 & 0 & \hA_{33} & \hA_{34} & 0& \hA_{36} \\
		0 & 0 & -\hA_{34}^{\top} & \hA_{44} & 0 & \hA_{46} \\
		0 & -\hA_{25}^{\top} & 0 & 0 & 0 & 0 \\
		-\hA_{16}^{\top} & -\hA_{26}^{\top} & -\hA_{36}^{\top} & \hA_{46}^{\top} & 0 & \hA_{66}
		\end{bmatrix}, \ \cB_{\rmb}=\begin{bmatrix}
		0 \\ 0 \\ 0 \\ 0 \\ 0 \\ \hB_6
		\end{bmatrix}=\cC_{\rmb}\T,
		\end{aligned}
	\end{equation}
	where $\hA_{16},\hB_6\in\Gl_m(\R)$, $\hA_{66}\in\R^{m\times m}$, $\hA_{33}\in\R^{p_1\times p_1}$, $\hA_{44}\in\R^{p_2\times p_2}$, $\hA_{25}\in\R^{\ell\times \ell}$ and $\hn=2m+p_1+p_2+2\ell$ is the order of the minimal system. Further, the minimal solutions of the KYP inequalities $\sW_{[\cA_{\rmb},\cB_{\rmb},\cC_{\rmb},0]}(P)\leq 0$ and  $\sW_{[\cA_{\rmb}^{\top},\cC_{\rmb}\T,\cB_{\rmb}\T,0]}(Q)\leq 0$ are given by $P=Q=\diag(I_{m+\ell},\Pi,I_{m+\ell})$ for some diagonal matrix $\Pi\in\R^{(p_1+p_2)\times (p_1+p_2)}$.
	Moreover, the system
	\begin{equation}\label{eq:factor of inverse of Gr}
	    \left[\begin{bmatrix}
			\hA_{33} & \hA_{34} \\ -\hA_{34}\T & \hA_{44}
			\end{bmatrix},\begin{bmatrix}
			\hA_{36} \\ \hA_{46}
			\end{bmatrix},\begin{bmatrix}
			\hA_{36}\T & -\hA_{46}\T
			\end{bmatrix},-\hA_{66}\right]=:[\cA_{\rmz},\cB_{\rmz},\cC_{\rmz},\cD_{\rmz}]
	\end{equation}
	is asymptotically stable and positive real balanced, where $\Pi$ is the minimal solution of $\sW_{[\cA_{\rmz},\cB_{\rmz},\cC_{\rmz},\cD_{\rmz}]}(\hP)\leq 0$ and $\sW_{[\cA_{\rmz}\T,\cC_{\rmz}\T,\cB_{\rmz}\T,\cD_{\rmz}\T]}(\hQ)\leq 0$ and the spectrum of $\cA_{\rmz}$ coincides with the set of zeros of $[\cA_{\rmb},\cB_{\rmb},\cC_{\rmb}]$ with negative real part.
\end{theorem}
\begin{proof}
	Theorem~\ref{nf io} provides that $\tfG(s) $ has a realization $[\cA_{\rmn},\cB_{\rmn},\cC_{\rmn}]$ of the block form~\eqref{input output normal form2}, where $P_{\min}$ and $Q_{\min}$ have the block structure $P_{\min}=\diag(I_{m+\ell},P_{2},I_{\ell+m})$ and $Q_{\min}=\diag(I_{m+\ell},Q_{2},I_{\ell+m})$ for some matrices $P_{2},Q_{2}>0$.
	Let $[\cA_{\rmb},\cB_{\rmb},\cC_{\rmb}]$ be the reduced system of $[\cA_{\rmn},\cB_{\rmn},\cC_{\rmn}]$ constructed as in~\eqref{eq:redfo} for the choice $r^-= n-\dim\ker\Sigma^{-}$ and $r^+= n-\dim\ker\Sigma^{+}$. Here, the matrices $W\T$ and $V$ from~\eqref{eq:redfo} have the same block structure as $P_{\min}$, namely $W\T=\diag(I_{m+\ell},W_{2}\T,I_{m+\ell})$ and  $V=\diag(I_{m+\ell},V_{2},I_{m+\ell})$ for some matrices $W_{2},V_{2}\in\R^{(p_1+p_2)\times (p_1+p_2)}$. Inspecting the proof of Theorem~\ref{thm: prop redsys}, its statements are also valid for the reduced system $[\cA_{\rmb},\cB_{\rmb},\cC_{\rmb}]$. Thus, by Theorem~\ref{thm: prop redsys}~d)~and~c) it follows that $[\cA_{\rmb},\cB_{\rmb},\cC_{\rmb}]$ is a~minimal, positive real balanced realization of $\tfG(s)$ and $\diag(-I_{r^-},I_{r^+})\cA_{\rmb}\diag(-I_{r^-},I_{r^+})=\cA_{\rmb}\T$.
	 Altogether, the realization $[\cA_{\rmb},\cB_{\rmb},\cC_{\rmb}]$ has the block form~\eqref{balanced so2}.
    Moreover, Theorem~\ref{nf io} and the block form of the reduction matrices $V$ and $W$ provide that all solutions of the KYP inequalities $\sW_{[\cA_{\rmb},\cB_{\rmb},\cC_{\rmb},0]}(P)\leq 0$ and  $\sW_{[\cA_{\rmb}^{\top},\cC_{\rmb}\T,\cB_{\rmb}\T,0]}(Q)\leq 0$ have the block form $P=\diag(I_{m+\ell},P_2,I_{m+\ell})$ and $ Q=\diag(I_{m+\ell},Q_2,I_{m+\ell})$ for some $P_2,Q_2\in\R^{(p_1+p_2)\times (p_1+p_2)}$. This in particular holds for the minimal solutions of the two KYP inequalities, which we therefore write as $\diag(I_{m+\ell},\Pi,I_{m+\ell})$ for some suitable diagonal matrix $\Pi\in\R^{(p_1+p_2)\times (p_1+p_2)}$.
    
    Straightforward calculations show that the matrix $\Pi$ solves $\sW_{[\cA_{\rmz},\cB_{\rmz},\cC_{\rmz},\cD_{\rmz}]}(\hP)\leq 0$ and $\sW_{[\cA_{\rmz}\T,\cC_{\rmz}\T,\cB_{\rmz}\T,\cD_{\rmz}\T]}(\hQ)\leq 0$.
    Now, the minimality of $\Pi$ follows directly from the minimality of $\diag(I_{m+\ell},\Pi,I_{m+\ell})$. \\
	Recall that the zeros of $[\cA_{\rmn},\cB_{\rmn},\cC_{\rmn}]$ with negative real part are given by the union of the eigenvalues of $\cA_{44},\cA_{55}$ and $\begin{bsmallmatrix} 0 & \cA_{36} \\ -\cA_{36}\T & \cA_{66} \end{bsmallmatrix}$ from~\eqref{input output normal form2}. As positive real balanced realizations are minimal, the uncontrollable and unobservable modes of $[\cA_{\rmn}, \cB_{\rmn},\cC_{\rmn}]$ are removed by the above construction of $[\cA_{\rmb},\cB_{\rmb},\cC_{\rmb}]$, while the remaining ones are preserved. The latter are, however, the eigenvalues of $\cA_{\rmz}$, which leads to the fact that the spectrum of $\cA_{\rmz}$ is contained in the open left complex half plane.
\end{proof}

The following remark is devoted to positive real balanced truncation of systems $[\cA,\cB,\cC]$ structured as in \eqref{eq:ABstruc2} for some $G,D\in\R^{n\times n}$ with $D=D^\top\geq 0$, $G\in\Gl_n(\R)$ and $B\in\R^{n\times m}$ by using the approach as in \eqref{symmeig2}--\eqref{eq:redfo}. 
\begin{remark}\label{rem:struc redsysso}
Let a~stabilizable system $[\cA,\cB,\cC]$ be given. Assume that $\cA\in\R^{2n\times 2n}$ and $\cB,\cC^\top\in\R^{2n\times m}$ are structured as in \eqref{eq:ABstruc2}  for some $G,D\in\R^{n\times n}$ with $D=D^\top\geq 0$, $G\in\Gl_n(\R)$ and $B\in\R^{n\times m}$. Suppose that the system has semi-simple zeros. We apply positive real balanced truncation by using the approach as in \eqref{symmeig2}--\eqref{eq:redfo}. In particular, we assume that there exist $r^+,r^-\in\N$ that fulfill~\eqref{symmeig2}--\eqref{eq:rpm 1} and $r^+=r^-=:r$ such that the positive real characteristic values $\sigma_{q^{\pm}+1}^{\pm},\ldots,\sigma_{h^{\pm}}^\pm$ are all strictly below one. Note that Theorem~\ref{nf io} implies that $r\geq \ell+m$, where $2\ell$ is the number of nonzero and purely imaginary zeros of $[\cA,\cB,\cC]$ (counted with multiplicities).

A~positive real realization of the reduced order system can be directly constructed from the realization $[\cA_{\rmb},\cB_{\rmb},\cC_{\rmb}]$ in Theorem~\ref{prbt second order} by truncating certain rows and columns belonging to the third and fourth block rows of
the matrices in \eqref{balanced so2}. More precisely, the reduced order model has a~realization of the form
\begin{equation}\label{eq:trunc prbreal}
  [T\T\cA_{\rmb}T,T\T\cB_{\rmb},\cC_{\rmb}T]\quad\text{ for }
T=\diag\left(\begin{bmatrix}I_{r}\\0\end{bmatrix},\begin{bmatrix}0\\I_{r}\end{bmatrix}\right).  
\end{equation}
However, a~direct application of Theorem~\ref{thm: prop redsys}
does not necessarily result into the system $[\widetilde{\cA},\widetilde{\cB},\widetilde{\cC}]=[T\T\cA_{\rmb}T,T\T\cB_{\rmb},\cC_{\rmb}T]$, but rather in a~system which is similar to $[T\T\cA_{\rmb}T,T\T\cB_{\rmb},\cC_{\rmb}T]$ under a state space transformation which preserves the property of the signature structure together with $\diag(\Sigma_1^-,\Sigma_1^+)$ being a~solution of the KYP inequalities for $[\widetilde{\cA},\widetilde{\cB},\widetilde{\cC}]$ and $[\widetilde{\cA}^\top,\widetilde{\cC}^\top,\widetilde{\cB}^\top]$. Such transformations are of block-diagonal and orthogonal type, where the sizes of the orthogonal block correspond to the multiplicities of  the respective positive real characteristic values.
As a consequence, the reduced system can be represented by 
\[[\widetilde{\cA},\widetilde{\cB},\widetilde{\cC}]=
[\widetilde{T}\T\cA_{\rmb}\widetilde{T},\widetilde{T}\T\cB_{\rmb},\cC_{\rmb}\widetilde{T}],\]
where for some orthogonal matrices $\wtU^{\pm}_j\in\R^{n_j^{\pm}\times n_j^{\pm}}$ for $j=1,\ldots,h^{\pm}$
\[\widetilde{T}=\diag(\wtU^-_1,\ldots \wtU_{h^-}^-,\wtU_{h^+}^+,\ldots,\wtU_1^+)\cdot \diag\left(\begin{bmatrix}I_{r}\\0\end{bmatrix},\begin{bmatrix}0\\I_{r}\end{bmatrix}\right).\]


\end{remark}

\end{section}


\begin{section}{Construction of second order realizations}\label{sec:constMDK}

Now, we treat the second order realization problem. In particular, we prove a necessary condition on the zero sign characteristics of the reduced system for having a~representation as a second order realization. 
An important tool hereby are the so called \emph{standard triples} \cite{GohbLanc82a,GohbLanc83,GohbLanc05,LancZaba12}. 

\begin{definition}\label{standard triple}
A~standard triple is defined as a triple $(X,Z,Y)\in\R^{n\times 2n}\times \R^{2n\times 2n}\times \R^{2n\times n}$ such that
\begin{equation*}
	\begin{bmatrix}
	X \\ XZ
	\end{bmatrix}\in\Gl_{2n}(\R)  \quad \text{ and } \quad
	\begin{bmatrix}
	X \\ XZ
	\end{bmatrix}Y=\begin{bmatrix}
	0 \\ M^{-1}
	\end{bmatrix}
\end{equation*}
for some nonsingular matrix $M\in\R^{n\times n}$.

Let $\tfL(s)=s^2 M+s D+K$ for some $M\in\Gl_n(\R)$ and $D,K\in\Rnn$. Then $\tfL(s)$ is said to be \emph{generated} by a~standard triple $(X,Z,Y)$ if $M=(XZY)^{-1}$ and the equation
\begin{equation*}
    MXZ^2+DXZ+KX=0
\end{equation*}
is fulfilled.
\end{definition}

The tuple $(X,Z)$ from a~standard triple together with a~mass matrix $M$ completely determine a~quadratic matrix polynomial with nonsingular leading term as the next result shows.

\begin{theorem}\emph{~\cite[Thm.~1]{Lanc05}}\label{repr triple}
A~standard triple $(X,Z,Y)\in\R^{n\times 2n}\times \R^{2n\times 2n}\times \R^{2n\times n}$ generates a~uniquely defined quadratic matrix polynomial. 
\end{theorem}

Actually, in~\cite{Lanc05} standard triples are considered, where $Z$ is in Jordan canonical form, but it is not made use of this additional property in the proof of the theorem above. 
Even though a~quadratic matrix polynomial is uniquely defined by a~standard triple, the reverse is not true in general. This can be seen as there is a~one-to-one correpondence between standard triples of some quadratic matrix polynomial $\tfL(s)$ and realizations of $\tfL^{-1}(s)$.

\begin{theorem}\emph{~\cite[Thm.~14.2]{LancTism85}}\label{resolvent form st}
	Let $\tfL(s)=s^2 M+s D+K$ for some $M\in\Gl_n(\R)$ and $D,K\in\Rnn$. Then, $(X,Z,Y)\in\R^{n\times 2n}\times \R^{2n\times 2n}\times \R^{2n\times n}$ is a standard triple of $\tfL(s)$, if and only if $\tfL(s)^{-1}=X(s I_{2n}-Z)^{-1}Y$.
	Moreover, $[Z,Y,X]$ is minimal as a realization of $\tfL^{-1}(s)$. In particular, if $K^{-1}$ exists, then
	$K^{-1}=-XZ^{-1}Y$.
\end{theorem}

It can be further inferred from Theorem~\ref{repr triple} that the coefficients of $\tfL(s)$ are given by the so-called {\em moments} $\Gamma_j:=XZ^jY$, i.\,e.,
\begin{equation}\label{coefficient moments}
	M=\Gamma_1^{-1}, \quad D=-M\Gamma_2M, \quad K=-M\Gamma_3M+D\Gamma_1D.
\end{equation}
A standard triple $(X,Z,Y)$ is said to be \emph{self-adjoint} if there exists a symmetric matrix $S\in\Gl_{n}(\R)$, such that $SY=X^{\top}, \ Z^{\top}=SZS^{-1}$. The theorem above states that $[Z,Y,X]$ is the realization of a~symmetric transfer function und thus $S$ is unique. Note that if $\tfL(s)$ possesses a self-adjoint standard triple, then the moments $\Gamma_j$ are symmetric and hence, by~\eqref{coefficient moments} the coefficients of $\tfL(s)$ are symmetric as well. Thus, by Theorem~\ref{resolvent form st} together with~\eqref{symmetry of W}, if one standard triple is self-adjoint, then all standard triples of $\tfL(s)$ are self-adjoint. Moreover, in this case the pole sign characteristics of $\tfL^{-1}(s)$ are given by the sign characteristics of $(S,Z)$ for any self-adjoint standard triple $(X,Z,S^{-1}X\T)$ of $\tfL(s)$.
We need to introduce some further notation.

\begin{definition}
	Let $(S,A)\in\Rnn\times\Rnn$ be given, where $S\in\Gl_n(\R)$ is symmetric and $A$ is $S$-self-adjoint and diagonalizable over $\C$. For $\alpha<0$, we denote by $n_-(\alpha)$ $(p_-(\alpha))$ the number of eigenvalues of $(S,A)$ of negative (positive) type in $[\alpha,0)$ $((\alpha,0])$ and for $\alpha>0$, $n_+(\alpha)$ $(p_+(\alpha))$ the number of eigenvalues of $(S,A)$ of negative (positive) type in $[0,\alpha)$ $((0,\alpha])$.	
\end{definition}

Next we present a~result from \cite{LancZaba14} that gives a~necessary condition on the sign characteristics of quadratic matrix polynomials with positive definite leading and trailing coefficient. 

\begin{theorem}\emph{~\cite[Thm.~16]{LancZaba14}}\label{leading trailing signs}
	Let $(S,A)\in\R^{2n\times 2n}\times\R^{2n\times 2n}$ be given, where $S=U\T\sS_n U$ for some $U\in\Gl_{2n}(\R)$ and $A$ is $S$-self-adjoint and diagonalizable over $\C$. Let $\lambda_{\min}$ and $\lambda_{\max}$ be its minimal and maximal real eigenvalue, respectively. Then there exists $X\in\R^{n\times 2n}$ such that $X(sI_{2n}-A)^{-1}S^{-1}X\T=(s^2M+sD+K)^{-1}$ for some $M>0$, $K\geq 0$ and $D=D\T$, if and only if
	\begin{equation}\label{sign outside intervall}
	\begin{split}
	n_-(\alpha)=&p_-(\alpha) \quad \forall \, \alpha<\lambda_{\min} \ \text{ and } \
	n_{-}(\alpha)\leq p_{-}(\alpha) \quad \forall \, \alpha\in[\lambda_{\min},0), \\
	n_+(\alpha)=&p_+(\alpha) \quad \forall \, \alpha>\lambda_{\max} \ \text{ and } \
	n_+(\alpha)\geq p_+(\alpha) \quad \forall \, \alpha\in (0,\lambda_{\max}].
	\end{split}
	\end{equation}
\end{theorem}

\begin{remark}
	In \cite[Thm.~16]{LancZaba14}, pairs $(S,A)$ are considered with a~structure similar to the canonical form of Theorem~\ref{can unsim}. However, we have seen earlier that the sign characteristics of $(S,A)$ from any self-adjoint standard triple $(X,A,S^{-1}X\T)$ of $\tfL(s)$ are exactly the pole sign characteristics of $\tfL^{-1}(s)$ and hence, do not depend on the choice of the minimal realization of $\tfL^{-1}(s)$ or the standard triple of $\tfL(s)$.
\end{remark}

We are now able to prove the main result of this section on a~necessary and sufficient condition for our reduced system possessing a second order realization with positive definite leading and trailing coefficient. Note that the proof of this theorem is constructive, i.\,e., we can infer a~method for constructing such a~second order realization. We restrict ourselves to the case where the reduced system is minimal. However, later in Section~\ref{sec:algorithm}, when we present the whole reduction procedure, minimality is not required.

\begin{theorem}\label{thm:eq cond so}
	Let a~stabilizable system $[\cA,\cB,\cC]$ be given in which
 $\cA\in\R^{2n\times 2n}$ and $\cB\in\R^{2n\times m}$ are structured as in \eqref{eq:ABstruc2} for some $G,D\in\R^{n\times n}$ with $D=D^\top\geq 0$, $G\in\Gl_n(\R)$ and $B\in\R^{n\times m}$. Suppose that the system has semi-simple zeros.
	Let further $[\widetilde{\cA},\widetilde{\cB},\widetilde{\cC}]$ be the reduced system from Theorem~\ref{thm: prop redsys} of order $2r$, where $r^+=r^-=:r$. Suppose it is minimal and its zeros are semi-simple. Then, the following are equivalent:
	\begin{enumerate}[i)]
		\item The transfer function $\tilde{\tfG}(s)$ of $[\widetilde{\cA},\widetilde{\cB},\widetilde{\cC}]$ has a second order realization of the form~\eqref{eqn:redsosys}, where $\wtK,\wtM>0$ and $\wtD=\wtD\T\in\R^{r\times r}$ has at least $r-m$ positive eigenvalues. 
		\item The real zeros of $[\widetilde{\cA},\widetilde{\cB},\widetilde{\cC}]$ are given by $\mu_1^-\leq\ldots\leq\mu_k^-$ and $\mu_1^+\leq\ldots\leq\mu_k^+<\mu_{k+1}^+=\ldots=\mu_{k+m}^+=0$ and fulfill $\mu_i^-<\mu_i^+$ for all $i=1,\ldots,k$.	
		\item The transfer function $\tilde{\tfG}(s)$ has a minimal realization $[ \cA_{\rmn},\cB_{\rmn},\cC_{\rmn}]$ structured as in~\eqref{input output normal form2} that fulfills the properties a)--c)~from Theorem~\ref{nf io} and  $\cA_{44}-\cA_{55}=\diag(\mu_1^+,\ldots,\mu_k^+)-\diag(\mu_1^-,\ldots,\mu_k^-)>0$.
	\end{enumerate}
\end{theorem}
\begin{proof}
	$ii) \Rightarrow iii)$: Let $\tfG(s)$ be the transfer function of $[\cA,\cB,\cC]$ and  $[\widetilde{\cA}_{\rmb},\widetilde{\cB}_{\rmb},\widetilde{\cC}_{\rmb}]$ be its the positive real balanced realization given by Theorem~\ref{prbt second order} of the form~\eqref{balanced so2}. Due to Remark~\ref{rem:struc redsysso} we can assume w.\,l.\,o.\,g.\  that $[\widetilde{\cA},\widetilde{\cB},\widetilde{\cC}]=[T\T\cA_{\rmb}T,T\T\cB_{\rmb},\cC_{\rmb}T]$  for 
    $T=\diag\left(\begin{bsmallmatrix}I_{r}\\0\end{bsmallmatrix},\begin{bsmallmatrix}0\\I_{r}\end{bsmallmatrix}\right)$. Let 
    \[[\widetilde{\cA}_{\rmz},\widetilde{\cB}_{\rmz},\widetilde{\cC}_{\rmz},\cD_{\rmz}]=[\wtT\T\cA_{\rmz}\wtT,\wtT\T\cB_{\rmz},\cC_{\rmz}\wtT,\cD_{\rmz}] \] 
    for $\wtT:=\diag\left(\begin{bsmallmatrix}0\\I_{r-m-\ell}\end{bsmallmatrix},\begin{bsmallmatrix}I_{r-m-\ell}\\0\end{bsmallmatrix}\right)$ and $[\cA_{\rmz},\cB_{\rmz},\cC_{\rmz},\cD_{\rmz}]$ be as in~\eqref{eq:factor of inverse of Gr}. 
    Moreover, the eigenvalues of $\widetilde{\cA}_{\rmz}$ are zeros of $[\widetilde{\cA},\widetilde{\cB},\widetilde{\cC}]$ and the real nonzero zeros of $[\widetilde{\cA},\widetilde{\cB},\widetilde{\cC}]$ are eigenvalues of $\widetilde{\cA}_{\rmz}$. 
	Since $[\widetilde{\cA},\widetilde{\cB},\widetilde{\cC}]$ is received by positive real balanced truncation of the minimal and asymptotically stable system $[\cA_{\rmz},\cB_{\rmz},\cC_{\rmz},\cD_{\rmz}]$ from~\eqref{eq:factor of inverse of Gr}, the discussion in Section~III of~\cite{HarsJonc83} provides that the system $[\widetilde{\cA},\widetilde{\cB},\widetilde{\cC}]$ is asymptotically stable as well. Hence we can apply  Corollary~\ref{right symmetry form} to find some~$\hT\in\Gl_{2p}(\R)$ such that
	\begin{equation*}
		\hT^{-1}\widetilde{\cA}_{\rmz}\hT=\begin{bmatrix}
		0 & 0 & \cV \\
		0 & \Lambda  & 0 \\
		-\cV & 0 & \cE \\
		\end{bmatrix} \ \text{ and } \ \hT\T \sS_{p} \hT=\sS_{p},
	\end{equation*}
	where $\Lambda=\diag(\mu^+_1,\ldots,\mu^+_k,\mu^-_1,\ldots,\mu^-_k)$ and $\cE=\diag(\eta_1,\ldots,\eta_c)$ for $c:=p-k$, $\cV=\diag(\nu_1,\ldots,\nu_c)$, with $\eta_i< 0$ and $\nu_i\in\R$ for $i=1,\ldots,c$. Applying the state space transformation $\brT:=\diag(I_{m+\ell},\hT,I_{m+\ell})$ to $[\widetilde{\cA},\widetilde{\cB},\widetilde{\cC}]$ then gives a minimal realization $[\cA_{\rmn},\cB_{\rmn},\cC_{\rmn}]$ of $\tilde{\tfG}(s)$ structured as in Theorem~\ref{nf io} and which fulfills the properties a)--c)~of this theorem. Now the zero sign characteristics of $[\widetilde{\cA},\widetilde{\cB},\widetilde{\cC}]$ are given by the sign characteristics of $(\diag(I_m,-\sS_{2k}),\diag(0,\Lambda))$. Hence, by assumption we get that  $\diag(\mu_1^+,\ldots,\mu_k^+)-\diag(\mu_1^-,\ldots,\mu_k^-)>0$. 
	
	$iii) \Rightarrow ii)$: This follows since $\mu_1^{+},\ldots,\mu_k^{+}$ and $\mu_1^{-},\ldots,\mu_k^{-}$ are respectively the negative zeros of positive and negative type of the minimal realization $[\cA_{\rmn},\cB_{\rmn},\cC_{\rmn}]$ of $\tilde{\tfG}(s)$ which is structured as in~\eqref{input output normal form2} and $\tilde{\tfG}(0)=\cC_{\rmn}\cA_{\rmn}^{-1}\cB_{\rmn}=0$. 
	
	$iii) \Rightarrow i)$: Suppose we are given a minimal realization $[\cA_{\rmn},\cB_{\rmn},\cC_{\rmn}]$ structured as in Theorem~\ref{nf io} and such that $\diag(\mu_1^+,\ldots,\mu_k^+)-\diag(\mu_1^-,\ldots,\mu_k^-)>0$.
	For $i=1,\ldots,k$ we set
	\begin{equation*}
		a_i:= \sqrt{\tfrac{\mu_i^-}{\mu_i^--\mu_i^+}}\in\R\backslash\{0\}, \quad b_i:=\sqrt{\tfrac{\mu_i^+}{\mu_i^--\mu_i^+}}\in\R\backslash\{0\},\quad
		T_i=\begin{bmatrix}
		a_i & b_i \\ b_i & a_i
	\end{bmatrix}\in\Gl_2(\R).
	\end{equation*}
	Straightforward calculations give $T_i^{-1}=\sS_2 T_i\sS_2$ and
	\begin{equation}\label{eq:Tireal}
		T_i^{-1}\diag(\mu_i^+,\mu_i^-)T_i=\begin{bmatrix}
		0 & * \\ * & \tfrac{(\mu_i^-)^2-(\mu_i^+)^2}{\mu_i^--\mu_i^+}
		\end{bmatrix}.
	\end{equation}
	Note that since $\mu_i^{\pm}<0$ for all $i=1,\ldots,k$, the lower right entry of the matrix on the right-hand side is negative. We set $\widetilde{T}:=\begin{bsmallmatrix}
	A & B \\ B & A
	\end{bsmallmatrix}$, where $A:=\diag(a_1,\ldots,a_k)$ and $B:=\diag(b_1,\ldots,b_k)$ and $T:=\diag(I_{r-k},\widetilde{T},I_{r-k})$. Interchanging some rows and columns, we arrive at a minimal realization $[\cA,\cB,\cC]$ of $\tilde{\tfG}(s)$ of the form~\eqref{eq:redsys1}, where a second order realization can be derived directly with $\wtM=I_r$, $\wtK=\wtG\wtG\T$ and $\wtD=\wtD\T$. Note that $\wtD$ has the block form $\wtD=\begin{bsmallmatrix}
	\wtD_{11} & \wtD_{12} \\ \wtD_{12}\T & \wtD_{22}
	\end{bsmallmatrix}$, where $\wtD_{11}\in \R^{r-m\times r-m}$ and $\wtD_{22}\in\R^{m\times m}$ are positive semidefinite and diagonal matrices. Hence, Sylvester's law of inertia \cite[Thm.~4.5.8]{HornJohn12} implies that $\wtD$ has at least $r-m$ non-negative eigenvalues.
	
	$i) \Rightarrow iii)$: Suppose $\tilde{\tfG}(s)$ has a minimal second order realization of the form~\eqref{eqn:redsosys}, where $\wtM\in\R^{r\times r}$. By Lemma~\ref{inverse real lemma}, $\tilde{\tfG}(s)$ then also has a minimal realization $[\cA_{\rms},\cB_{\rms},\cC_{\rms}]$ of the form~\eqref{zero char form}, where its real and nonzero zeros coincide with the real eigenvalues of $\hA_{\rms}$
	and the zero sign characteristics of $[\cA_{\rms},\cB_{\rms},\cC_{\rms}]$ coincide with the sign characteristics of $(\diag(I_m,-\sS_{r-m}),\diag(0,\hA_{\rms}))$. Since, as mentioned in the first part, all zeros of $[\widetilde{\cA},\widetilde{\cB},\widetilde{\cC}]$ have nonpositive real part, we can apply Theorem~\ref{nf io} to obtain a~minimal realization $[\cA_{\rmn},\cB_{\rmn},\cC_{\rmn}]$ of $\tilde{\tfG}(s)$ of the form~\eqref{input output normal form2}.
	Recall that the transformation $T\in\Gl_{2r}(\R)$ we use to transform $[\cA_{\rms},\cB_{\rms},\cC_{\rms}]$ into $[\cA_{\rmn},\cB_{\rmn},\cC_{\rmn}]$ has the block structure $T=\diag(I_m,\hT,I_m)$, where $\hT \in\Gl_{2(r-m)}(\R)$ has to fulfill~\eqref{eq:change in int symmetry}, i.\,e., $\hT\T \sS_{r-m} \hT=\sS_{r-m}$. This implies that the sign characteristics of $(-\sS_{r-m},\hA_{\rms})$ coincide with the sign characteristics of $(-\sS_k,\diag(\cA_{44},\cA_{55}))$, with $\cA_{44},\cA_{55}$ from~\eqref{input output normal form2}. Setting $X:=\begin{bmatrix}
	0 & I_{r-m}
	\end{bmatrix}$, with $\hA_{\rms}=\begin{bsmallmatrix}
	0 & G_{13}\T \\ -G_{13} & -D_{11}
	\end{bsmallmatrix}$ as in Lemma~\ref{inverse real lemma} we obtain
	\begin{equation*}
		\begin{bmatrix}
		X\hA_{\rms}^{-1} \\ X\hA_{\rms}^{-1}\hA_{\rms}
		\end{bmatrix}=\begin{bmatrix}
		-G_{13}^{-\top} & 0 \\ 0 & I_{r-m}
		\end{bmatrix} \quad \text{and} \quad \begin{bmatrix}
		X\hA_{\rms}^{-1} \\ X\hA_{\rms}^{-1}\hA_{\rms}
		\end{bmatrix}\sS_{r-m}X\T=\begin{bmatrix}
		0 \\ I_{m-r}
		\end{bmatrix}.
	\end{equation*}
	Thus, using~\eqref{coefficient moments} and Theorem~\ref{resolvent form st}, $(X\hA_{\rms}^{-1},\hA_{\rms},\sS_{r-m}X\T)$ is a standard triple for
	\begin{equation*}
		\mu^2I_{r-m}-\mu X\hA_{\rms}\sS_{r-m}X\T-(X(\hA_{\rms}^{-1})^2\sS_{r-m}X\T)^{-1}=\mu^2I_{r-m}+\mu D_{11}+G_{13} G_{13}\T.
	\end{equation*}
	Moreover, the standard triple is $S$-self-adjoint for $S:=\sS_{r-m}\hA_{\rms}^{-1}$, since
	\begin{equation*}
		\begin{split}
			 S\hA_{\rms}S^{-1}=&\sS_{r-m}\hA_{\rms}^{-1}\hA_{\rms}\hA_{\rms}\sS_{r-m}=\sS_{r-m}\hA_{\rms}\sS_{r-m}=\hA_{\rms}^{\top}, \\ (X\hA_{\rms}^{-1})\T=&\hA_{\rms}^{-\top}X\T=\hA_{\rms}^{-\top}\sS_{r-m}\sS_{r-m}X\T=\sS_{r-m}\hA_{\rms}^{-1}\sS_{r-m}X\T=S\sS_{r-m}X\T.
		\end{split}
	\end{equation*}
	 Now Theorem~\ref{leading trailing signs} implies that for this standard triple, $n_{-}(\alpha)\leq p_{-}(\alpha)$ for all $\alpha\in[\mu_{\min},0)$. Further, we get that
	\begin{equation*}
		\hT\T \hA_{\rms}^{-\top}\hT^{-\top}=\begin{bmatrix}
		-\cA_{36}^{-\top}\cA_{66}\cA_{36}^{-1} & 0 & 0 & \cA_{36}^{-\top}  \\
		0 & \cA_{44}^{-1} & 0 & 0 \\
		0 & 0 & \cA_{55}^{-1} & 0 \\
		-\cA_{36}^{-1} & 0 & 0 & 0
		\end{bmatrix}.
	\end{equation*}
	The sign characteristics only depend on the real eigenvalues and thus, $(\sS_{r-m}\hA_{\rms}^{-1},\hA_{\rms})$ and  $(\hT\T\sS_{r-m}\hT\hT^{-1}\hA_{\rms}^{-1}\hT,\hT^{-1}\hA_{\rms}\hT)$ have the same sign characteristics as the tuple $(\diag(\cA_{44}^{-1},\cA_{55}^{-1})\sS_k,\diag(\cA_{44},\cA_{55}))$. Recall that $\cA_{44}$ and $\cA_{55}$ are negative definite. Hence, the sign characteristics of the latter coincide with the sign characteristics of $(-\sS_k,\diag(\cA_{44},\cA_{55}))$. Therefore, we get that $n_{-}(\alpha)\leq p_{-}(\alpha)$ for all $\alpha\in[\mu_{\min},0)$ also holds for $(-\sS_{k},\diag(\cA_{44},\cA_{55}))$, which means that $\cA_{55}<\cA_{44}$.
\end{proof}

We close this section with a~discussion on the treatment of the necessary condition that is needed in order to recover a~second order realization which we have learned previously. 

\begin{remark}[Handling of the necessary condition]
		Unfortunately, the condition $\diag(\mu_1^+,\ldots,\mu_k^+)-\diag(\mu_1^-,\ldots,\mu_k^-)$ on the zeros of the reduced system is not fulfilled in general. 
		Moreover, we cannot guarantee to establish this property by small perturbations, since the sign characteristics are stable in the following sense. Consider a pair $(H,A)\in\Rnn\times \Rnn$, where $H$ is symmetric and nonsinuglar and $A$ is $H$-self-adjoint and diagonalizable over $\C$. Then for every simple real eigenvalue $\lambda$, there exist structure preserving neighborhoods $U_A$ of $A$, $U_H$ of $H$ and $U_{\lambda}$ of $\lambda$, such that if $A_1\in U_A$ is $H_1$-self-adjoint for some $H_1\in U_H$, then there is exactly one real simple eigenvalue $\lambda_1$ of $A_1$ in $U_{\lambda}$ and it holds that it has the same sign as $\lambda$, see~\cite[Sec.~III.5.1, Thm.~5.1]{GohbLanc83}. Due to these facts, we have decided that in our reduction method to add blocks instead of relying on perturbations of the system matrices, whenever it is necessary, see Section~\ref{sec:algorithm}.
\end{remark}


\begin{section}{Positive real balanced truncation for overdamped systems - the exceptional case}\label{sec:overdamped}
	Here we consider second order systems of the form~\eqref{eqn:sosys} that additionally fulfill that $M$, $D$ and $K$ are all symmetric and positive definite, and the so called \emph{overdamping condition}
	\begin{equation}\label{overdamping condition}
	(v^{*}Dv)^2>4(v^{*}Mv)(v^{*}Kv) \ \text{for all} \ v\in\Cn
	\end{equation}
	is fulfilled. It is known that~the overdamping condition implies that all zeros of $\det(s^2M+s D+K)$ are real, see \cite{Duff55}; it therefore generalizes the discriminant condition to the matrix-valued case. By forming a~self-adjoint standard triple $(X,A,S^{-1}X\T)$ of $\tfL(s)=s^2M+s D+K$, we can define the sign characteristics of $\tfL(s)$ to be those of $(S,A)$ \cite{GohbLanc80}.
	 The sign characteristics of overdamped systems have a~certain structure, which is summarized in the following result.
	
	\begin{theorem}\cite[Thm~3.6]{MahaTiss12}\label{overdamping and signs}
		Let $(X,A,S^{-1}X\T)$ form a self-adjoint standard triple of some matrix polynomial $s^2M+sD+K$, where $M,D,K\in\Rnn$ are symmetric. Then the following are equivalent:
		\begin{enumerate}[i)]
			\item All eigenvalues of $(S,A)$ are real and negative and every eigenvalue of negative type is smaller than every eigenvalue of positive type. In detail, they fulfill $\lambda_{1}^-\leq \ldots \leq \lambda_n^-<\lambda_1^+\leq \ldots \leq \lambda^+_{n}<0$.
			\item It holds that $M,D,K>0$ and~\eqref{overdamping condition}.
		\end{enumerate}
	\end{theorem}

	Noticing that the distribution of the real eigenvalues and their signs are system invariants, we obtain the following result.
	
	\begin{corollary}\label{overdamping and real}
		Let a~second order system be given of the form~\eqref{eqn:sosys}, with $M,D,K>0$, which fulfills the condition~\eqref{overdamping condition} and let $\tfG(s)$ be its transfer function. Then every second order realization of the form~\eqref{eqn:sosys} with symmetric $\tM,\tD,\tK\in\Rnn$ fulfills $\tM,\tD,\tK>0$ and~\eqref{overdamping condition}. In this case, for the first order representation $[\cA,\cB,\cC]$ of the form~\eqref{eq:linsys1}, the eigenvalues of $(\sS_n,\cA)$ fulfill $\lambda_{1}^-\leq \ldots \leq \lambda_n^-<\lambda_1^+\leq \ldots \leq \lambda^+_{n}<0$.
	\end{corollary}

	In order to show that positive real balanced truncation preserves the overdamping structure, we use some results of~\cite{BindYe95} and therefore need some notation.
	
	\begin{definition}\label{rhodeltasigma}
		Let $(S,A)\in\R^{2n\times 2n}\times\R^{2n\times 2n}$, where $S$ is symmetric and nonsingular, $A$ is $S$-self-adjoint and diagonalizable over $\C$.
		Consider the set $C^{\pm}:=\{x\in\mathbb{C}^{2n}\ | \  x^*Sx\gtrless0\}$ and, for a~subspace $\cS\subseteq \mathbb{C}^{2n}$, let
		\begin{equation}
		\iota^{+}(\cS):= \begin{cases}
		\  \ +\infty & \text{if } \cS\cap C^+=\emptyset, \\
		\sup\limits_{x\in\cS\cap C^+} \rho(x) & \text{otherwise}
		\end{cases}, \ \iota^{-}(\cS):= \begin{cases}
		\  \ -\infty & \text{if } \cS\cap C^-=\emptyset, \\
		\inf\limits_{x\in\cS\cap C^-} \rho(x) & \text{otherwise},
		\end{cases}
		\end{equation}
		where $\rho(x):=\frac{x^*SAx}{x^* Sx}$. We denote by
		\begin{equation*}
		\begin{split}
		\delta^+_h(S,A):=&\{\inf\iota^+(\cS)\ | \ \cS\text{ subspace of } \mathbb{C}^{2n} \text{ with }\dim\cS=2n-h+1\} \\
		\text{and} \quad \sigma^-_h(S,A):=&\{\sup\iota^-(\cS)\ | \  \cS\text{ subspace of } \mathbb{C}^{2n} \text{ with }\dim\cS=2n-h+1\}.
		\end{split}
		\end{equation*}
	\end{definition}

	Note that one has $\delta_{2n}^+(S,A)\leq \ldots \leq \delta_1^+(S,A)$ and $\sigma_1^-(S,A)\leq \ldots \leq \sigma_{2n}^-(S,A)$. We present two results from~\cite{BindYe95}. The second one treats the assignment of the values $\sigma_j^-(S,A)$ and $\delta_j^+(S,A)$ to eigenvalues of positive and negative type for tuples $(S,A)$ as in Definition~\ref{rhodeltasigma} where all eigenvalues are real and $\lambda_1^-\leq \ldots \leq \lambda^-_{n}\leq\lambda_{1}^+\leq \ldots \leq \lambda_n^+$. In the first result, we study the values $\sigma_{j}^-(S,A)$ for tuples $(S,A)$ that do not satisfy such a~distribution of the sign characteristics.
	
	\begin{lemma}\cite[Lem.~3.2]{BindYe95}\label{muplus}
		Let $(S,A)\in\R^{2n\times 2n}\times\R^{2n\times 2n}$, where $S=U\T \sS_n U$ for some $U\in\Gl_{2n}(\R)$ and $A$ is $S$-self-adjoint and diagonalizable over $\C$. Suppose the eigenvalues of $(S,A)$ do not satisfy $\lambda_1^-\leq \ldots \leq \lambda^-_{n}\leq\lambda_{1}^+\leq \ldots \leq \lambda_n^+$, i.\,e., $(S,A)$ has at least one non-real eigenvalue, or there exists an eigenvalue of positive type which is larger than an eigenvalue of negative type. Then, $\sigma_{n+j}^-(S,A)=\infty$ for $1\leq j\leq n$.
	\end{lemma}

	\begin{theorem}\cite[Cor.~4.4 \& Thm.~4.3]{BindYe95}\label{exceptional case}
		Let $(S,A)\in\R^{2n\times 2n}\times\R^{2n\times 2n}$, where $S=U\T \sS_n U$ for some $U\in\Gl_{2n}(\R)$ and $A$ is $S$-self-adjoint and diagonalizable. Suppose that all eigenvalues of $(S,A)$ are real and  satisfy $\lambda_1^-\leq \ldots \leq \lambda^-_{n}\leq\lambda_{1}^+\leq \ldots \leq \lambda_n^+$. Then, for $1\leq j\leq n$ we have
		\begin{align*}
			\sigma^-_j(S,A)=&-\infty,  \quad & \sigma^-_{n+j}(S,A)&=\lambda_j^-, \\ \quad \delta^+_{2n-j+1}(S,A)=& \ \lambda^+_{j}, \quad\qquad\text{ and } & \delta^+_j(S,A)&=\infty. 
		\end{align*}
		In particular, if $\lambda^-_{n}={\tiny \lambda}_{1}^+$, then  $\sigma^-_{2n}(S,A)=\lambda^-_{n}(S,A)=\lambda_{1}^+(S,A)=\delta^+_{2n}(S,A)$.
	\end{theorem}

	A consequence of Theorem~\ref{exceptional case} is that for $S$-self-adjoint and diagonalizable $A$ with only real eigenvalues which additionally fulfill $\lambda_1^-\leq \ldots \leq \lambda^-_{n}\leq\lambda_{1}^+\leq \ldots \leq \lambda_n^+$, this property is preserved under the application of certain two-sided reduction matrices on the pair $(S,A)$.
	
	\begin{lemma}\label{exceptional case subspace}
		Let $(S,A)\in\R^{2n\times 2n}\times\R^{2n\times 2n}$, where $S=U\T \sS_n U$ for some $U\in\Gl_{2n}(\R)$ and $A$ is $S$-self-adjoint. Suppose that $A$ is diagonalizable over $\mathbb{C}$ and that $(S,A)$ has only real eigenvalues that satisfy $\lambda_1^-\leq \ldots \leq \lambda^-_{n}<\lambda_{1}^+\leq \ldots \leq \lambda_n^+$. Let $W,V\in\R^{2n\times r}$ be full rank matrices with $r\leq 2n$ and suppose that for a symmetric matrix $H\in\R^{r\times r}$ it holds that $H W\T=V\T S$ and $W\T V=I_r$. If the eigenvalues of $(H,W^{\top}AV)$ are semi-simple, 
		then they are all real and, if we denote them by $\tilde{\lambda}_1^{\pm}\leq \ldots \leq \tilde{\lambda}^{\pm}_{r^{\pm}}$, where $r^++r^-=r$, they  fulfill $\tilde{\lambda}_1^-\leq \ldots \leq \tilde{\lambda}^-_{r^-}<\tilde{\lambda}_1^+\leq \ldots \leq \tilde{\lambda}_{r^+}^+$. 
	\end{lemma}
	\begin{proof}
		We set $B:=W^{\top}AV$ and for $v\in\C^r$ we write
		\begin{equation*}
			\tilde{\rho}(v):=\frac{v^*V^{\top}SAVv}{v^*V\T SVv}=\frac{v^*HBv}{v^* Hv}.
		\end{equation*}  By the previous Theorem we have that $\lambda^-_{n}=\sigma^-_{2n}(S,A)=\sup\{\rho(x) \ | \ x\in C^-\}$. Then,
		\begin{equation*}
			\begin{split}
				\sigma_r^-(H,B)=\sup\{\tilde{\rho}(v) \ | \ v\in \C^r, \ v^{*}Hv>0\}&=\sup\{\rho(x) \ | \ x\in\C^{2n}, \, x\in C^-\cap \bild V\} \\
				&\leq \sigma^-_{2n}(S,A)<\infty.
			\end{split}
		\end{equation*}  Now, from Lemma~\ref{muplus} it follows that all eigenvalues of $(H,B)$ are real and $\tilde{\lambda}_1^-\leq \ldots \leq \tilde{\lambda}^-_{r^-}\leq\tilde{\lambda}_1^+\leq \ldots \leq \tilde{\lambda}_{r^+}^+$. Hence, it remains to show that $\tilde{\lambda}^-_{r^-}<\tilde{\lambda}_{1}^+$. Note that
		\begin{equation*}
			\begin{split}
				\delta_r^+(H,B)=\inf\{\tilde{\rho}(v) \ | \ v\in \C^r, \ v^{*}Hv<0\}&=\inf\{\rho(x) \ | \ x\in\C^{2n},\, x\in C^+\cap \bild V\} \\
				&\geq\inf\{\rho(x) \ | \ x\in C^+ \}=\delta_{2n}^+(S,A).
			\end{split}
		\end{equation*} In particular, by Theorem~\ref{exceptional case} it holds that $\tilde{\lambda}^-_{r^-}=\tilde{\sigma}_r^-(H,B)\leq \sigma_{2n}^-(S,A)=\lambda^-_{n}<\lambda^-_{1}=\delta_{2n}^+(S,A)\leq \delta_{r}^+(H,B)=\tilde{\lambda}_{1}^+$.
	\end{proof}
	
	For the main theorem of this section we combine the lemma above with the results from Section~\ref{sec:balancing}. More precisely, we prove that the structure of an~overdamped second order system is preserved applying the reduction ansatz from Section~\ref{sec: prbt so}. Naturally, the reduction matrices $V$ and $W\T$ from~\eqref{eq:redfo} that are used to derive the reduced system fulfill the requirements of the lemma above, i.\,e., $W\T V=I$ and $\sS_{r}W\T=V\T \sS_n$. Nevertheless, in order to prove that the poles and zeros of the reduced model are all smaller than zero we have to work a~bit more.
	
	\begin{theorem}
		Let a~stabilizable system $[\cA,\cB,\cC]$ be given. Assume that
        $\cA\in\R^{2n\times 2n}$ and $\cB\in\R^{2n\times m}$ are structured as in \eqref{eq:ABstruc2} for some $G,D\in\R^{n\times n}$ with $D=D^\top\geq 0$, $G\in\Gl_n(\R)$ and $B\in\R^{n\times m}$. Let $M:=I_n$ and $K:=GG\T$ and assume that $ D>0$ and they fulfill~\eqref{overdamping condition}.
    	Let further $[\widetilde{\cA},\widetilde{\cB},\widetilde{\cC}]$ be the reduced system from Theorem~\ref{thm: prop redsys} of order $2r$ for $r:=r^+=r^-$ and let $\tilde{\tfG}(s)$ be its transfer function. Suppose that the eigenvalues of $\cA$ and $\widetilde{\cA}$ and the zeros of $[\cA,\cB,\cC]$ and $[\widetilde{\cA},\widetilde{\cB},\widetilde{\cC}]$ are all semi-simple. Then all eigenvalues of $\cA$ are real and fulfill $\tilde{\lambda}_{1}^-\leq \ldots \leq \tilde{\lambda}_r^-<\tilde{\lambda}_1^+\leq \ldots \leq \tilde{\lambda}^+_{r}$. In particular, every second order realization of $\tilde{\tfG}(s)$ of the form~\eqref{eqn:sosys} with symmetric $\tM,\tD,\tK\in\R^{r\times r}$ fulfills $\tM,\tD,\tK>0$ and~\eqref{overdamping condition}. Moreover, $\tilde{\tfG}(s)$ possesses such a~second order realization.
	\end{theorem}	
	\begin{proof}
        A~combination of Lemma~\ref{inverse real lemma} and Corollary~\ref{overdamping and real} implies that we can w.\,l.\,o.\,g. assume that $[\cA_{\rms},\cB_{\rms},\cC_{\rms}]$ is of the form~\eqref{zero char form} and all eigenvalues of $(\sS_n,\cA)$ are real, negative, and they fulfill $\lambda_{1}^-\leq \ldots \leq \lambda_n^-<\lambda_1^+\leq \ldots \leq \lambda^+_{n}<0$. Since $D>0$, the second last block rows and columns of size $\ell$ in~\eqref{zero char form} are absent. 
        Let $\tfG(s)$ denote the transfer function of $[\cA,\cB,\cC]$ and let $[\cA_{\rmb},\cB_{\rmb},\cC_{\rmb}]$ be a~positive real balanced realization of the form~\eqref{balanced so2} which exists by Theorem~\ref{prbt second order}. 
        Due to Remark~\ref{rem:struc redsysso} we can w.\,l.\,o.\,g. assume that $[\widetilde{\cA},\widetilde{\cB},\widetilde{\cC}]=[T\T\cA_{\rmb}T,T\T\cB_{\rmb},\cC_{\rmb}T]$  for 
        $T = \diag \left(\begin{bsmallmatrix} I_{r} \\ 0 \end{bsmallmatrix},\begin{bsmallmatrix} 0 \\ I_{r} \end{bsmallmatrix}\right)$. Let
        \begin{equation*}
         [\widetilde{\cA}_{\rmz},\widetilde{\cB}_{\rmz},\widetilde{\cC}_{\rmz},\cD_{\rmz}] = [\wtT\T\cA_{\rmz}\wtT,\wtT\T\cB_{\rmz},\cC_{\rmz}\wtT,\cD_{\rmz}]
         \end{equation*}
        for $\wtT:=\diag\left(\begin{bsmallmatrix}0\\I_{r-m}\end{bsmallmatrix},\begin{bsmallmatrix}I_{r-m}\\0\end{bsmallmatrix}\right)$ and $[\cA_{\rmz},\cB_{\rmz},\cC_{\rmz},\cD_{\rmz}]$ as in~\eqref{eq:factor of inverse of Gr}.
        Since the eigenvalues of $\cA_{\rmb}$ are a~subset of the eigenvalues of $\cA$ they are all real and the eigenvalues of negative type of $(\diag(-I_{m+p_1},I_{p_2+m}),\cA_{\rmb})$ are all strictly smaller than the eigenvalues of positive type of $(\diag(-I_{m+p_1},I_{p_2+m}),\cA_{\rmb})$. Now the tuple $(\diag(-I_{m+p_1},I_{p_2+m}),\cA_{\rmb})$ and reduction matrices $V:=T$ and $W\T:=T\T$ fulfill the assumptions of Lemma~\ref{exceptional case subspace} and hence, so does $(T\T\diag(-I_{m+p_1},I_{p_2+m}) T,T\T \cA_{\rmb} T)=(\sS_r,\widetilde{\cA})$ and $\wtT$. It follows that the eigenvalues of $(\sS_r,\widetilde{\cA})$ and of  $(\sS_{r-m},\widetilde{A}_{\rmz})$ are all real and fulfill $\tilde{\lambda}_{1}^-\leq \ldots \leq \tilde{\lambda}_r^-<\tilde{\lambda}_1^+\leq \ldots \leq \tilde{\lambda}^+_{r}$ and $\tilde{\mu}_{1}^-\leq \ldots \leq \tilde{\mu}_{r-m}^-<\tilde{\mu}_1^+\leq \ldots \leq \tilde{\mu}^+_{r-m}$, respectively. On the other hand, the system $[\widetilde{\cA}_{\rmz},\widetilde{\cB}_{\rmz},\widetilde{\cC}_{\rmz},\cD_{\rmz}]$ is received by positive real balanced truncation of the asymptotically stable and positive real balanced (and thus, minimal) system $[\cA_{\rmz},\cB_{\rmz},\cC_{\rmz},\cD_{\rmz}]$. Now, the discussion in Section~III of~\cite{HarsJonc83} provides that the system $[\widetilde{\cA}_{\rmz},\widetilde{\cB_{\rmz}},\widetilde{\cC}_{\rmz},\cD_{\rmz}]$ is asymptotically stable as well and hence, $\tilde{\mu}^+_{r-m}<0$. In particular, $\widetilde{\cA}_{\rmz}$ is invertible. Since by Theorem~\ref{prbt second order},  $\hA_{16}\in\Gl_m(\R)$, which is the block matrix on the upper right of $\cA_{\rmb}$ and also of $\widetilde{\cA}$, it follows from the block form of $\widetilde{\cA}$ that it is invertible as well. By the passivity of the reduced system we have that  $\tilde{\lambda}^+_{r}\leq0$ and hence, $\tilde{\lambda}^+_{r}<0$. We have now shown that property a) from Theorem~\ref{overdamping and signs} is fulfilled for the pole sign characteristics of the reduced system.
        
        Since the second condition of Theorem~\ref{thm:eq cond so} is fulfilled, we follow the steps ``$ii)\Rightarrow iii)$'' and ``$iii)\Rightarrow i)$'' to obtain a~state space transformation $\brT\in\Gl_{2r}(\R)$ such that $\brT^{-1}\widetilde{\cA}\brT$, $\brT^{-1}\widetilde{\cB}$ and $\widetilde{\cC}\brT$ are structured as in~\eqref{eq:ABstruc2} for some matrices $\wtG\in\Gl_r(\R)$, $\widetilde{B}\in\R^{r\times m}$ and some symmetric matrix $\wtD\in\R^{r\times r}$. Note that the minimality assumption of Theorem~\ref{thm:eq cond so} is not needed in order to derive the transformation $\brT$. This shows that $\tilde{\tfG}(s)$ possesses a~second order realization of the form~\eqref{eqn:sosys} with symmetric coefficients $\wtM= I_r$, $\wtK= \wtG\wtG\T$, and $\wtD$. Now Theorem~\ref{overdamping and signs} provides that it fulfills $\wtM,\wtD,\wtK>0$ and~\eqref{overdamping condition}.
		The rest of the theorem follows from Corollary~\ref{overdamping and real}.
    \end{proof}

\end{section}


\begin{section}{Numerical aspects}\label{sec:algorithm}
	
	This part is devoted to the numerical issues of the presented results. 
	One problem that occurs considering the original large-scale system is that we need some factorization of the mass and stiffness matrices $M$ and $K$. Here, for many applications in mechanics those are band matrices or have an equally sparse structure, a property that one can be exploited to compute sparse Cholesky factorizations.
	
	The bottleneck in the computation of the reduced-order system the the determination of the minimal solution of the KYP inequality $\sW_{[\cA,\cB,\cC,0]}(P)\leq 0$. Here one can use the method from \cite{PoloReis12}, which provides the minimal solution directly in factored form. That is, this method delivers low rank approximations of the type $P_{\min}\approx L\T L$ for a~matrix $L$ which has a small number of rows compared to the number of columns. We arrive at the following numerical procedure.\\ \newline 
\textbf{Numerical procedure: Second order positive real balanced truncation and structure recovery.}
		For a~second order system of the form~\eqref{eqn:sosys}, where $M,K>0$ and $D\geq 0$, compute a~reduced system of the form~\eqref{eqn:redsosys}, where $\wtM,\wtK>0$ and $\wtD=\wtD\T$. 
		\\
	{\bf Inputs:} Matrices $M,K>0$, $D\geq 0$ and $B$, a~tolerance value $\tol>0$. \\
	{\bf Outputs:} Matrices $\wtM,\wtK>0$, $\wtD=\wtD\T$, and $\widetilde{B}$, gap metric error bound $\err$.
	\begin{enumerate}[1)]
		\item Start with the reduction by postitive real balanced truncation as presented in Section~\ref{sec: prbt so}:
	\begin{enumerate}[a)]
		\item Use the sparse Cholesky factors $G$ and $H$ of $M,K$ to derive the first order representation~\eqref{eq:linsys1}.
		\item Compute the low rank Cholesky factor $L\T L=P_{\min}$. One can use for example the Lur'e solver, introduced in~\cite{PoloReis12}, that computes low rank factors of a~$P>0$ from a~\emph{stabilizing solution triple} $(P,\cK,\cL)\in\R^{2n\times 2n}\times \R^{m\times 2n}\times \R^{m\times m}$ of the so called \emph{Lur'e equation}
		\begin{equation}\label{LMI ODE2}
		\begin{bmatrix}
		\cA^{\top}P+P \cA & P \cB-\cC^{\top} \\
		\cB^{\top}P-\cC & 0
		\end{bmatrix}+\begin{bmatrix}
		\cK\T  \\ \cL\T
		\end{bmatrix}\begin{bmatrix}
		\cK& \cL
		\end{bmatrix}=0,
		\end{equation}
		where the term stabilizing refers to the rank condition
		\begin{equation*}
		\rank \begin{bmatrix}
		\lambda I_{2n} +\cA & \cB \\ \cK & \cL
		\end{bmatrix}= n+\rank \begin{bmatrix}
		\cK & \cL
		\end{bmatrix} \quad \text{ for all } \ \lambda\in\C^+.
		\end{equation*}
		\item Take the partitioned eigendecomposition
		\begin{equation}\label{Sigma pm12 alg}
		L\sS_nL\T =\begin{bmatrix}U^-_{1}&U_{2}&U^+_{1}\end{bmatrix}\begin{bmatrix}-\Sigma^-_{1}&0&0\\0&\sS_{n-r}\Sigma_{2}&0\\0&0&\Sigma^+_{1}\end{bmatrix}
		\begin{bmatrix}(U^-_{1})\T\\(U_{2})\T\\(U^+_{1})\T\end{bmatrix},
		\end{equation}
		and define reduction matrices
		$
		W\T:=\Sigma_1^{-\frac12}\sS_rU_1\T L$ and $V:=\sS_nL\T U_1\Sigma_1^{-\frac12},
		$
		where $\Sigma_1:=\diag(\Sigma_1^-, \Sigma_1^+)$ and $U_1:=\begin{bmatrix}U^-_{1} & U^+_{1}\end{bmatrix}$, to
		compute a reduced first order model
		\begin{equation}
		\dot{\wtx}(t) = \widetilde{\cA}\wtx(t)+\widetilde{\cB} u(t), \quad
		\wty(t)=\widetilde{\cC} \wtx(t),
		\label{eq:redfo alg}
		\end{equation}
		where $\widetilde{\cA}:=W\T \cA V$, $\widetilde{\cB}:=W\T \cB$ and $\widetilde{\cC}:=\cC V$. The gap metric error bound $\err$ is twice the sum of the diagonal entries of $\Sigma_2$.
	\end{enumerate}
	\item Continue with the second order structure recovery from Section~\ref{sec:balancing} and Section~\ref{sec:constMDK}.
	\begin{enumerate}[a)]
		\item If necessary apply a~block orthogonal state space transformation as in Remark~\ref{rem:struc redsysso} to derive a~realization $[\widetilde{\cA}_{\rmb},\widetilde{\cB}_{\rmb},\widetilde{\cC}_{\rmb}]$ of the form~\eqref{balanced so2} with $p_1=p_2=\widetilde{k}$.
		\item For $\widetilde{\cA}_{\rmz}:=\begin{bsmallmatrix}
		\widetilde{\cA}_{22} & \widetilde{\cA}_{23} \\ -\widetilde{\cA}_{23}\T & \widetilde{\cA}_{33}
		\end{bsmallmatrix}\in\R^{2\widetilde{k}\times 2\widetilde{k}}$, compute $\wtT\in\Gl_{2\widetilde{k}}(\R)$ such that $\wtT^{-1} \widetilde{\cA}_{\rmz}\wtT=\diag(\brA_{22},\brA_{33})$ with $\brA_{22}=\diag(\widetilde{\mu}_1,\ldots,\widetilde{\mu}_{2\widetilde{k}})$ and $\brA_{33}=\diag(\cP_{\sigma_1,\tau_1},\ldots,\cP_{\sigma_c,\tau_c})$, where $\cP_{\sigma_1,\tau_1}=\begin{bsmallmatrix}\sigma_i & \tau_i \\ -\tau_i & \sigma_i\end{bsmallmatrix}$. Perform a state space transformation with the matrix $\diag(I_{m+\ell},\wtT,I_{m+\ell})$ to derive a~realization $[\brA,\brB,\brC]$ that we partition as 
		\begin{equation}\label{eq:blockform redsys alg}
		\begin{split}
		\brA=&\begin{bmatrix}
		0 & 0 & 0 & \brA_{14} \\
		0 & \brA_{22} & 0 & \brA_{24} \\
		0 & 0 & \brA_{33} & \brA_{34} \\
		-\brA_{14}\T & \brA_{42} & \brA_{43} & \brA_{44}
		\end{bmatrix}, \ \brA_{24}=\begin{bmatrix}
		\ba_1 \\ \vdots \\ \ba_{2\widetilde{k}}
		\end{bmatrix}, \ \brA_{42}\T=\begin{bmatrix}
		\bb_1 \\ \vdots \\ \bb_{2\widetilde{k}}
		\end{bmatrix}, \\
		\widetilde{\cA}_{34}\T=&\begin{bmatrix}
		\bc_1 & \bd_1 & \cdots & \bc_{c} & \bd_c
		\end{bmatrix}, \
		\brA_{43}=\begin{bmatrix}
		\bolde_1 & \boldf_1 & \cdots & \bolde_{c} & \boldf_c
		\end{bmatrix}.
		\end{split}
		\end{equation}
	 \item To retrieve the symmetry structure construct a~transformation $\brT$:
        \begin{itemize}
            \item For $i\in\{1,\ldots,q\}$ choose $j\in\{1,\ldots,m\}$ such that $|a_{i,j}|\cdot|b_{i,j}|>\tol$, where $a_{i,j},b_{i,j}$ are the $j$th entry of $\ba_i$ and $\bb_i$, respectively, and set $\breve{t}_i:=\big|\tfrac{a_{i,j}}{b_{i,j}}\big|$. 
            \item For $i\in\{1,\ldots,c\}$ we either find some $j\in\{1,\ldots,m\}$ such that for the $j$'th entries of $\bc_i,\bd_i,\bolde_i$, and $\boldf_i$, which we call $c_{i,j},d_{i,j},e_{i,j},d_{i,j}$, respectively, $|c_{i,j}|,|f_{i,j}|,|e_{i,j}|,|d_{i,j}|>\tol$ or $\sigma_i+\iunit \tau_i$ is close to being an unobservable or uncontrollable mode in which case $|c_{i,j}|,|f_{i,j}|,|e_{i,j}|,|d_{i,j}|\leq\tol$ for all $j\in\{1,\ldots,m\}$. In the first case set $z_i^2:=\tfrac{d_i^2+c_i^2}{e_i^2+f_i^2}$ and find $\begin{bsmallmatrix}\widetilde{x}_i\\ \widetilde{y}_i\end{bsmallmatrix}\neq 0$ which solves
            \[\begin{bmatrix}
            e_i z_i^2 -d_i & -f_i z_i^2 -c_i \\
            f_i z_i^2 -c_i & e_i z_i^2 +d_i
            \end{bmatrix}\begin{bmatrix}\widetilde{x}_i\\ \widetilde{y}_i\end{bmatrix}=0.\]
            Scale the vector above via $\begin{bsmallmatrix}x_i\\ y_i\end{bsmallmatrix}:=\sqrt{\tfrac{z_i}{\widetilde{x}_i^2+\widetilde{y}_i^2}}\begin{bsmallmatrix}\widetilde{x}_i\\ \widetilde{y}_i\end{bsmallmatrix}$ and set
            $\brT_i:=\begin{bsmallmatrix}
            x_i & y_i \\ -y_i & x_i
            \end{bsmallmatrix}$. 
            \item If for all $j\in\{1,\ldots,m\}$ $|a_{i,j}|\cdot|b_{i,j}|\leq\tol$ or $|c_{i,j}|,|f_{i,j}|,|e_{i,j}|,|d_{i,j}|\leq\tol$ we can freely choose $\breve{t}_1=1$ and $\breve{T}_i=I_2$.
        \end{itemize}
        Define the transformation $\brT:=\diag(I_{m+\ell},\breve{t}_1,\ldots,\breve{t}_{2\widetilde{k}},\brT_1,\ldots,\brT_c,I_{m+\ell})$ to get a~realization $[\hA,\hB,\hC]:=[\brT^{-1}\brA \brT,\brT^{-1}\brB,\brC\brT]$. 
        Apply a~state space transformation $\breve{V} := \diag(I_{m+\ell+2\widetilde{k}},\Theta_1,\ldots,\Theta_c,I_{m+\ell})$, where $\Theta_i$ is taken from~\eqref{Pi} and exchange the rows and columns to get a~realization $[\cA_{\rmn},\cB_{\rmn},\cC_{\rmn}]$ of the form of~\eqref{input output normal form2}.

        \textbf{Note}:
        The fact that $z_i>0$ and thus, $\sqrt{z_i}\in\R$ is guaranteed by the following arguments: By following the steps in the proof of $ii) \Rightarrow iii)$ in Theorem~\ref{leading trailing signs}, we see that there exists a~state space transformation for the reduced system that leads to a~system of the form $[\cA_{\rmn},\cB_{\rmn},\cC_{\rmn}]$. Moreover, such a~state space transformation, except from the rows and columns corresponding to unobservable or uncontrollable modes, has to be of the form $\brT\breve{V}$. 
		
		\item If the submatrices $\cA_{44}=\diag(\widetilde{\mu}_1^+,\ldots,\widetilde{\mu}_{\widetilde{k}}^+)$ and $\cA_{55}=\diag(\widetilde{\mu}_{1}^-,\ldots,\widetilde{\mu}_{\widetilde{k}}^-)$ from $\cA_{\rmn}$ as in~\eqref{input output normal form2} do not fulfill $\cA_{44}-\cA_{55}>0$: find $0>\widetilde{\mu}_{q}^+\geq\ldots\geq\widetilde{\mu}_{\widetilde{k}+1}^+>\max\{\widetilde{\mu}_{\widetilde{k}}^-, \widetilde{\mu}_{\widetilde{k}}^+\}$ and $\widetilde{\mu}_{q}^-\leq\ldots\leq\widetilde{\mu}_{\widetilde{k}+1}^-<\min\{\widetilde{\mu}_1^-, \widetilde{\mu}_1^+\}$ and replace $\cA_{44}$, $\cA_{55}$, $\cA_{48}\T$, $\cA_{58}\T$ with $\brA_{44}:=\diag(\cA_{44},\Lambda^+)$, $\brA_{55}:=\diag(\Lambda^-,\cA_{55})$, $\brA_{48}\T:=\begin{bmatrix}
		\cA_{48}\T & 0
		\end{bmatrix}$, and $\brA_{58}\T:=\begin{bmatrix}
		0 & \cA_{58}\T
		\end{bmatrix}$, where $\Lambda^-=\diag(\widetilde{\mu}_{q}^-,\ldots,\widetilde{\mu}_{\widetilde{k}+1}^-)$ and $\Lambda^+=\diag(\widetilde{\mu}_{\widetilde{k}+1}^+,\ldots,\widetilde{\mu}_{q}^+)$ such that $\brA_{44}-\brA_{55}>0$. Abusing notations we rename $\brA_{55}=\diag(\widetilde{\mu}_1^-,\ldots,\widetilde{\mu}_{q}^-)$.
		\item For $i=1,\ldots,q$ set $a_i:= \sqrt{\tfrac{\widetilde{\mu}_i^-}{\widetilde{\mu}_i^--\widetilde{\mu}_i^+}}$, $b_i:=\sqrt{\tfrac{\widetilde{\mu}_i^+}{\widetilde{\mu}_i^--\widetilde{\mu}_i^+}}$ and
		\begin{equation*}
			T_i:=\begin{bmatrix}
			a_i & b_i \\ b_i & a_i
			\end{bmatrix}, \quad T:=\diag(I_{m+\ell},T_1,\ldots,T_{2\hat{k}},I_{\ell+m}).
		\end{equation*}
		Apply the state space transformation $T$ to the system.
		\item Suitable block exchanges brings the system into a realization $[\cA_{\rms},\cB_{\rms},\cC_{\rms}]$ of the block form~\eqref{zero char form} from Lemma~\ref{inverse real lemma}. A second order realization of $\tilde{\tfG}(s)$ is given by:
		\begin{equation*}
			\begin{split}
			\ddot{\wtp}(t)+\underbrace{\begin{bmatrix}
			D_{11} & D_{12} \\ D_{12}\T & D_{22}
			\end{bmatrix}}_{=:\wtD}\dot{\wtp}(t)+\underbrace{\begin{bmatrix}
			0 & G_{21}\T \\ G_{12}\T & G_{22}\T
			\end{bmatrix}\begin{bmatrix}
			0 & G_{12} \\ G_{21} & G_{22}
			\end{bmatrix}}_{=:\wtK}\wtp(t)=&\underbrace{\begin{bmatrix}
			0 \\ \hB_6
			\end{bmatrix}}_{=:\widetilde{B}} u(t), \\ \wty(t)=&\widetilde{B}\T\dot{\wtp}(t).
			\end{split}
		\end{equation*}
	\end{enumerate}
	\end{enumerate}	
	\begin{remark}\label{jordan blocks}
		In the case that the submatrix $\widetilde{\cA}_{\rmz}$ of $\wtA$ as in~\eqref{eq:factor of inverse of Gr} does not have semi-simple eigenvalues, one can perturb the blocks of $\widetilde{\cA}_{\rmz}$ corresponding to positive real characteristic values lower than one. In doing so, if additionally $\widetilde{\cC}\widetilde{\cA}\widetilde{\cB}>0$, the newly formed system $[\widetilde{\cA}+\Delta,\widetilde{\cB},\widetilde{\cC}]$, for some sufficiently small $\|\Delta\|_2$, will then still be passive, see~\cite{BeatMehr19}. An $H^{\infty}$-error bound can be computed.
	\end{remark}
	
We illustrate the performance of the above procedure with an~example of three coupled mass-spring-damper
chains, see~\cite[Ex.~2]{TruhVese09}.

\begin{example}\label{ex:msd}
	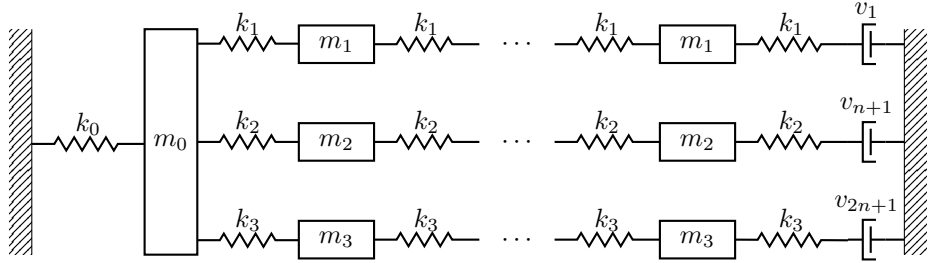
\begin{figure}
	\centering
\begin{tikzpicture}[every node/.style={draw,outer sep=0pt,thick}]
\tikzstyle{spring}=[thick,decorate,decoration={zigzag,pre length=0.3cm,post length=0.3cm,segment length=6}]
\tikzstyle{damper}=[thick,decoration={markings,  
	mark connection node=dmp,
	mark=at position 0.5 with 
	{
		\node (dmp) [thick,inner sep=0pt,transform shape,rotate=-90,minimum width=15pt,minimum height=3pt,draw=none] {};
		\draw [thick] ($(dmp.north east)+(2pt,0)$) -- (dmp.south east) -- (dmp.south west) -- ($(dmp.north west)+(2pt,0)$);
		\draw [thick] ($(dmp.north)+(0,-5pt)$) -- ($(dmp.north)+(0,5pt)$);
	}
}, decorate]
\tikzstyle{ground}=[fill,pattern=north east lines,draw=none,minimum width=0.75cm,minimum height=0.3cm]

\begin{scope}[xshift=7cm]
\node (M0) [minimum width=0.5cm, minimum height=3cm] {$m_0$};

\node (wall) [ground, rotate=-90, minimum width=3cm,yshift=-2cm] {};
\draw (wall.north east) -- (wall.north west);

\draw [spring] (wall.80) -- ($(M0.north west)!(wall.80)!(M0.south west)$) node[draw=none,fill=none,pos=0.5,above] {$k_0$};

\node (M1) [minimum width=1cm, minimum height=0.5cm, xshift=+2.2cm, yshift=+1.3cm] {$m_1$};
\draw [spring] (M0.75) -- ($(M1.north west)!(M0.75)!(M1.south west)$) node[draw=none,fill=none,pos=0.5,above] {$k_1$};

\node (M2) [minimum width=1cm, minimum height=0.5cm, xshift=+2.2cm, yshift=+0.0cm] {$m_2$};
\draw [spring] (M0.0) -- ($(M2.north west)!(M0.0)!(M2.south west)$) node[draw=none,fill=none,pos=0.5,above] {$k_2$};

\node (M3) [minimum width=1cm, minimum height=0.5cm, xshift=+2.2cm, yshift=-1.30cm] {$m_3$};
\draw [spring] (M0.-75) -- ($(M3.north west)!(M0.-75)!(M3.south west)$) node[draw=none,fill=none,pos=0.5,above] {$k_3$};

\node (dots1) [draw=none,minimum width=1cm, minimum height=0.5cm, xshift=+4.6cm, yshift=+1.3cm] {$\bold{\cdots}$};
\draw [spring] (M1.0) -- ($(dots1.north west)!(M1.0)!(dots1.south west)$) node[draw=none,fill=none,pos=0.5,above] {$k_1$};

\node (dots2) [draw=none,minimum width=1cm, minimum height=0.5cm, xshift=+4.6cm, yshift=+0.0cm] {$\bold{\cdots}$};
\draw [spring] (M2.0) -- ($(dots2.north west)!(M2.0)!(dots2.south west)$) node[draw=none,fill=none,pos=0.5,above] {$k_2$};

\node (dots3) [draw=none,minimum width=1cm, minimum height=0.5cm, xshift=+4.6cm, yshift=-1.30cm] {$\bold{\cdots}$};
\draw [spring] (M3.0) -- ($(dots3.north west)!(M3.0)!(dots3.south west)$) node[draw=none,fill=none,pos=0.5,above] {$k_3$};

\node (M12) [minimum width=1cm, minimum height=0.5cm, xshift=+7cm, yshift=+1.3cm] {$m_1$};
\draw [spring] (dots1.0) -- ($(M12.north west)!(dots1.0)!(M12.south west)$) node[draw=none,fill=none,pos=0.5,above] {$k_1$};

\node (M22) [minimum width=1cm, minimum height=0.5cm, xshift=+7cm, yshift=+0.0cm] {$m_2$};
\draw [spring] (dots2.0) -- ($(M22.north west)!(dots2.0)!(M22.south west)$) node[draw=none,fill=none,pos=0.5,above] {$k_2$};

\node (M32) [minimum width=1cm, minimum height=0.5cm, xshift=+7cm, yshift=-1.3cm] {$m_3$};
\draw [spring] (dots3.0) -- ($(M32.north west)!(dots3.0)!(M32.south west)$) node[draw=none,fill=none,pos=0.5,above] {$k_3$};

\draw [damper] (9,1.3) -- ($(9.5,1.3)!(9,1.3)!(9.5,1.3)$) node[draw=none,fill=none,pos=0.5,above,label=$v_1$] {};
\draw [spring] (M12.0) -- ($(9,1.3)!(M12.0)!(9,1.3)$) node[draw=none,fill=none,pos=0.5,above] {$k_1$};

\draw [damper] (9,0) -- ($(9.5,0)!(9,0)!(9.5,0)$) node[draw=none,fill=none,pos=0.5,above,label=$v_{n+1}$] {};
\draw [spring] (M22.0) -- ($(9,0)!(M22.0)!(9,0)$) node[draw=none,fill=none,pos=0.5,above] {$k_2$};

\draw [damper] (9,-1.3) -- ($(9.5,-1.3)!(9,-1.3)!(9.5,-1.3)$) node[draw=none,fill=none,pos=0.5,above,label=$v_{2n+1}$] {};
\draw [spring] (M32.0) -- ($(9,-1.3)!(M32.-1.3)!(9,-1.3)$) node[draw=none,fill=none,pos=0.5,above] {$k_3$};

\draw[thick] (9.5,1.3) -- (9.75,1.3);
\draw[thick] (9.5,0) -- (9.75,0);
\draw[thick] (9.5,-1.3) -- (9.75,-1.3);

\end{scope}

\begin{scope}

\node (wallr) [ground, rotate=+90, minimum width=3cm,yshift=-16.9cm] {};
\draw (wallr.north west) -- (wallr.north east);

\end{scope}

\end{tikzpicture}	
\caption{Triple chain oscillator with $(3n + 1)$ masses and three dampers.}
\label{fig:triplechain}
\end{figure}
The triple chain consists of three rows that are coupled via a~mass $m_0$ which is connected to the fixed base with a spring with stiffness $k_0$. Each row contains $n$ masses, $n+1$ springs and one damper. The latter is attached to a~wall, see Figure~\ref{fig:triplechain}. One can write the free system as 
\begin{equation*}
M \ddot{x}(t) + D\dot{x}(t) +Kx(t) = 0,
\end{equation*}
where $M,D$, and $K$ are defined as $M=\diag(m_1,\ldots,m_1,m_2,\ldots,m_2,m_3,\ldots,m_3)$ and $D=\alpha M +\beta K + v(e_1e_1\T+e_{n+1}e_{n+1}\T+e_{2n+1}e_{2n+1}\T)$, where $e_i$ denotes the $i$'th unit vector in $\mathbb{R}^n$ and with the dampers' viscosity $v$. Moreover,
\begin{equation*}
K=\begin{bmatrix}
K_{11} & & & -\kappa_1 \\
& K_{22} & &  -\kappa_2 \\
& & K_{33} & -\kappa_3 \\
-\kappa_1\T & -\kappa_2\T & -\kappa_3\T & k_1+k_2+k_3+k_0
\end{bmatrix}, \ K_{ii}=k_i\begin{bmatrix}
2 & -1 \\
-1 & 2 & -1 \\
& \ddots & \ddots & \ddots \\
& & -1 & 2 & -1 \\
& & & -1 & 2
\end{bmatrix},
\end{equation*}
where $\kappa_i = \begin{bmatrix}
0 & \ldots & 0 & k_i
\end{bmatrix}\T\in\R^{1\times n}$ and $K_{ii}\in\R^{n\times n}$ for $i=1,2,3$.
We choose the input $b=\begin{bmatrix}
1 & \ldots & 1
\end{bmatrix}$ and equally measure the velocities $c_{\mathrm{v}}=\begin{bmatrix}
1 & \ldots & 1
\end{bmatrix}\T$. The second order control system then reads
\begin{equation*}
M \ddot{p}(t) + D\dot{p}(t) +Kp(t) = bu(t), \quad y(t) = c_{\mathrm{v}}\dot{p}(t).
\end{equation*}
We consider the triple chain with $n=500$, thus the number of positions is $3n+1=1501$, $k_0=50$, $k_1=10$, $ k_2=20$, $ k_3=1$, $m_0=1$, $m_1=1$, $m_2=2$, $m_3=3$, $\alpha=\beta=0.002$ and $v=5$. Following the previously presented numerical procedure we first compute a~first order reduced model of order $2r=300$. Applying Step~2) in this procedure, we obtain a~second order model of order $r=150$. The latter has the form
\begin{align*}
\begin{aligned}
\wtM\ddot{\wtp}(t) + \wtD\dot{\wtp}(t) + \wtK \wtp(t) & = \widetilde{B} u(t), &
\wty(t) & = \widetilde{B}^{\top}\dot{\wtp}(t),
\end{aligned}
\end{align*}
with symmetric $\wtM, \ \wtD,\, \wtK \in \mathbb{R}^{r \times r}$, where $\wtM=I_r$, $K>0$ and $D$ has one negative eigenvalue $\lambda\approx -3.535\cdot10^{-2}$ and its largest eigenvalue is $\lambda_{\max}\approx 3.162\cdot10^0$.
The sigma plots of the original and reduced transfer function $\tfG(s)$ and $\tilde{\tfG}(s)$ together are depicted in Figure~\ref{fig:systems}, whereas Figures~\ref{fig:errorabs} and~\ref{fig:errorrel} show the absolute and relative error of the transfer function evaluated on the imaginary axis. With a maximum relative error of approximately $4.3\cdot 10^{-2}$ we obtain a good match between the original and the reduced second order system.
\begin{figure}[t]
	\subfigure[Sigma plots of the original and the reduced transfer functions.\label{fig:systems}]{\input{systems.tikz}}
	\subfigure[Absolute error.\label{fig:errorabs}]{\input{errorabs.tikz}}
	\subfigure[Relative error.\label{fig:errorrel}]{\input{errorrel.tikz}}
	\caption{Sigma plots of th original and reduced transfer functions as well as the absolute and relative errors for Example~\ref{ex:msd}.}
	\label{fig:fig}
\end{figure}
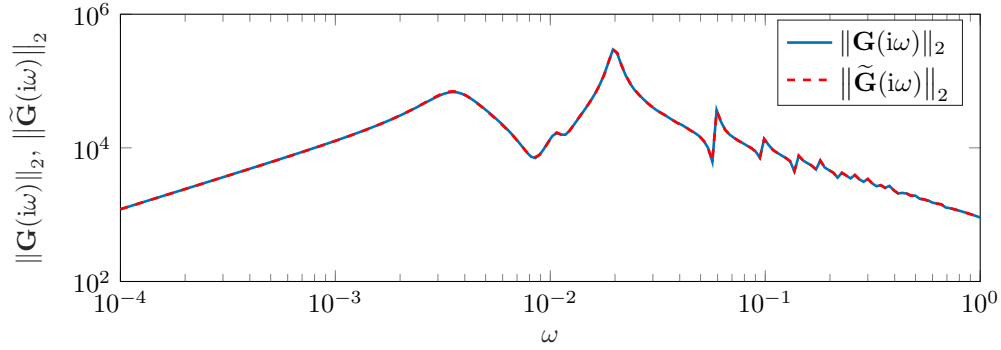
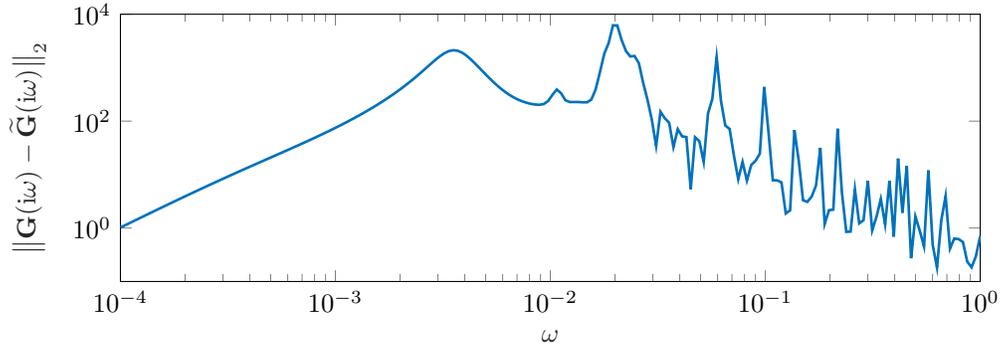
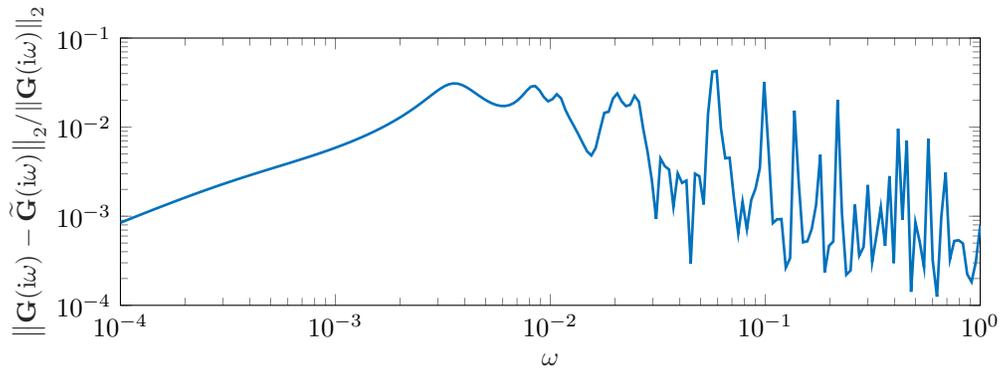
\end{example}
\end{section}


\begin{section}{Conclusion}
	
In this article we have presented a~numerical procedure for model reduction of passive second order systems that preserves asymptotic stability, passivity and restores the second order structure. Since the underlying reduction method is positive real balanced truncation we have received the gap metric error bound from~\cite{GuivOpme13}. One remaining question would be how to restore the definiteness of the damping matrix, as this might be lost. In detail, there exist systems that are asymptotically stable and passive, which can be written as  a~minimal second order system~\eqref{eqn:sosys}, where $M,K>0$ and $D=D\T$ is indefinite. Hence, to establish properties that guarantee the existence of a second order system with definite damping matrix might be an interesting topic for future research.
\end{section}

\section*{Acknowledgement}
Ines Dorschky thanks Serkan Gugercin for the hospitality and useful advises during the research stay at Virginia Tech.
\end{section}



\bigskip\bigskip
\bibliographystyle{siam}
\bibliography{BibDRV20}
\end{document}

%% file: systems.tikz
%
%
\definecolor{mycolor1}{rgb}{0.00000,0.44700,0.74100}%
\begin{tikzpicture}

\begin{axis}[%
width=4.5in,
height=1.4in,
scale only axis,
xmode=log,
xmin=0.0001,
xmax=1,
xminorticks=true,
xlabel style={font=\color{white!15!black}},
xlabel={$\omega$},
ymode=log,
ymin=100,
ymax=1000000,
yminorticks=true,
ylabel style={font=\color{white!15!black}},
ylabel={${\|\mathbf{G}(\mathrm{i}\omega)\|}_2$, $\big\|\widetilde{\mathbf{G}}(\mathrm{i}\omega)\big\|_2$},
axis background/.style={fill=white},
legend style={legend cell align=left, align=left, draw=white!15!black}
]
\addplot [color=mycolor1, line width = 1pt]
  table[row sep=crcr]{%
0.0001	1206.89787970076\\
0.000104737089795945	1264.13654315613\\
0.000109698579789238	1324.09661514034\\
0.000114895100018731	1386.90849701441\\
0.000120337784077759	1452.70898920036\\
0.000126038292967973	1521.64162748821\\
0.000132008840083142	1593.85704025137\\
0.000138262217376466	1669.51332831152\\
0.000144811822767453	1748.77646939065\\
0.000151671688847092	1831.8207493117\\
0.000158856512942805	1918.82922236382\\
0.000166381688607613	2009.99420353879\\
0.000174263338600965	2105.51779567671\\
0.000182518349431904	2205.61245493848\\
0.00019116440753857	2310.50159845681\\
0.000200220037181558	2420.42025851554\\
0.000209704640132323	2535.61578817912\\
0.000219638537241655	2656.34862395123\\
0.000230043011977292	2782.89311179901\\
0.000240940356023952	2915.53840375174\\
0.000252353917043477	3054.58943329062\\
0.000264308148697411	3200.36797891138\\
0.000276828663039207	3353.21382659026\\
0.000289942285388288	3513.48604344801\\
0.000303677111803546	3681.56437672328\\
0.000318062569279412	3857.8507942799\\
0.000333129478793467	4042.77118533501\\
0.000348910121340677	4236.77724296879\\
0.000365438307095725	4440.34855333755\\
0.000382749447851631	4653.99492044927\\
0.000400880632889846	4878.25895998255\\
0.000419870708444391	5113.7190010665\\
0.000439760360930272	5360.99234134824\\
0.00046059220411451	5620.73890824534\\
0.000482410870416537	5893.66538824758\\
0.000505263106533568	6180.52989677349\\
0.000529197873595844	6482.14727375199\\
0.000554266452066311	6799.39510520648\\
0.00058052255160949	7133.22058918836\\
0.000608022426164942	7484.64838607636\\
0.000636824994471859	7854.78961931448\\
0.000666991966303012	8244.85222407523\\
0.000698587974678525	8656.15287930897\\
0.000731680714342719	9090.13080466023\\
0.000766341086800746	9548.36375963967\\
0.000802643352225717	10032.5866505081\\
0.000840665288561833	10544.7132333629\\
0.000880488358164347	11086.8615033563\\
0.000922197882333432	11661.3834840081\\
0.000965883224115871	12270.9002822368\\
0.00101163797976621	12918.3434599632\\
0.00105956017927762	13607.0039987039\\
0.00110975249641207	14340.5904067816\\
0.00116232246867985	15123.2978466572\\
0.00121738272773966	15959.8905474858\\
0.00127505124071301	16855.8002148719\\
0.0013354515629299	17817.2436433416\\
0.00139871310264724	18851.3632397847\\
0.00146497139830728	19966.3945921386\\
0.00153436840893001	21171.8653888479\\
0.00160705281826164	22478.8295604974\\
0.00168318035333096	23900.1388040003\\
0.00176291411809595	25450.7493970552\\
0.00184642494289554	27148.0530585948\\
0.00193389175045523	29012.2021991668\\
0.00202550193923067	31066.3642368557\\
0.00212145178491063	33336.7722493188\\
0.00222194686093952	35852.3144577164\\
0.00232720247896041	38643.1794381324\\
0.00243744415012222	41737.6801367684\\
0.00255290806823952	45155.7337491435\\
0.00267384161583995	48896.5411543892\\
0.00280050389418363	52917.0484407027\\
0.00293316627839004	57098.0094881893\\
0.00307211299886176	61199.3319620985\\
0.00321764175025074	64821.8879766411\\
0.00337006432927193	67420.6758857948\\
0.00352970730273065	68428.5421165198\\
0.00369691270719502	67488.4204471481\\
0.00387203878181256	64652.4072800549\\
0.00405546073584083	60366.5923814627\\
0.0042475715525369	55250.0372217548\\
0.00444878283112759	49858.9725714433\\
0.00465952566866468	44574.2361010848\\
0.00488025158365443	39603.2342869553\\
0.00511143348344017	35029.6340767472\\
0.00535356667741072	30864.2494675373\\
0.00560716993820546	27081.4200364226\\
0.00587278661318948	23641.121783352\\
0.0061509857885805	20501.4261421035\\
0.00644236350872137	17625.7738289707\\
0.00674754405311069	14988.6731932718\\
0.00706718127392749	12583.406627783\\
0.00740195999691564	10437.0061525602\\
0.00775259748862946	8641.14803119937\\
0.00811984499318401	7402.19754696119\\
0.00850448934180268	7045.7925285968\\
0.00890735463861044	7809.57492422334\\
0.00932930402628469	9617.81305119372\\
0.0097712415353465	12215.0501737502\\
0.0102341140210545	15086.6949708672\\
0.0107189131920513	16658.2927790107\\
0.0112266777351081	15762.9100273782\\
0.0117584955405216	15706.8278251458\\
0.0123155060329283	18211.5355378947\\
0.0128989026125331	22311.2894571675\\
0.0135099352119803	27499.1737222134\\
0.0141499129743458	33878.5613927079\\
0.0148202070579886	41925.0148187597\\
0.0155222535742705	52541.863266063\\
0.0162575566644379	67410.9237303351\\
0.017027691722259	89870.7024666681\\
0.0178343087693191	126713.06236724\\
0.0186791359902078	193916.440794499\\
0.0195639834351706	295526.481175981\\
0.0204907468981585	258493.344754126\\
0.021461411978584	168987.587847081\\
0.0224780583354873	119220.94670317\\
0.0235428641432242	91070.0223155499\\
0.024658110758226	73528.3726655125\\
0.0258261876068267	62628.3451240988\\
0.0270495973046313	53795.313355247\\
0.0283309610183932	46886.5001475373\\
0.0296730240818887	41428.3043985159\\
0.0310786618778201	36944.6880651952\\
0.0325508859983506	33717.6406296786\\
0.0340928506974681	30952.6184790773\\
0.0357078596490046	28098.3185525677\\
0.037399373024788	25599.2240374391\\
0.0391710149080926	23411.2963934672\\
0.0410265810582719	21956.5816029193\\
0.0429700470432084	20006.6191067148\\
0.045005576757005	18167.2599159035\\
0.0471375313411672	16648.823843762\\
0.04937047852839	14927.8018787547\\
0.0517092024289676	12755.5792810176\\
0.0541587137807947	10025.0488608881\\
0.0567242606849198	6192.980043967\\
0.0594113398496503	36703.3497161674\\
0.0622257083673023	24644.1819452097\\
0.0651733960488243	18507.3047058534\\
0.0682607183427239	15773.6787970239\\
0.0714942898659758	14167.3034409058\\
0.0748810385759002	12809.804222049\\
0.0784282206133768	11839.9581114429\\
0.0821434358491943	10804.730524638\\
0.086034644166845	9935.35626668595\\
0.0901101825166502	8830.33600945979\\
0.0943787827777538	7136.13629037504\\
0.0988495904662559	13595.7705688192\\
0.103532184329566	11111.4068256747\\
0.108436596868961	9336.51427997056\\
0.113573335834311	8438.59096242294\\
0.118953406737032	7705.33867982389\\
0.124588336429501	7064.5192081543\\
0.13049019780144	6317.59689901259\\
0.136671635646201	4439.21553481343\\
0.143145893752348	7605.84669092633\\
0.149926843278605	6474.92120243586\\
0.157029012472938	5906.36473333355\\
0.164467617799466	5417.0883049004\\
0.172258596539879	4743.03407638076\\
0.180418640939207	6460.39337022488\\
0.188965233969121	5089.00715992742\\
0.197916686785356	4641.60322490468\\
0.207292177959537	4243.93056367736\\
0.21711179456945	3575.3739278614\\
0.227396575235793	4236.12672119097\\
0.238168555197616	3850.52200959032\\
0.249450813523032	3487.9104074328\\
0.261267522556333	3919.52358506237\\
0.273643999707467	3393.95876364135\\
0.286606761694825	3103.26552666706\\
0.300183581357559	3417.59558028103\\
0.31440354715915	2942.08641532018\\
0.329297125509715	2662.29354613651\\
0.344896226040576	2762.8098206521\\
0.361234269970943	2517.04286997597\\
0.378346261713193	2696.4168120938\\
0.396268863870148	2318.6206761117\\
0.415040475785048	2079.1450506895\\
0.434701315812503	2116.66811843737\\
0.455293507486695	2066.19417350961\\
0.476861169771447	1921.70239857328\\
0.499450511585514	1924.6853087103\\
0.523109930805626	1730.00377524603\\
0.547890117959394	1693.81191878469\\
0.573844164830239	1624.4655354559\\
0.601027678207038	1511.4274263863\\
0.629498899022189	1466.37303976965\\
0.659318827133355	1426.6842267086\\
0.690551352016233	1280.07720669143\\
0.723263389648353	1249.39546458523\\
0.757525025877191	1204.28238863608\\
0.793409666579749	1154.70034321829\\
0.83099419493534	1104.40823233983\\
0.870359136148517	1054.39832722904\\
0.911588829975082	1005.01436499826\\
0.954771611420806	956.099403773017\\
1	905.869448710401\\
};
\addlegendentry{$\|{\mathbf{G}}(\mathrm{i}\omega)\|_2$}

\addplot [color=red, dashed, line width = 1pt]
  table[row sep=crcr]{%
0.0001	1206.5239718091\\
0.000104737089795945	1263.70896153428\\
0.000109698579789238	1323.60786729254\\
0.000114895100018731	1386.35009659562\\
0.000120337784077759	1452.0713367289\\
0.000126038292967973	1520.91387971562\\
0.000132008840083142	1593.02696781945\\
0.000138262217376466	1668.56716150306\\
0.000144811822767453	1747.69873201132\\
0.000151671688847092	1830.59408103249\\
0.000158856512942805	1917.43419021284\\
0.000166381688607613	2008.409103663\\
0.000174263338600965	2103.71844700053\\
0.000182518349431904	2203.57198692647\\
0.00019116440753857	2308.19023583723\\
0.000200220037181558	2417.80510652952\\
0.000209704640132323	2532.66062267049\\
0.000219638537241655	2653.01369138196\\
0.000230043011977292	2779.13494503333\\
0.000240940356023952	2911.30966016358\\
0.000252353917043477	3049.83876237138\\
0.000264308148697411	3195.03992704472\\
0.000276828663039207	3347.24878697478\\
0.000289942285388288	3506.82025925078\\
0.000303677111803546	3674.13000541042\\
0.000318062569279412	3849.57604068783\\
0.000333129478793467	4033.58051043129\\
0.000348910121340677	4226.591654449\\
0.000365438307095725	4429.08598328854\\
0.000382749447851631	4641.57069438761\\
0.000400880632889846	4864.58636078851\\
0.000419870708444391	5098.70993084422\\
0.000439760360930272	5344.55808423566\\
0.00046059220411451	5602.79099787055\\
0.000482410870416537	5874.11658507838\\
0.000505263106533568	6159.29528323289\\
0.000529197873595844	6459.14547886305\\
0.000554266452066311	6774.5496758759\\
0.00058052255160949	7106.46153223832\\
0.000608022426164942	7455.9139140162\\
0.000636824994471859	7824.02814389233\\
0.000666991966303012	8212.02465524439\\
0.000698587974678525	8621.23530390872\\
0.000731680714342719	9053.11763957286\\
0.000766341086800746	9509.27149945604\\
0.000802643352225717	9991.45836120162\\
0.000840665288561833	10501.6239830371\\
0.000880488358164347	11041.9249713714\\
0.000922197882333432	11614.7600542016\\
0.000965883224115871	12222.8070092648\\
0.00101163797976621	12869.0664064076\\
0.00105956017927762	13556.9135832422\\
0.00110975249641207	14290.1605922783\\
0.00116232246867985	15073.1302477737\\
0.00121738272773966	15910.7448725955\\
0.00127505124071301	16808.632907285\\
0.0013354515629299	17773.2571939031\\
0.00139871310264724	18812.0694637205\\
0.00146497139830728	19933.6962739712\\
0.00153436840893001	21148.1622003757\\
0.00160705281826164	22467.156173396\\
0.00168318035333096	23904.3457927847\\
0.00176291411809595	25475.7409860439\\
0.00184642494289554	27200.1000235321\\
0.00193389175045523	29099.3529575333\\
0.00202550193923067	31198.9811108242\\
0.00212145178491063	33528.2196527006\\
0.00222194686093952	36119.8130184065\\
0.00232720247896041	39008.7958830407\\
0.00243744415012222	42229.3066835236\\
0.00255290806823952	45807.6432484639\\
0.00267384161583995	49748.5510351781\\
0.00280050389418363	54010.3367446198\\
0.00293316627839004	58464.3045001264\\
0.00307211299886176	62839.7877310484\\
0.00321764175025074	66676.3281940183\\
0.00337006432927193	69342.2589713486\\
0.00352970730273065	70197.0202196136\\
0.00369691270719502	68883.4049099921\\
0.00387203878181256	65546.9550751979\\
0.00405546073584083	60764.5566541006\\
0.0042475715525369	55248.0987180716\\
0.00444878283112759	49583.4927725572\\
0.00465952566866468	44138.4005618319\\
0.00488025158365443	39091.9165636679\\
0.00511143348344017	34500.5287629342\\
0.00535356667741072	30354.4982657362\\
0.00560716993820546	26614.3970050245\\
0.00587278661318948	23231.7341972398\\
0.0061509857885805	20159.8636061646\\
0.00644236350872137	17360.0774409758\\
0.00674754405311069	14806.4651858719\\
0.00706718127392749	12492.820298108\\
0.00740195999691564	10446.0009376167\\
0.00775259748862946	8751.90603420714\\
0.00811984499318401	7592.76936393838\\
0.00850448934180268	7241.35909477091\\
0.00890735463861044	7899.62567218086\\
0.00932930402628469	9532.53718309729\\
0.0097712415353465	11979.1705854152\\
0.0102341140210545	14937.2451940132\\
0.0107189131920513	16987.1866759176\\
0.0112266777351081	15854.7013322712\\
0.0117584955405216	15462.5743134597\\
0.0123155060329283	18195.002351271\\
0.0128989026125331	22479.1306182853\\
0.0135099352119803	27725.7259936824\\
0.0141499129743458	34081.0163829895\\
0.0148202070579886	42039.6995670608\\
0.0155222535742705	52493.0690457678\\
0.0162575566644379	67074.2287225255\\
0.017027691722259	89057.4201964539\\
0.0178343087693191	125767.853893081\\
0.0186791359902078	194186.358118339\\
0.0195639834351706	297970.960383513\\
0.0204907468981585	264158.803700479\\
0.021461411978584	172281.556976075\\
0.0224780583354873	121181.883902887\\
0.0235428641432242	92493.9282161415\\
0.024658110758226	74557.7645615716\\
0.0258261876068267	62335.3445755195\\
0.0270495973046313	53457.6730262827\\
0.0283309610183932	46690.097634485\\
0.0296730240818887	41332.3113018142\\
0.0310786618778201	36978.3614079152\\
0.0325508859983506	33675.0444999422\\
0.0340928506974681	31013.2428516013\\
0.0357078596490046	28064.5262726939\\
0.037399373024788	25570.2347231109\\
0.0391710149080926	23459.16907245\\
0.0410265810582719	22001.7368660174\\
0.0429700470432084	19985.9139971946\\
0.045005576757005	18163.9942507688\\
0.0471375313411672	16609.9463420304\\
0.04937047852839	14913.5711320409\\
0.0517092024289676	12756.4157089407\\
0.0541587137807947	9969.27794982435\\
0.0567242606849198	6428.98785742867\\
0.0594113398496503	35547.6123316057\\
0.0622257083673023	24787.9044051112\\
0.0651733960488243	18439.2530316479\\
0.0682607183427239	15845.0353091426\\
0.0714942898659758	14177.1484235018\\
0.0748810385759002	12814.8357444048\\
0.0784282206133768	11848.9694539879\\
0.0821434358491943	10811.6534396315\\
0.086034644166845	9938.61320028277\\
0.0901101825166502	8840.93394093643\\
0.0943787827777538	7122.5835683321\\
0.0988495904662559	13550.1055209958\\
0.103532184329566	11173.6560892963\\
0.108436596868961	9333.50868159657\\
0.113573335834311	8441.72172888416\\
0.118953406737032	7708.54157661729\\
0.124588336429501	7064.83631924561\\
0.13049019780144	6316.5854905198\\
0.136671635646201	4500.15229904537\\
0.143145893752348	7606.12839071377\\
0.149926843278605	6476.84702587455\\
0.157029012472938	5908.73442073516\\
0.164467617799466	5419.79292822968\\
0.172258596539879	4745.14855255057\\
0.180418640939207	6449.49292705257\\
0.188965233969121	5089.7372950556\\
0.197916686785356	4642.84205739115\\
0.207292177959537	4245.15667516574\\
0.21711179456945	3584.8699014589\\
0.227396575235793	4237.2532139519\\
0.238168555197616	3850.56081315938\\
0.249450813523032	3488.11414547217\\
0.261267522556333	3924.60738964754\\
0.273643999707467	3394.58298313857\\
0.286606761694825	3104.2076124859\\
0.300183581357559	3422.10218199048\\
0.31440354715915	2942.42190972306\\
0.329297125509715	2662.94164991402\\
0.344896226040576	2766.25878088833\\
0.361234269970943	2516.5878373145\\
0.378346261713193	2703.59669774558\\
0.396268863870148	2318.8407629299\\
0.415040475785048	2083.98618255398\\
0.434701315812503	2118.50118868532\\
0.455293507486695	2076.19924374777\\
0.476861169771447	1921.93301458963\\
0.499450511585514	1926.24903202112\\
0.523109930805626	1729.45803549434\\
0.547890117959394	1693.89021290672\\
0.573844164830239	1612.67819926795\\
0.601027678207038	1511.75967983429\\
0.629498899022189	1466.20248844751\\
0.659318827133355	1427.10183466653\\
0.690551352016233	1280.7610414282\\
0.723263389648353	1249.77626495897\\
0.757525025877191	1204.78965496437\\
0.793409666579749	1155.25379150867\\
0.83099419493534	1104.81233610238\\
0.870359136148517	1054.59554972176\\
0.911588829975082	1005.10488191484\\
0.954771611420806	956.243751658956\\
1	906.36520473696\\
};
\addlegendentry{$\big\|\widetilde{\mathbf{G}}(\mathrm{i}\omega)\big\|_2$}

\end{axis}
\end{tikzpicture}%

%% file: errorabs.tikz
%
%
\definecolor{mycolor1}{rgb}{0.00000,0.44700,0.74100}%
\begin{tikzpicture}

\begin{axis}[%
width=4.5in,
height=1.4in,
scale only axis,
xmode=log,
xmin=0.0001,
xmax=1,
xminorticks=true,
xlabel style={font=\color{white!15!black}},
xlabel={$\omega$},
ymode=log,
ymin=0.1,
ymax=10000,
yminorticks=true,
ylabel style={font=\color{white!15!black}},
ylabel={$\big\|\mathbf{G}(\mathrm{i}\omega)- \widetilde{\mathbf{G}}(\mathrm{i}\omega)\big\|_2$},
axis background/.style={fill=white},
legend style={legend cell align=left, align=left, draw=white!15!black}
]
\addplot [color=mycolor1, line width = 1pt]
  table[row sep=crcr]{%
0.0001	1.02101146048168\\
0.000104737089795945	1.11839687154114\\
0.000109698579789238	1.22491009373388\\
0.000114895100018731	1.34137673391497\\
0.000120337784077759	1.46869177636815\\
0.000126038292967973	1.6078245502867\\
0.000132008840083142	1.75982393451838\\
0.000138262217376466	1.92582380192347\\
0.000144811822767453	2.1070487082628\\
0.000151671688847092	2.30481983507549\\
0.000158856512942805	2.52056120301803\\
0.000166381688607613	2.75580618236184\\
0.000174263338600965	3.01220434157995\\
0.000182518349431904	3.29152869414481\\
0.00019116440753857	3.59568342841369\\
0.000200220037181558	3.92671223730317\\
0.000209704640132323	4.28680740339141\\
0.000219638537241655	4.67831984269097\\
0.000230043011977292	5.10377036640513\\
0.000240940356023952	5.56586248528444\\
0.000252353917043477	6.06749715505791\\
0.000264308148697411	6.61178994385786\\
0.000276828663039207	7.20209119196478\\
0.000289942285388288	7.84200982965227\\
0.000303677111803546	8.53544161866082\\
0.000318062569279412	9.28660268539438\\
0.000333129478793467	10.1000693179013\\
0.000348910121340677	10.9808251039754\\
0.000365438307095725	11.9343165946861\\
0.000382749447851631	12.9665187895972\\
0.000400880632889846	14.0840118612723\\
0.000419870708444391	15.2940706755598\\
0.000439760360930272	16.6047688323304\\
0.00046059220411451	18.0250991620583\\
0.000482410870416537	19.565112887962\\
0.000505263106533568	21.2360800213925\\
0.000529197873595844	23.0506740293374\\
0.000554266452066311	25.0231844285779\\
0.00058052255160949	27.1697617596776\\
0.000608022426164942	29.5087004217\\
0.000636824994471859	32.0607661630883\\
0.000666991966303012	34.8495766930272\\
0.000698587974678525	37.9020459933181\\
0.000731680714342719	41.2489055859232\\
0.000766341086800746	44.925319397483\\
0.000802643352225717	48.9716131577176\\
0.000840665288561833	53.4341447356174\\
0.000880488358164347	58.3663488093555\\
0.000922197882333432	63.8299982482775\\
0.000965883224115871	69.8967361932556\\
0.00101163797976621	76.6499478963945\\
0.00105956017927762	84.1870610675293\\
0.00110975249641207	92.6223893048346\\
0.00116232246867985	102.090667228244\\
0.00121738272773966	112.751470957515\\
0.00127505124071301	124.794777266332\\
0.0013354515629299	138.447993943535\\
0.00139871310264724	153.984898799695\\
0.00146497139830728	171.737062961775\\
0.00153436840893001	192.108514088268\\
0.00160705281826164	215.594624481626\\
0.00168318035333096	242.80648956529\\
0.00176291411809595	274.502378102825\\
0.00184642494289554	311.628126923664\\
0.00193389175045523	355.368456704201\\
0.00202550193923067	407.210708371615\\
0.00212145178491063	469.020551184657\\
0.00222194686093952	543.123858379789\\
0.00232720247896041	632.376144727646\\
0.00243744415012222	740.172719121058\\
0.00255290806823952	870.294629490432\\
0.00267384161583995	1026.37658630908\\
0.00280050389418363	1210.61055919169\\
0.00293316627839004	1421.12626245535\\
0.00307211299886176	1647.64768210012\\
0.00321764175025074	1866.36683293259\\
0.00337006432927193	2038.36547807572\\
0.00352970730273065	2119.03256363223\\
0.00369691270719502	2080.14494525917\\
0.00387203878181256	1929.09967409918\\
0.00405546073584083	1704.97129474954\\
0.0042475715525369	1455.20758421872\\
0.00444878283112759	1215.590163266\\
0.00465952566866468	1004.79738638234\\
0.00488025158365443	828.506714177448\\
0.00511143348344017	685.298579459653\\
0.00535356667741072	570.892302367115\\
0.00560716993820546	480.41888318926\\
0.00587278661318948	409.395725984584\\
0.0061509857885805	354.023174062661\\
0.00644236350872137	311.182354199978\\
0.00674754405311069	278.330586907209\\
0.00706718127392749	253.386616052934\\
0.00740195999691564	234.645560739328\\
0.00775259748862946	220.743983015758\\
0.00811984499318401	210.705869017833\\
0.00850448934180268	204.166934044019\\
0.00890735463861044	202.110902958402\\
0.00932930402628469	209.086659926452\\
0.0097712415353465	238.320993158078\\
0.0102341140210545	311.025769298515\\
0.0107189131920513	390.526320041179\\
0.0112266777351081	331.716750462105\\
0.0117584955405216	245.080140032174\\
0.0123155060329283	227.324714030202\\
0.0128989026125331	227.896000527759\\
0.0135099352119803	227.694700523544\\
0.0141499129743458	225.00237785968\\
0.0148202070579886	224.928702146476\\
0.0155222535742705	252.828846104085\\
0.0162575566644379	393.312137008356\\
0.017027691722259	834.012187812838\\
0.0178343087693191	1835.10516239159\\
0.0186791359902078	2883.34681233578\\
0.0195639834351706	6182.4238350924\\
0.0204907468981585	6173.04731729354\\
0.021461411978584	3298.76995061098\\
0.0224780583354873	2053.00883346569\\
0.0235428641432242	1621.82995879734\\
0.024658110758226	1654.37499754275\\
0.0258261876068267	1208.83561281027\\
0.0270495973046313	516.885721039166\\
0.0283309610183932	249.435828309472\\
0.0296730240818887	105.467972480287\\
0.0310786618778201	34.4269700364062\\
0.0325508859983506	149.803416799237\\
0.0340928506974681	113.045382286138\\
0.0357078596490046	93.7452655529113\\
0.037399373024788	32.6246726560988\\
0.0391710149080926	70.7053473733673\\
0.0410265810582719	51.9477420185144\\
0.0429700470432084	50.3153378418311\\
0.045005576757005	5.33912704908646\\
0.0471375313411672	50.0029116359819\\
0.04937047852839	42.0067590761056\\
0.0517092024289676	17.280130992059\\
0.0541587137807947	139.553106312187\\
0.0567242606849198	259.293802293052\\
0.0594113398496503	1568.53907434774\\
0.0622257083673023	238.763202369674\\
0.0651733960488243	82.6484159358204\\
0.0682607183427239	71.521871053227\\
0.0714942898659758	22.6168717127788\\
0.0748810385759002	8.19931464299111\\
0.0784282206133768	17.235145789978\\
0.0821434358491943	7.71952647580891\\
0.086034644166845	15.0892744065632\\
0.0901101825166502	18.076323568154\\
0.0943787827777538	24.7081268549566\\
0.0988495904662559	436.184897700112\\
0.103532184329566	62.2839397515985\\
0.108436596868961	7.81554067474492\\
0.113573335834311	7.80486812970209\\
0.118953406737032	7.15906464627012\\
0.124588336429501	1.87775401649177\\
0.13049019780144	2.14181951460328\\
0.136671635646201	68.1409429112677\\
0.143145893752348	18.1431480386682\\
0.149926843278605	3.30320646154375\\
0.157029012472938	3.0912528555238\\
0.164467617799466	3.88662192669682\\
0.172258596539879	6.27506923484162\\
0.180418640939207	31.6822319904211\\
0.188965233969121	1.19301059513438\\
0.197916686785356	2.15523852323635\\
0.207292177959537	2.20914512422199\\
0.21711179456945	72.2265444780631\\
0.227396575235793	4.56783564707026\\
0.238168555197616	0.848255557016825\\
0.249450813523032	0.858817139273358\\
0.261267522556333	5.29014590352296\\
0.273643999707467	1.22250078607154\\
0.286606761694825	1.40141071165608\\
0.300183581357559	7.66980421974604\\
0.31440354715915	0.903933518668502\\
0.329297125509715	1.6252498985977\\
0.344896226040576	3.48667640932117\\
0.361234269970943	1.16973022838936\\
0.378346261713193	7.6071647020214\\
0.396268863870148	0.690295821239438\\
0.415040475785048	19.9485531041687\\
0.434701315812503	1.92645419889013\\
0.455293507486695	14.5765263873917\\
0.476861169771447	0.272708169374685\\
0.499450511585514	1.64693478783878\\
0.523109930805626	0.894026729414142\\
0.547890117959394	0.454153591660581\\
0.573844164830239	11.9874771045853\\
0.601027678207038	0.486151890153652\\
0.629498899022189	0.183422917475906\\
0.659318827133355	1.39737727976444\\
0.690551352016233	3.95841372449192\\
0.723263389648353	0.411828418603738\\
0.757525025877191	0.635724623021809\\
0.793409666579749	0.62128135486955\\
0.83099419493534	0.546967461403733\\
0.870359136148517	0.235332778554203\\
0.911588829975082	0.184162198432043\\
0.954771611420806	0.295741779737692\\
1	0.70952105639393\\
};
\end{axis}
\end{tikzpicture}%

%% file: errorrel.tikz
%
%
\definecolor{mycolor1}{rgb}{0.00000,0.44700,0.74100}%
\begin{tikzpicture}

\begin{axis}[%
width=4.5in,
height=1.4in,
scale only axis,
xmode=log,
xmin=0.0001,
xmax=1,
xminorticks=true,
xlabel style={font=\color{white!15!black}},
xlabel={$\omega$},
ymode=log,
ymin=0.0001,
ymax=0.1,
yminorticks=true,
ylabel style={font=\color{white!15!black}},
ylabel={$\big\|\mathbf{G}(\mathrm{i}\omega)- \widetilde{\mathbf{G}}(\mathrm{i}\omega)\big\|_2/{\|\mathbf{G}(\mathrm{i}\omega)\|}_2$},
axis background/.style={fill=white},
legend style={legend cell align=left, align=left, draw=white!15!black}
]
\addplot [color=mycolor1, line width = 1pt]
  table[row sep=crcr]{%
0.0001	0.000845979993547445\\
0.000104737089795945	0.000884712080823854\\
0.000109698579789238	0.000925091175166285\\
0.000114895100018731	0.000967170319312734\\
0.000120337784077759	0.00101100205704419\\
0.000126038292967973	0.00105663812111972\\
0.000132008840083142	0.0011041290969489\\
0.000138262217376466	0.00115352406552584\\
0.000144811822767453	0.00120487023078312\\
0.000151671688847092	0.00125821253850385\\
0.000158856512942805	0.00131359329618346\\
0.000166381688607613	0.00137105180577635\\
0.000174263338600965	0.00143062402405953\\
0.000182518349431904	0.00149234226836854\\
0.00019116440753857	0.00155623498846106\\
0.000200220037181558	0.00162232662839777\\
0.000209704640132323	0.00169063760502527\\
0.000219638537241655	0.00176118443208412\\
0.000230043011977292	0.0018339800205642\\
0.000240940356023952	0.00190903418666077\\
0.000252353917043477	0.00198635439805132\\
0.000264308148697411	0.00206594678719004\\
0.000276828663039207	0.0021478174564514\\
0.000289942285388288	0.00223197409429764\\
0.000303677111803546	0.00231842791412972\\
0.000318062569279412	0.00240719591829829\\
0.000333129478793467	0.00249830347919242\\
0.000348910121340677	0.002591787218032\\
0.000365438307095725	0.00268769815056879\\
0.000382749447851631	0.0027861050583926\\
0.000400880632889846	0.0028870980357555\\
0.000419870708444391	0.00299079215584002\\
0.000439760360930272	0.00309733119823007\\
0.00046059220411451	0.00320689138141898\\
0.000482410870416537	0.00331968505150909\\
0.000505263106533568	0.00343596429045326\\
0.000529197873595844	0.00355602442460323\\
0.000554266452066311	0.00368020743630812\\
0.00058052255160949	0.00380890530721257\\
0.000608022426164942	0.00394256335095111\\
0.000636824994471859	0.00408168362450507\\
0.000666991966303012	0.00422682854051226\\
0.000698587974678525	0.00437862483735891\\
0.000731680714342719	0.00453776809952791\\
0.000766341086800746	0.00470502805804063\\
0.000802643352225717	0.00488125494089178\\
0.000840665288561833	0.00506738718759603\\
0.000880488358164347	0.00526446089289437\\
0.000922197882333432	0.00547362140485399\\
0.000965883224115871	0.00569613757634695\\
0.00101163797976621	0.00593341926029057\\
0.00105956017927762	0.00618703875412606\\
0.00110975249641207	0.00645875704399408\\
0.00116232246867985	0.00675055588162005\\
0.00121738272773966	0.00706467695514847\\
0.00127505124071301	0.00740366969681008\\
0.0013354515629299	0.00777044961133895\\
0.00139871310264724	0.00816836940867588\\
0.00146497139830728	0.00860130566734335\\
0.00153436840893001	0.00907376419412998\\
0.00160705281826164	0.00959100757009589\\
0.00168318035333096	0.0101592083442064\\
0.00176291411809595	0.01078563046692\\
0.00184642494289554	0.0114788388784663\\
0.00193389175045523	0.0122489307865918\\
0.00202550193923067	0.0131077684297707\\
0.00212145178491063	0.0140691650552414\\
0.00222194686093952	0.0151489204140597\\
0.00232720247896041	0.0163644957253085\\
0.00243744415012222	0.0177339209245846\\
0.00255290806823952	0.0192731809945828\\
0.00267384161583995	0.020990780985271\\
0.00280050389418363	0.0228775148059943\\
0.00293316627839004	0.0248892435164366\\
0.00307211299886176	0.0269226416249205\\
0.00321764175025074	0.0287922319326019\\
0.00337006432927193	0.0302335366902661\\
0.00352970730273065	0.0309670862200156\\
0.00369691270719502	0.0308222496759749\\
0.00387203878181256	0.0298380177205607\\
0.00405546073584083	0.0282436232937525\\
0.0042475715525369	0.0263385810651676\\
0.00444878283112759	0.0243805698467648\\
0.00465952566866468	0.0225421111896046\\
0.00488025158365443	0.0209201780888473\\
0.00511143348344017	0.019563395322892\\
0.00535356667741072	0.0184968794711037\\
0.00560716993820546	0.0177397966038387\\
0.00587278661318948	0.0173171023666432\\
0.0061509857885805	0.0172682218109505\\
0.00644236350872137	0.0176549612640837\\
0.00674754405311069	0.0185693945900527\\
0.00706718127392749	0.0201365674294813\\
0.00740195999691564	0.0224820755405772\\
0.00775259748862946	0.0255456777523946\\
0.00811984499318401	0.0284653128589271\\
0.00850448934180268	0.0289771424882816\\
0.00890735463861044	0.025879885258736\\
0.00932930402628469	0.0217395221568069\\
0.0097712415353465	0.0195104391523682\\
0.0102341140210545	0.0206158983063629\\
0.0107189131920513	0.0234433579252034\\
0.0112266777351081	0.0210441314380375\\
0.0117584955405216	0.015603414181431\\
0.0123155060329283	0.0124824572621668\\
0.0128989026125331	0.0102143805253957\\
0.0135099352119803	0.00828005607818012\\
0.0141499129743458	0.00664143837902487\\
0.0148202070579886	0.00536502379590882\\
0.0155222535742705	0.00481195051693927\\
0.0162575566644379	0.0058345460237532\\
0.017027691722259	0.00928013429206434\\
0.0178343087693191	0.0144823677062835\\
0.0186791359902078	0.0148690167812608\\
0.0195639834351706	0.0209200333265934\\
0.0204907468981585	0.0238808752432881\\
0.021461411978584	0.0195207825180396\\
0.0224780583354873	0.0172202024076958\\
0.0235428641432242	0.0178086039462891\\
0.024658110758226	0.0224998179283616\\
0.0258261876068267	0.0193017332713318\\
0.0270495973046313	0.00960837829172625\\
0.0283309610183932	0.00531999248236858\\
0.0296730240818887	0.00254579505513301\\
0.0310786618778201	0.00093185169071272\\
0.0325508859983506	0.0044428795728779\\
0.0340928506974681	0.00365220740088766\\
0.0357078596490046	0.00333633008599885\\
0.037399373024788	0.00127443990522466\\
0.0391710149080926	0.00302013806433621\\
0.0410265810582719	0.00236593031456261\\
0.0429700470432084	0.00251493456107953\\
0.045005576757005	0.000293887304623887\\
0.0471375313411672	0.00300339003554999\\
0.04937047852839	0.00281399494830445\\
0.0517092024289676	0.0013547115823877\\
0.0541587137807947	0.0139204415109279\\
0.0567242606849198	0.0418689872165255\\
0.0594113398496503	0.042735583713135\\
0.0622257083673023	0.0096884206950146\\
0.0651733960488243	0.00446571865808644\\
0.0682607183427239	0.0045342543089391\\
0.0714942898659758	0.0015964133052643\\
0.0748810385759002	0.000640081183198567\\
0.0784282206133768	0.00145567624714152\\
0.0821434358491943	0.000714458029120311\\
0.086034644166845	0.00151874517647231\\
0.0901101825166502	0.002047070864437\\
0.0943787827777538	0.00346239559469766\\
0.0988495904662559	0.0320823961755038\\
0.103532184329566	0.005605405393643\\
0.108436596868961	0.000837094063199951\\
0.113573335834311	0.000924901818853074\\
0.118953406737032	0.000929104474669729\\
0.124588336429501	0.00026580068100379\\
0.13049019780144	0.000339024402607586\\
0.136671635646201	0.0153497712325274\\
0.143145893752348	0.00238542121290885\\
0.149926843278605	0.000510153924390784\\
0.157029012472938	0.000523376559879175\\
0.164467617799466	0.000717474352998991\\
0.172258596539879	0.00132300741124548\\
0.180418640939207	0.00490407165242296\\
0.188965233969121	0.000234428948052688\\
0.197916686785356	0.000464330624313674\\
0.207292177959537	0.000520542240518604\\
0.21711179456945	0.0202011162847141\\
0.227396575235793	0.00107830476935922\\
0.238168555197616	0.00022029624941868\\
0.249450813523032	0.000246226834681076\\
0.261267522556333	0.00134969105012766\\
0.273643999707467	0.000360199068759436\\
0.286606761694825	0.000451592266151072\\
0.300183581357559	0.00224421059764928\\
0.31440354715915	0.000307242341340314\\
0.329297125509715	0.000610469833785325\\
0.344896226040576	0.0012620037699512\\
0.361234269970943	0.000464723999079335\\
0.378346261713193	0.00282121245791905\\
0.396268863870148	0.000297718306556749\\
0.415040475785048	0.00959459422879285\\
0.434701315812503	0.000910135217755501\\
0.455293507486695	0.00705477083145203\\
0.476861169771447	0.000141909678406579\\
0.499450511585514	0.0008556904239802\\
0.523109930805626	0.000516777328585308\\
0.547890117959394	0.000268125159956625\\
0.573844164830239	0.00737933606034987\\
0.601027678207038	0.000321650832627803\\
0.629498899022189	0.00012508612235855\\
0.659318827133355	0.000979458000309029\\
0.690551352016233	0.00309232420029029\\
0.723263389648353	0.00032962214949328\\
0.757525025877191	0.000527886672611568\\
0.793409666579749	0.00053804552715206\\
0.83099419493534	0.000495258406617371\\
0.870359136148517	0.000223191532532736\\
0.911588829975082	0.000183243349394674\\
0.954771611420806	0.000309321163229072\\
1	0.000783248687108288\\
};

\end{axis}
\end{tikzpicture}%